\documentclass[english, 11pt]{amsart}

\usepackage{listings}  
\usepackage{xcolor}    

\usepackage{tikz}
\usepackage{aliascnt}
\usepackage{mathrsfs}
\usepackage[all,poly,knot]{xy}
\usepackage{hyperref}
\usepackage{comment}  
\usepackage{csquotes}   
\usepackage{amssymb,amsbsy,amsmath,amsfonts,amssymb,amscd,
	graphics,color,footmisc,fancyhdr,multicol,fancybox,
	graphicx,mathrsfs,rotating,ifthen,wasysym}

\usepackage{geometry} 
\usepackage{times}
\usepackage{pdfpages}
\usepackage{multirow}
\usepackage{nccrules}
\usepackage{textcomp}
\usepackage{comment}
\usepackage{framed}
\usepackage[all]{xy}
\usepackage{graphicx}
\usepackage{pgf,tikz}
\usepackage{mathrsfs}
\usetikzlibrary{arrows}
\usepackage{lipsum}
\usepackage{mathtools}

\renewcommand{\P}{\mathbb{P}}

\newcommand{\style}[1]{{\sf #1}}
\newcommand{\function}{\style{function}}

\newcommand{\Twist}{\style{Twist}}

\newcommand{\Eq}{\style{Eq}}
\newcommand{\HEq}{\style{HEq}}

\def\log{\mathrm{log}\,}

\usepackage{mathtools}
\newtagform{EngelLie}[\scriptstyle]{$}{$}
\makeatletter\newcommand{\leqnomode}{\tagsleft@true}
\newcommand{\reqnomode}{\tagsleft@false}\makeatother

\newtheorem{Observation}[equation]{Observation}

\theoremstyle{plain}

\newtheorem{thm}{Theorem}[section]  

\newtheorem{cor}[thm]{{Corollary}} 

\newtheorem{lem}[thm]{{Lemma}}

\newtheorem{pro}[thm]{Proposition}

\newtheorem{ques}[thm]{{Question}}
\newtheorem{rem}[thm]{Remark}

\theoremstyle{remark}

\theoremstyle{definition}
\newtheorem{Question}[equation]{Question}
\newtheorem{Convention}[equation]{Convention}

\numberwithin{equation}{section}

\usepackage[hyperpageref]{backref}

\usepackage{xcolor}
\hypersetup{
	colorlinks,
	linkcolor={red!50!black},
	citecolor={blue!62!black},
	urlcolor={blue!80!black}
}
\theoremstyle{plain}
\newcommand{\thistheoremname}{}
\newtheorem*{genericthm*}{\thistheoremname}
\newenvironment{namedthm*}[1]{\renewcommand{\thistheoremname}{#1}%
	\begin{genericthm*}}
	{\end{genericthm*}}

\newtheoremstyle{named}{}{}{\itshape}{}{\bfseries}{.}{.5em}{\thmnote{#3's }#1}
\theoremstyle{named}

\makeatletter
\newcommand\thankssymb[1]{\textsuperscript{\@fnsymbol{#1}}}
\makeatother
\begin{document} 
	\title[Improved Hyperbolicity Degree Bounds in Dimension Two]{\textbf{Vanishing of Invariant 2-Jet Differentials and } \\ \textbf{Improved Hyperbolicity Degree Bounds in Dimension Two}}

	\keywords{Kobayashi hyperbolicity, entire curves, jet differentials,   slanted vector fields, Demailly–Semple tower, Second Main Theorem}
	
	\author{Lei Hou}
	\address{Academy of Mathematics and Systems Science, Chinese Academy of Sciences, Beijing 100190, China}
	\email{houlei@amss.ac.cn}
	
	\author{Dinh Tuan Huynh}
	
	\address{Department of Mathematics, University of Education, Hue University, 34 Le Loi St., Hue City, Vietnam}
	\email{dinhtuanhuynh@hueuni.edu.vn}
	
	\author{Joël  Merker}
	\address{Laboratoire de Math\'{e}matiques d’Orsay, Universit\'{e} Paris-Saclay, 91405 Orsay
Cedex, France}
\email{joel.merker@universite-paris-saclay.fr}
	
	\author{Song-Yan Xie}
	\address{State Key Laboratory of Mathematical Sciences, Academy of Mathematics and Systems Science, Chinese Academy of Sciences, Beijing 100190, China;  School of Mathematical Sciences, University of Chinese Academy of Sciences, Beijing 100049, China}
	\email{xiesongyan@amss.ac.cn}
	
\begin{abstract}
 This paper establishes new degree bounds for Kobayashi hyperbolicity in dimension two. Our main results are:

\begin{itemize}
    \item A {\em very generic} surface in $\mathbb{P}^3$ of degree at least $17$ is Kobayashi hyperbolic.
    \item The complement of a {\em very generic} curve in $\mathbb{P}^2$ of degree at least $12$ is Kobayashi hyperbolic.
\end{itemize}

These bounds improve the long-standing records in the field, lowering the threshold from $18$ to $17$ for surfaces (P\u{a}un) and from $14$ to $12$ for complements (Rousseau).

Central to the proofs are new vanishing results for certain negatively twisted invariant $2$-jet differentials, obtained through a novel combination of algebraic reduction and computer algebra. Since Demailly's Santa Cruz lectures in 1995, the thresholds for the existence of such differentials---and consequently the limits of what $2$-jet techniques can accomplish toward the Kobayashi conjecture in dimension two---have been recognized as $d = 15$ in the compact case and $d = 11$ in the logarithmic case. While previous approaches were unable to reach these targets, the present work provides both the theoretical foundations and the algorithmic framework required to access them, and has already improved the known bounds to $d = 17$ and $d = 12, 13$, respectively.

As an unexpected byproduct, our computational method reveals the existence of nonzero negatively twisted invariant $2$-jet differentials with weighted degree $3$ for hyperelliptic-type equations of degree at least $11$ in the logarithmic case and degree at least $15$ in the compact case.

Moreover, in the logarithmic setting, we establish an effective quantitative refinement via a Second Main Theorem in Nevanlinna theory. Specifically, for every generic smooth curve $\mathcal{C} \subset \mathbb{P}^2$ of degree $d \geqslant 12$ and every algebraically nondegenerate entire curve $f \colon \mathbb{C} \to \mathbb{P}^2$, we prove the inequality
\[
T_f(r) \leqslant C_d \, N_f^{[1]}(r, \mathcal{C}) + o\big(T_f(r)\big) \quad \parallel,
\]
where $C_d$ is an explicit constant depending only on $d$. For very generic $\mathcal{C}$, the same estimate holds for every nonconstant entire curve. The notation ``$\parallel$'' signifies that the estimate holds for all $r>1$ outside a set of finite Lebesgue measure.
\end{abstract}

\date{\today}

	\maketitle
	
	\tableofcontents

\section{\bf Introduction}
\subsection{Motivation}
In 1967, Shoshichi Kobayashi introduced an intrinsically defined pseudo-distance for every connected complex manifold $X$. The infinitesimal version of this pseudo-distance was later characterized by Royden~\cite{Royden1971}. It is given by
\[
F_p(v) := \inf\Big\{ \lambda > 0 \ \Big\vert\ 
\exists\, f\colon \{|z|<1\} \to X \text{ holomorphic},\ 
f(0)=p,\ f'(0)=v/{\lambda} \Big\}
\quad (p\in X,\ v\in T_pX).
\]
If this Finsler pseudometric is nondegenerate everywhere, $X$ is called \emph{Kobayashi hyperbolic}. Throughout, an \emph{entire curve} in $X$ means a nonconstant holomorphic map $f\colon\mathbb{C}\to X$. For compact $X$, Brody's criterion~\cite{brody1978} states that Kobayashi hyperbolicity is equivalent to the absence of entire curves.

In dimension one, the Uniformization Theorem and Liouville's Theorem imply that all Riemann surfaces of genus $\geqslant 2$ are hyperbolic. It is natural to expect that, in higher dimensions, hyperbolic complex manifolds should likewise represent the ``most'' case (see the preface of~\cite{Kobayashi1998}). This leads to the following celebrated conjecture:

\begin{namedthm*}{Kobayashi Hyperbolicity Conjecture}
A generic hypersurface $H$ in the projective $n$-space $\mathbb{P}^n$ of degree $d\geqslant 2n-1$ is 
hyperbolic, for every $n\geqslant 3$.
\end{namedthm*}

Here and throughout the paper, we work over the complex number field $\mathbb{C}$. The term \emph{generic} means that the corresponding property holds outside a proper algebraic subvariety of the relevant parameter space.

\medskip
On the other hand, the Little Picard theorem may be viewed as stating that the complement of $3$ distinct points in the Riemann sphere $\mathbb{P}^1$ is hyperbolic. This leads naturally to the following conjecture in higher dimensions:

\begin{namedthm*}{Logarithmic Kobayashi Hyperbolicity Conjecture}
	The complement of a generic hypersurface $H\subset\mathbb{P}^n$ of degree $d\geqslant 2n+1$ is hyperbolic, for every $n\geqslant 2$.
\end{namedthm*}

The expected optimal degree bounds $2n-1$ in the compact case and $2n+1$ in the logarithmic case were suggested by Zaidenberg~\cite{Zaidenberg1987}. 
Fundamental work of Clemens~\cite{Clemens1986}, Ein~\cite{Ein1988,Ein1991}, and Voisin~\cite{Voisin1996} shows that, for very generic hypersurfaces of sufficiently high degree, all subvarieties are of general type; see also the logarithmic counterpart due to Pacienza and Rousseau~\cite{PacienzaRousseau2007}.   Here, a \emph{very generic} hypersurface means one lying in the complement of a countable union of proper Zariski-closed subsets of the parameter space.
Assuming the Green--Griffiths--Lang conjecture and its logarithmic analogue, these results would imply the above Kobayashi hyperbolicity conjectures in the expected optimal degree range.

\smallskip

Since then, the study of hyperbolicity has attracted considerable attention, in part due to its deep connections with Diophantine geometry. 
In particular, Lang \cite{Lang1986} conjectured that an algebraic variety $V_{\mathbb{K}}$ defined over a number field $\mathbb{K}$ contains only finitely many $\mathbb{K}$-rational points provided that, after a base change $\mathbb{K}\hookrightarrow\mathbb{C}$, the complex variety $V_{\mathbb{C}}$ is hyperbolic.

Substantial progress has been made over the past decades toward the above conjectures, especially in the regime where the degree of the hypersurface is sufficiently large compared with the dimension (see~\cite{Demailly2020} for an overview).

The Kobayashi hyperbolicity conjecture for surfaces in $\mathbb{P}^3$ was first verified by McQuillan~\cite{Mcquillan1999} for degrees $d\geqslant 36$, and subsequently improved by Demailly--El Goul~\cite{Demailly-Elgoul2000} to $d\geqslant 21$, and by P\u{a}un~\cite{mihaipaun2008} to $d\geqslant 18$.
In the logarithmic setting, hyperbolicity of the complement of a generic plane algebraic curve was established by Siu--Yeung~\cite{siu_yeung1996} for sufficiently high degree. For very generic plane curves, El Goul~\cite{ElGoul2003} subsequently improved the degree bound to \(d\geq 15\), and Rousseau~\cite{Rousseau2009} later lowered it further to \(d\geq 14\), using Siu's slanted vector fields and P\u{a}un's strategy (see also~\cite{ElGoul2003}).

In higher dimensions, the Kobayashi conjectures are known to hold for generic hypersurfaces of sufficiently large degree. We refer the reader to several key milestones~\cite{MR1989197, DMR2010, Siu2015, Brotbek2017, Riedl-Yang22, Berczi-Kirwan2024} and the extensive references therein for further background on the topic.

From a different perspective,  constructions of hyperbolic hypersurfaces were developed. For general $n$, the existence of such hypersurfaces in $\mathbb{P}^n$ was first established by Masuda and Noguchi~\cite{Masuda_Noguchi_1996}, who also gave examples for $n=3,4$ (with verification by ``Mathematica'' code) and considered the complement case. Subsequently, quadratic degree bounds were obtained by Siu and Yeung~\cite{siu_yeung1997} ($d(n)\geqslant 16(n-1)^2$) and later improved by Shiffman and Zaidenberg~\cite{shiffman_zaidenberg2002_pn} ($d(n)\geqslant 4(n-1)^2$). The current best asymptotic bound is $d\geqslant \big(\frac{n+2}{2}\big)^2$, due to Huynh~\cite{Huynh2016}.
In low dimensions, using a deformation  method of Zaidenberg~\cite{Zaidenberg1988}, examples of hyperbolic hypersurfaces of minimal degree $d=2n$ were constructed by Duval~\cite{duval2004} for $n=3$ and by Huynh~\cite{Huynh2015} for $4\leqslant n\leqslant 6$.

\subsection{Fundamental Vanishing Theorem for Entire Curves}

Most progress in the study of hyperbolicity problems has relied heavily on the theory of jet differentials, a framework originally introduced by Andr\'e Bloch~\cite{Bloch26}. Building on Bloch's ideas, Green and Griffiths~\cite{Green-Griffiths1980} introduced, for every projective manifold \(X\), the bundle \(E_{k,m}^{GG}T_X^* \to X\) of jet polynomials of order \(k\) and weighted degree \(m\) (see Section~\ref{subsect:jet-definition}).

Focusing on properties intrinsic to the entire curve itself, Demailly~\cite{Demailly1997} further refined the Green–Griffiths construction by introducing the subbundle \(E_{k,m}T_X^* \subset E_{k,m}^{GG}T_X^*\) (see Subsection~\ref{subsection:Invariant Jet}). In this work, we will primarily use \(E_{2,m}T_X^*\) with \(k=2\) and \(m \geqslant 3\).

Every global section of such a jet differential bundle, when twisted by the dual of an ample line bundle, gives rise to an \emph{algebraic differential equation} that must be satisfied by \emph{every} nonconstant entire holomorphic curve \(f\colon \mathbb{C} \to X\).

\begin{namedthm*}{Fundamental vanishing theorem} [\cite{Green-Griffiths1980, Siu-Yeung1996MathA, Demailly1997}]
Let \(X\) be a projective complex manifold and \(\mathscr{A}\) an ample line bundle on \(X\). For any nonzero global holomorphic jet differential with negative twist
\[
0 \neq \mathscr{P} \in H^0\big( X, E_{k,m}^{GG}T_X^* \otimes \mathscr{A}^{-1} \big),
\]
every nonconstant entire curve \(f\colon \mathbb{C} \to X\) satisfies the associated differential equation
\[
\mathscr{P}\big( f^{(1)}, \dots, f^{(k)} \big) \equiv 0.
\]
\end{namedthm*}

The principal challenge in the jet differential approach to the hyperbolicity problem lies in constructing a sufficient number of algebraically independent, negatively twisted jet differentials. The goal of such a construction is to eliminate all derivatives, thereby yielding \emph{purely algebraic equations} that constrain entire curves. For a generic hypersurface \(H \subset \mathbb{P}^{n+1}\) of dimension \(n \geqslant 3\), this challenge has been addressed in the works cited above~\cite{Siu2004, DMR2010, Siu2015, Brotbek2017, Riedl-Yang22, Berczi-Kirwan2024}, albeit only under the condition that the degree of \(H\) is sufficiently large relative to its dimension.

In dimension \(2\), for a smooth surface \(X \subset \mathbb{P}^3\), a result of Sakai~\cite{Sakai1979} shows that nontrivial symmetric differentials do not exist; equivalently, \(H^0(X, E_{1,m}T_X^*) = 0\) for all \(m\geqslant 1\). Hence, one must consider jet differentials of order \(k \geqslant 2\). 

By combining standard Riemann–Roch computations with the Bogomolov vanishing theorem~\cite{Bogomolov1978}, Demailly~\cite{Demailly1997} established the existence of nonzero negatively twisted $2$-jet differentials for smooth surfaces $X \subset \mathbb{P}^3$ of degree $d \geqslant 15$, thereby identifying $d = 15$ as the threshold down to which the $2$-jet approach has the potential to address the Kobayashi conjecture in the compact case. 

In the logarithmic setting, an analogous Riemann–Roch computation (see~\cite{Demailly1997, ElGoul2003}) shows that nontrivial negatively twisted invariant logarithmic $2$-jet differentials exist for complements of smooth curves $\mathcal{C} \subset \mathbb{P}^2$ of degree $d \geqslant 11$, marking $d = 11$ as the threshold down to which the $2$-jet approach has the potential to address the logarithmic Kobayashi conjecture.

Motivated by these thresholds, the present work introduces a new method that lays the theoretical and algorithmic groundwork for reaching them. Before presenting our approach, we first recall the classical strategy developed by P\u{a}un~\cite{mihaipaun2008}, which forms the starting point of our investigation.

\subsection{P\u{a}un's Strategy}
The proof of P\u{a}un~\cite{mihaipaun2008} relies on four essential ingredients.

\medskip
Denote by $\mathbb{P}^{N_d}$ the parameter space of surfaces of degree $d$ in $\mathbb{P}^3$, where $N_d = \binom{d+3}{3} - 1$. The first key result is the following celebrated theorem.

\begin{thm}[Clemens~\cite{Clemens1986}]\label{thm:Clemens-Xu}
    A very generic surface \( X_a \subset \mathbb{P}^{3} \) of degree \( d \geqslant 5 \) contains no rational curves. Furthermore, a very generic surface \( X_a \subset \mathbb{P}^{3} \) of degree \( d \geqslant 6 \) contains no elliptic curves.
\end{thm}
 The degree bound for elliptic curves was subsequently improved from \( d \geqslant 6 \) to \( d \geqslant 5 \) by Xu~\cite{Xu1994}.

\medskip
The second technique originates from McQuillan's deep work on entire curves tangent to foliations on surfaces of general type~\cite{Mcquillan1998}.

\begin{thm}[McQuillan~\cite{Mcquillan1998}]\label{thm:McQuillan}
Every parabolic leaf of an algebraic (multi-) foliation on a surface \(X\) of general type is algebraically degenerate.
\end{thm}

Consequently, the problem reduces to constructing two algebraically independent families of negatively twisted invariant $2$-jet differentials $\omega_{1,a}$ and $\omega_{2,a}$ on $X_a$, as $a$ varies over the parameter space of surfaces of degree $\geqslant 5$ in $\mathbb{P}^3$. By the fundamental vanishing theorem for entire curves, any entire curve $f \colon \mathbb{C} \to X_a$ must satisfy $f^*\omega_{1,a} \equiv 0$ and $f^*\omega_{2,a} \equiv 0$. This implies that $f$ is tangent to a foliation, and McQuillan’s Theorem then forces $f$ to be algebraically degenerate. Finally, Theorem~\ref{thm:Clemens-Xu} guarantees that a very generic surface $X_a \subset \mathbb{P}^3$ of degree $\geqslant 6$ contains no rational or elliptic curves. Combining these results yields the hyperbolicity of $X_a$ for very generic parameters $a$.

In practice, however, proving the existence of two such independent families of differentials $\omega_{1,a}$ and $\omega_{2,a}$ presents a substantial difficulty. The remaining two techniques are precisely designed to circumvent this obstacle.

\medskip
The third technique uses Miyaoka's method~\cite{Miyaoka} (see also Lu--Yau~\cite{Lu-Yau}) in the form developed by Demailly and El Goul~\cite{Demailly-Elgoul2000}. It employs a Riemann--Roch type computation on the zero locus \(\{\omega_1 = 0\}\) inside the second level \(X_2\) (here we omit the parameter $a$ for brevity) of the Demailly--Semple tower.

More precisely, let \(c_1\) and \(c_2\) denote the first and second Chern classes of \(X\), respectively. Suppose there exists a nonzero  irreducible negatively twisted invariant $2$-jet differential
\[
\omega_1 \in H^0\bigl(X,\; E_{2,m}T_X^* \otimes \mathcal{O}_X(-t(d-4))\bigr)
\]
that satisfies
\[
t \;<\; \frac{13c_1^2 - 9c_2}{12c_1^2} \, m .
\]
Then one can show that the line bundle \(\mathcal{O}_{X_2}(1)\) is big on the “nontrivial” component of the zero locus \(\{\omega_1=0\}\subset X_2\); see Proposition~\ref{prop:DEG-numerical-criterion}. By the bigness of \(\mathcal{O}_{X_2}(1)\) restricted to this component, one can produce a second nonzero negatively twisted differential \(\omega_2\) defined only there. Because the entire curve \(f\) satisfies \(f^*\omega_1\equiv 0\), its lift \(f_{[2]}\) lies entirely inside \(\{\omega_1=0\}\); consequently \(f\) also satisfies \(f^*\omega_2 \equiv 0\). Hence \(f\) is tangent to the foliation defined by the pair \((\omega_1,\omega_2)\), and McQuillan’s theorem can be applied as before.

\medskip
One therefore focuses on the remaining case where 
\[
t \;\geqslant\; \frac{13c_1^{2}-9c_2}{12c_1^{2}} \, m.
\]
For this situation, P\u{a}un developed a fourth key technique: the construction of \emph{slanted vector fields}, a method originally envisioned by Siu~\cite{Siu2004} in the spirit of Clemens~\cite{Clemens1986}. The method works effectively when the vanishing order \(t(d-4)\) of the initial nonzero irreducible negatively twisted \(2\)-jet differential \(\omega_1\) exceeds the pole order of a carefully chosen slanted vector field \(V_1\). Differentiating \(\omega_1\) by \(V_1\) then yields a second algebraically independent negatively twisted jet differential \(\omega_2\), thereby completing the proof.

When \( d \geqslant 18 \) and \( m \geqslant 6 \), it is  assumed that the vanishing order of the initial \(2\)-jet differential \(\omega_1\) satisfies
\[
t(d-4) \; \geqslant \; \frac{13c_1^2-9c_2}{12c_1^2} \, m (d-4) \; > \; 7.
\]
Thus slanted vector fields of pole order \(7\) can be applied directly, as they are abundant and (generically) globally generated.

In the complementary range \( 3 \leqslant m \leqslant 5 \) (still with \( d \geqslant 18 \)), the same assumption ensures that
\[
t(d-4) \; > \; 3,
\]
so that, at least formally, only slanted vector fields of pole order \(3\) are accessible. Their effective construction, which ultimately succeeds, rests on a key technical contribution of Păun~\cite{mihaipaun2008}. The principal difficulty lies in ensuring that the resulting differential \(\omega_2\) is algebraically independent of \(\omega_1\). Păun proceeds by contradiction: assuming otherwise leads to a certain ODE, and he shows that this ODE admits no nontrivial solution. This ODE‑based argument, however, breaks down when \( m \geqslant 6 \), leaving the extension to that case an open problem.


\medskip
In the logarithmic case, by adapting the three key techniques described above, Rousseau~\cite{Rousseau2009} proved that the complement \(\mathbb{P}^2 \setminus \mathcal{C}\) is hyperbolic for very generic plane curves \(\mathcal{C}\) of degree \(d \geqslant 14\), thus improving the previous bounds of Siu--Yeung~\cite{siu_yeung1996} and El Goul~\cite{ElGoul2003}.

\subsection{Our Main Results}
In our earlier work~\cite{Hou-Huynh-Merker-Xie2025}, we introduced a method for proving the vanishing of negatively twisted logarithmic \(2\)-jet differentials, motivated by an effective Second Main Theorem in Nevanlinna theory.  
That paper treated three generic conics in \(\mathbb{P}^2\), a setting where the computations simplify considerably.  
Writing the three conics as homogeneous polynomials \(A,B,C\), one can work on a suitable large Zariski‑open set where the logarithmic \(2\)-jet differential frame is transparently generated by \(\bigl(\log\frac{A}{C}\bigr)'\!\), \(\bigl(\log\frac{B}{C}\bigr)'\!\) and their Wronskian determinant.  
This convenient semi‑global frame drastically reduces the algebraic complexity.

In the present paper we study two settings in which such a frame is no longer available:  
the complement of a single smooth curve in \(\mathbb{P}^2\), and a smooth surface \(X=\{R=0\}\) in \(\mathbb{P}^3\).  
For surfaces, the defining equation \(R=0\) itself imposes an additional relation among the jet coordinates, which must be taken into account throughout the reduction; this introduces further algebraic complications.  
In both cases, substantial extra effort is required to control the jet differentials across all affine charts and to convert the holomorphicity conditions into a tractable linear system.  
The earlier three‑conics approach was tailored to that specific configuration, but its conceptual framework is the one we adapt and significantly extend here.

Building on this framework, we develop a new method aimed at reaching the \(2\)-jet thresholds \(d=15\) in the compact case and \(d=11\) in the logarithmic case for the Kobayashi conjectures in dimension two.  
For more than two decades, a major obstacle to lowering the degree bounds below \(18\) (compact) and \(14\) (logarithmic) has been the lack of essential vanishing results for certain invariant (logarithmic) \(2\)-jet differentials.  
In this work, we make a decisive advance by proving the following two key vanishing lemmas, which allow us to lower the bounds to \(d=17\) in the compact case and to \(d=12,13\) in the logarithmic case.

\begin{lem}[Key Vanishing Lemma – Logarithmic Case]
\label{lem-1.2}
Let \(\mathcal{C} \subset \mathbb{P}^2\) be a generic curve of degree \(d\). Then:
\begin{itemize}
    \item For \(d = 13\), \(H^{0}\!\bigl(\mathbb{P}^{2},\;E_{2,m}T_{\mathbb{P}^{2}}^{*}(\log \mathcal{C}) \otimes \mathcal{O}(-t)\bigr) = 0\) for all \((m,t)\) in
    \[
    \{(3,3),\,(4,4),\,(5,5),\,(6,6),\,(7,7),\,(8,7)\}.
    \]
    \item For \(d = 12\), \(H^{0}\!\bigl(\mathbb{P}^{2},\;E_{2,m}T_{\mathbb{P}^{2}}^{*}(\log \mathcal{C}) \otimes \mathcal{O}(-t)\bigr) = 0\) for all \((m,t)\) in
    \[
    \{(3,2),\,(4,2),\,(5,3),\,(6,3),\,(7,4),\,(8,4),\,(9,5),\,(10,5),\,(11,6),\,(12,6),\,(13,7),\,(14,7)\}.
    \]
\end{itemize}
\end{lem}

The proof of Lemma~\ref{lem-1.2} is given in Section~\ref{sect:proof of lemma 2}. It relies on a delicate adaptation and generalization of the method introduced in our previous work~\cite{Hou-Huynh-Merker-Xie2025}, following a three-step strategy.

\smallskip
\begin{itemize}
    \item \textbf{Algebraic Reduction.} Working in affine coordinates \((x,y)\) of \(\mathbb{C}^2 \subset \mathbb{P}^2\), let \(R\) be the polynomial defining the curve \(\mathcal{C}\) in this affine chart. We first derive the most general local expression of a global section of the negatively twisted logarithmic invariant \(2\)-jet differential bundle on an open set where a given partial derivative of \(R\) is nonvanishing, say \(\{R_y \neq 0\}\). Such sections are expressed as polynomials in the jet variables divided by powers of the nonvanishing partial derivative. Imposing holomorphicity on the alternative chart \(\{R_x \neq 0\}\) yields severe restrictions on the admissible pole orders. A key innovation of the present paper, compared to our earlier work~\cite{Hou-Huynh-Merker-Xie2025}, is the discovery and systematic exploitation of the triangular structure of the transition formulas; this allows the restrictions to be determined recursively. The analysis reduces an a priori infinite-dimensional problem to a finite-dimensional linear system satisfied by the coefficients of the polynomials appearing in the local expression. Indeed, for generic \(R\), the presumed negatively twisted logarithmic invariant \(2\)-jet differentials must take a finite form (see Proposition~\ref{prop:4.2} and \eqref{best J_X widehatR inhomogeneous}). The remaining task is to determine whether this finite-dimensional linear system admits nontrivial solutions. The same algebraic reduction strategy will be applied, with the necessary modifications, to the compact case treated later in this paper. The ideas and methods developed in this step constitute a foundational contribution of our paper; they will serve as a key tool in future work on the lower degree case, in the study of higher‑order jet differentials, and more broadly in other related problems.

    \smallskip
    
\item \textbf{Computational Verification.} Solving the resulting linear systems directly for a generic polynomial \(R\) would be computationally prohibitive, as the number of variables easily exceeds ten thousand; even extracting the linear system from the geometric constraints is itself a nontrivial task. A key technical innovation lies in our choice of a special polynomial \(R\). The trick is to select an equation that is a deformation of a Fermat-type polynomial, so that the relevant partial derivative simplifies drastically to \(R_y = d y^{d-1}\). This makes the divisibility conditions particularly simple to implement: requiring divisibility by \((y^{d-1})^i\) means that after clearing denominators, all coefficients of \(y^k\) for \(k < (d-1)i\) must vanish identically. These conditions translate into explicit linear systems on the coefficients of the polynomials appearing in the local expressions. Solving them by computer algebra confirms that, for each specified \((m,t)\) with \(d = 12,13\), the only solution is the trivial one.\footnote{The Maple file, along with accompanying explanatory and output PDFs, is available at: \url{https://xiesongyan.github.io/}.} 

    \smallskip
    
    \item \textbf{Geometric Extension.} Finally, by the semicontinuity theorem in algebraic geometry, the vanishing established for this single carefully chosen example automatically extends to generic curves. This standard argument allows us to deduce the generic vanishing statements from computations performed on a specific smooth curve \(\mathcal{C}\).
\end{itemize}

\begin{lem}[Key Vanishing Lemma – Compact Case]
\label{lem-1.1}
Let \(X \subset \mathbb{P}^3\) be a generic surface of degree \(d=17\). Then \(H^{0}\!\bigl(X,\;E_{2,m}T_{X}^{*} \otimes \mathcal{O}_X (-t)\bigr) = 0\) for all \((m,t)\) in
\[
\{(3,3),\,(4,4),\,(5,5),\,(6,6),\,(7,7)\}.
\]
\end{lem}

The proof of Lemma~\ref{lem-1.1} is given in Section~\ref{sect:3}. Building upon the techniques developed for the logarithmic case, a similar argument allows us to control the structure of the presumed negatively twisted invariant $2$-jet differentials, showing that they must take a finite form; see Proposition~\ref{Prp-compact-Jyx}  and~\eqref{best J_yx}. Nevertheless, the compact case presents additional algebraic challenges. In affine coordinates \((x,y,z)\) of \(\mathbb{C}^3\subset \mathbb{P}^3\), let \(R \in \mathbb{C}[x,y,z]\) be the polynomial defining the surface \(X\). The key difference from the logarithmic case lies in the local coordinate ring of \(X\): there it is the polynomial ring \(\mathbb{C}[x,y]\), whereas here it is the quotient ring \(\mathbb{C}[x,y,z]/(R)\). To convert divisibility conditions in this quotient ring into concrete constraints in \(\mathbb{C}[x,y,z]\), we introduce a new theoretical framework that ensures effective computability (see Lemmas~\ref{Lemma-divisibility}).

\begin{itemize}
    \item \textbf{An algorithmic reduction.} 
We select the surface \(X \subset \mathbb{P}^3\) to be a deformation of the Fermat type:
    \[
    (X^d + Y^d + Z^d + T^d) + X^a T^b = 0,
    \]
    where the positive integers \(a, b\) satisfy \(a+b = d\) and are chosen so that \(\lvert a-b \rvert\) is small. The Fermat type is chosen for its simplicity, which will significantly reduce the subsequent computational complexity. We must include the extra term \(X^a T^b\) precisely because, as shown by Brotbek's Wronskian construction for \(1\)-jet differentials obtained via cohomological computations~\cite{Brotbek2016Symmetricdifferentialforms}, the pure Fermat surface itself (with \(d \geqslant 9\)) is known to carry nonvanishing negatively twisted invariant jet differentials. The observation that this construction extends to the \(2\)-jet case was noted, in a geometric way, by Xie~\cite[Proposition~6.10]{xie2015ampleness} (see also \cite{Xie2018, BD2018, Brotbek2017}).

    Working in the affine chart \(\{T \neq 0\}\), the equation dehomogenizes as
    \[
    R(x,y,z) = 1 + x^{d} + y^{d} + z^{d} + x^a,
    \]
    where \(a\) will be chosen as \(\lfloor d/2 \rfloor\). This particular form of the defining equation yields two crucial simplifications. 
    
    The first is that the relevant partial derivative simplifies to \(R_z = d z^{d-1}\), which is independent of \(x\) and \(y\). This renders the divisibility conditions—requiring an expression to be divisible by \((R_z)^i\) in the quotient ring \(\mathbb{C}[x,y,z]/(R)\)—far more tractable (see Lemma~\ref{Lemma-divisibility}).
    
    Second, since the extra term \(X^a T^b\) involves no variable \(Y\), we can use the relation
    \[
    y^{d} = -1 - x^{d} - z^{d} - x^a
    \]
    on the affine chart \(X\cap \{T\neq 0\}\) to express every element in the coordinate ring \(\mathbb{C}[x,y,z]/(R)\) uniquely and efficiently as a polynomial in \(\mathbb{C}[x,y,z]\) with \(\deg_y < d\). The reduction is achieved by directly eliminating all higher powers \(y^k\) with \(k \geqslant d\): performing Euclidean division on the exponent \(k\) by \(d\) gives \(k = d p + q\), and substituting \((y^d)^p = (-1 - x^{d} - z^{d} - x^a)^p\) reduces the term to \(y^q\) multiplied by an expression in \(x\) and \(z\). The strategic choice of \(y\)—rather than the seemingly more natural choice \(z\)—as the reduction variable is essential to our algorithm (see again Lemma~\ref{Lemma-divisibility}).

This careful choice of the defining equation enables us to rapidly set up a finite linear system, transforming a computationally prohibitive problem into a tractable one. As a further technical simplification, one may work modulo a prime --- say $p = 19$ --- since a linear system with integer coefficients that has no nontrivial solution modulo $p$ also has none over $\mathbb{Z}$ (and hence over $\mathbb{C}$). To demonstrate the effectiveness of our method, we focus here on the degree $17$ case. Finally, we implement this procedure in Maple to verify all the required vanishing results.\footnote{The Maple file, together with an explanatory PDF and a PDF of the run results, is available at: \url{https://xiesongyan.github.io/}.}
\end{itemize}

\smallskip
Equipped with the two Key Vanishing Lemmas, we are now ready to state our main results. A notable methodological departure from the earlier works of P\u{a}un~\cite{mihaipaun2008} and Rousseau~\cite{Rousseau2009} is that our proof avoids the subtle and somewhat exceptional slanted vector fields of pole order $3$. For the qualitative hyperbolicity bounds in Theorems~\ref{thm:compact-hyperbolic} and~\ref{thm:log-hyperbolic}, our proof avoids the exceptional pole-order-3 vector fields and uses only the abundant pole-order-7
fields. The pole-order-3 construction is retained only in the quantitative optimization for Theorem~\ref{SMT for curves of d geqslant 12} in the $d \geqslant 14$ case.

The later pole-lowering arguments of
\cite{CuiHouLiuXie2026,HouWangXie2026} also show that the present proof can
be reorganized with pole order $4$.  We retain pole order $7$ here because
it gives the cleanest presentation of the argument developed in this paper;
the lower pole order changes the finite Key Vanishing list but introduces no
new principle needed for the degree bounds proved below.

\begin{thm}\label{thm:GGL-bound}
Let $d \geqslant 17$. There exists a Zariski-open subset $U \subset \mathbb{P}^{N_d^3}$ in the parameter space of degree-$d$ surfaces in $\mathbb{P}^3$ with the following property: for each $a \in U$, there exists a proper subvariety $S_a$ in the first level $X_{a,1}$ of the Demailly--Semple tower for the surface $X_a\subset \mathbb{P}^3$ such that:
\begin{enumerate}
    \item The family $\{ S_a \}_{a \in U}$ depends algebraically on the parameter $a \in U$;
    \item For every entire curve $f \colon \mathbb{C} \to X_a$, its lift $f_{[1]} \colon \mathbb{C} \to X_{a,1}$ satisfies $f_{[1]}(\mathbb{C}) \subset S_a$.
\end{enumerate}
\end{thm}

The algebraic dependence on the parameter in
Theorem~\ref{thm:GGL-bound} allows us to apply the family criterion of
Xie--Zhao~\cite[Theorem~B]{xie2026kobayashihyperbolicitygeneralsurfaces}.
Together, these results imply that a generic surface of degree
$d\geqslant17$ contains no rational or elliptic curves and,
by Theorem~\ref{thm:GGL-generic}, is Kobayashi hyperbolic;
see Corollary~\ref{cor:generic-absence} and Remark~\ref{rem 7.5}.
The Poincar\'e-type bound underlying this application is discussed
in Remark~\ref{rem 7.4}.

The case $d \geqslant 18$ was essentially established by Păun~\cite{mihaipaun2008} (earlier, Demailly--El Goul~\cite{Demailly-Elgoul2000} obtained a similar conclusion for $d \geqslant 21$). Their results, together with our application of the Grauert's
direct image theorem (an idea suggested to us by Siu, see Section~\ref{subsection: 5.2}), already yield Theorem~\ref{thm:GGL-bound} for $d \geqslant 18$. Consequently, we will focus on the remaining cases $d = 17$. The proof of Theorem~\ref{thm:GGL-bound} for $d = 17$ will be presented in Sections~\ref{subsection: 5.2} and~\ref{subsection: 5.3}.

\medskip 
Theorem~\ref{thm:GGL-bound} shows that for a generic surface $X_a$ of degree $d \geqslant 17$, every entire curve $f: \mathbb{C} \to X_a$ is tangent to a multi-foliation. Combined with McQuillan's Theorem~\ref{thm:McQuillan}, this immediately yields the following  conclusion.

\begin{thm}\label{thm:GGL-generic}
For a generic surface $X \subset \mathbb{P}^3$ of degree $d \geqslant 17$, every entire curve $f: \mathbb{C} \to X$ is algebraically degenerate. \qed
\end{thm}

For a very generic surface $X_a$ of degree $\geqslant 17$, Clemens' Theorem~\ref{thm:Clemens-Xu} asserts the absence of rational and elliptic curves. Hence, Theorem~\ref{thm:GGL-generic} yields the following hyperbolicity result.

\begin{thm}\label{thm:compact-hyperbolic}
    A very generic surface \( X \subset \mathbb{P}^{3} \) of degree \( d \geqslant 17\) is Kobayashi hyperbolic. \qed
\end{thm}

 \noindent 
This strengthens P\u{a}un's result~\cite{mihaipaun2008} by lowering the degree threshold from \( d \geqslant 18 \) to \( d = 17 \).

\medskip
In the logarithmic setting, we prove the analogous result:

\begin{thm}\label{thm:log-hyperbolic}
    Let \(\mathcal{C} \subset \mathbb{P}^{2}\) be a very generic algebraic curve of degree \(d \geqslant 12\). Then the complement \(\mathbb{P}^{2} \setminus \mathcal{C}\) is Kobayashi hyperbolic and hyperbolically embedded.
\end{thm}

The case $d \geqslant 14$ was already established by Rousseau in~\cite{Rousseau2009} (see Remark~\ref{rem:vanishing results for log case d >= 14} for an alternative approach).
Consequently, we will
focus on the remaining cases $d=12$ and $13$. 
The proof of Theorem~\ref{thm:log-hyperbolic} will be given in Section~\ref{subsection: 5.4}.

\medskip

Moreover, we strengthen the above hyperbolicity result by establishing a Second Main Theorem in Nevanlinna theory.

\begin{thm}\label{SMT for curves of d geqslant 12}
    For every integer $d\geqslant12$, there exists a nonempty Zariski-open
    subset $U_d$ of the parameter space of smooth degree-$d$ plane curves
    such that, for every $a\in U_d$ and every algebraically nondegenerate
    entire curve $f\colon\mathbb{C}\to\mathbb{P}^2$, one has
    \begin{equation}\label{eq:smt for curves of d geqslant 12}
        T_f(r) \leqslant C_d \, N_f^{[1]}(r, \mathcal{C}_a) + o\bigl(T_f(r)\bigr) \quad\parallel,
    \end{equation}
    where \(C_d\) is a constant depending only on \(d\). Explicitly,
    \begin{itemize}
        \item For $d \geqslant 14$,
            \begin{equation*}
                C_d = \frac{72}{32d - 117 - \sqrt{640d^2 - 1248d + 1017}}.
            \end{equation*}

        \item For \(d = 13\),
            \begin{equation*}
                C_{13} = \frac{63}{393-2\sqrt{38607}} \approx 2357.920.
            \end{equation*}
        \item For \(d = 12\),
            \begin{equation*}
                C_{12} = \frac{30}{173 - \sqrt{29809}} \approx 86.4.
            \end{equation*}
    \end{itemize}
    Moreover, for very generic parameters $a\in U_d$, the same estimate
    holds for every nonconstant entire curve
    $f\colon\mathbb{C}\to\mathbb{P}^2$.
\end{thm}

\begin{rem}\rm
For \(d=13\), using only the vanishing results in Key Vanishing
Lemma~\ref{lem-1.2} leads, within the present argument, to an estimate
with a relatively large constant \(C_{13}\).
If one additionally proves that
\[
H^{0}\!\bigl(\mathbb{P}^{2},\;
E_{2,m}T_{\mathbb{P}^{2}}^{*}(\log\mathcal{C})
\otimes\mathcal{O}(-t)\bigr)=0
\qquad\text{for }(m,t)\in\{(6,5),(7,6)\},
\]
then a smaller splitting parameter can be used in the proof,
yielding a substantially smaller constant \(C_{13}\).
\end{rem}

We recall the standard notation of Nevanlinna theory in Section~\ref{Nevanlinna theory notation}. A related result due to Campana, Darondeau, and Rousseau~\cite[Proposition~4.8]{Campana-Darondeau-Rousseau2020} shows that every entire curve ramifying over a smooth curve of degree $d \geqslant 12$ with sufficiently high order must satisfy an algebraic differential equation of order $2$.  Here Theorem~\ref{SMT for curves of d geqslant 12} provides a quantitative estimate in Nevanlinna theory for algebraically nondegenerate entire curves over generic boundaries, and for every nonconstant entire curve over very generic boundaries. Its proof follows the strategy introduced in our previous work~\cite{Hou-Huynh-Merker-Xie2025} and is given in Section~\ref{subsect:SMT} (see also Remark~\ref{smt for degree 11} for the case $d = 11$).


\medskip

\subsection{Overview of the Paper's Approach}

The story begins with Andr\'e Bloch's foundational work a century ago~\cite{Bloch26}, which established that complex hyperbolicity is governed by the positivity of $k$-jet differential bundles. This positivity carries two distinct meanings: abstract existence of negatively twisted $k$-jet differentials (guaranteed by Riemann–Roch) and effective control of their base locus. The latter is far more delicate and lies at the heart of the field. To address it, one adopts a two-pronged strategy: the intrinsic side uses abstract algebro-geometric structures, while the extrinsic side introduces explicit constructions such as Siu's slanted vector fields. Yet even this combination falls short for low-degree cases. The present paper supplies the missing piece: the key vanishing lemmas, which prove that the weighted degree $m$ and negative twist $t$ of certain abstractly obtained jet differentials cannot fall into certain critical ranges, thereby bridging the gap between abstract existence and effective vanishing.

Our proof of the Kobayashi hyperbolicity conjectures in dimension two for degrees $\geqslant 17$ (compact case) and $\geqslant 12$ (logarithmic case) weaves together intrinsic and extrinsic algebro-geometric arguments. On the intrinsic side, we rely on the Demailly–Semple tower, Riemann–Roch computations, and the direct image theorem. On the extrinsic side — for those $(m,t)$ beyond intrinsic reach — we employ Siu's method of slanted vector fields of pole order $7$ and develop a systematic approach that culminates in the key vanishing lemmas. A useful analogy is the Riemann–Roch theorem for curves, which admits both an intrinsic sheaf-theoretic proof and an extrinsic proof based on coordinate changes and pole–zero counting (cf.~\cite{Griffiths1989}); our approach follows a similar extrinsic philosophy, though the algebraic complexity is vastly higher in dimension two.

The linear algebra techniques of Siu and Yeung~\cite{siu_yeung1996, Siu2015} were designed to construct sections when the degree is sufficiently high. In this paper, we develop a method that, in principle, computes all sections for a given $(m,t)$. However, such nonzero sections exist only when $m$ is taken sufficiently large — far beyond current computational reach — leaving explicit constructions in low-degree settings elusive. Nonetheless, by proving vanishing results, we establish hyperbolicity without ever exhibiting an explicit section.

Our first key contribution is to determine the finite algebraic structure of negatively twisted invariant $2$-jet differentials: we prove that any such global section must take a specific finite form (see Proposition~\ref{prop:4.2} together with \eqref{best J_X widehatR}, and Proposition~\ref{Prp-compact-Jyx} together with \eqref{best J}), reducing an infinite-dimensional problem to a finite-dimensional linear system. While a similar local description appeared in our previous work~\cite{Hou-Huynh-Merker-Xie2025}, it was tailored to three conics; the present general form applies to arbitrary curves and surfaces.

Our second contribution is the design of efficient algorithms to solve these linear systems. For the logarithmic case, we adapted the algorithmic framework of~\cite{Hou-Huynh-Merker-Xie2025}. The compact case required a fundamentally new theoretical approach, because the local coordinate ring is no longer $\mathbb{C}[x,y]$ but the quotient ring $\mathbb{C}[x,y,z]/(R)$, necessitating a novel reduction procedure to transform divisibility conditions into concrete algebraic constraints (see Lemmas~\ref{Lemma-divisibility}).

Our third contribution is the implementation of these algorithms in Maple, which systematically verifies the vanishing results for $d = 12,13$ (logarithmic) and $d = 17$ (compact). The scale is substantially larger than in our previous work: the hardest case in~\cite{Hou-Huynh-Merker-Xie2025} involved “only” $81{,}840$ variables, whereas here the largest linear systems involve over $300{,}000$ variables (logarithmic) and over $600{,}000$ variables (compact). Solving systems of this size required careful algorithm design and optimization within Maple's capabilities.

These vanishing lemmas improve the degree bounds for the Kobayashi conjecture, achieving $17$ in the compact case and $12$, $13$ in the logarithmic case.

The Kobayashi hyperbolicity conjectures conceal a vast ocean of complexity.  The degree of the hypersurface serves as a gauge of depth: lower degrees bring us closer to the seabed of hidden structures, where the true nature of the problem reveals itself. In dimension two, reaching the thresholds $d = 15$ (compact case) and $d = 11$ (logarithmic case) would mark the limit of what $2$-jet techniques alone can achieve, and would illuminate the intrinsic complexity beneath — a complexity rooted in deep theoretical structures, yet demanding cross-disciplinary collaboration to uncover.

Some parts of the algebraic reductions and our Maple implementation are designed for a one‑term perturbation of the Fermat equation. This simple choice works remarkably well for the cases we have settled — degree~\(17\) in the compact setting and degrees~\(12\) and~\(13\) in the logarithmic setting — but it is too rigid to reach the ultimate \(2\)-jet thresholds. Indeed, for the small parameters \((m,t)=(3,1)\) both in the compact case with \(d=15\) and in the logarithmic case with \(d=11\), the algorithm produces a non‑zero space of solutions, thereby exhibiting nontrivial negatively twisted jet differentials (see Remarks~\ref{surprise-1} and~\ref{surprise-2}). This outcome is actually reassuring: it confirms that the code correctly detects the existence of such differentials when they are present, and it shows that the hyperelliptic‑type equation we have used is not sufficiently generic for those lowest degrees.

To prove the Key Vanishing Lemma for degree~\(15\) (compact) and degree~\(11\) (logarithmic) one must start from a more involved defining equation, for instance a hypersurface that contains two carefully chosen extra monomials beyond the Fermat equation. The resulting linear systems will be substantially larger --- the number of unknown coefficients easily exceeds \(60\) million in the compact case and \(20\) million in the logarithmic case --- and they lie far beyond the capabilities of the present Maple implementation. Realising such a computation therefore poses a formidable engineering challenge: it will require an efficient C\texttt{++} implementation, massive parallelisation, and likely a more refined choice of the defining equation to keep the systems within manageable bounds. Nevertheless, we are optimistic that with sustained efforts these two bottom cases will eventually be settled.

\medskip

\subsection{Structure of the Paper}
Section~\ref{sect:2} provides the necessary background material.
Sections~\ref{sect:proof of lemma 2} and~\ref{sect:3}  form the core of the paper, where we establish two key vanishing lemmas for  invariant $2$-jet differentials with low vanishing orders.
For completeness, Section~\ref{sect:4} recalls Siu's slanted vector field method, following the constructions of P\u{a}un and Rousseau.
Section~\ref{sect:5} contains the proofs of the main theorems.
Finally, Section~\ref{section:6} outlines several directions for future research.

\section{\bf Preliminaries}
\label{sect:2}

\subsection{Jet Bundles and Jet Differentials}
\label{subsect:jet-definition}

Let $X$ be an $n$-dimensional complex manifold. The $k$-jet bundle $J_k X$ consists of equivalence classes of holomorphic germs $f \colon (\mathbb{C},0) \to (X,x)$, where two germs $f$ and $g$ are equivalent if their Taylor expansions agree up to order $k$ in some local coordinates around $x$. The equivalence class of a germ $f$ is called the \emph{$k$-jet of $f$} and is denoted by $j_k (f)$.

For a point $x \in X$, denote by $J_k(X)_x$ the space of all $k$-jets of holomorphic germs $(\mathbb{C},0) \to (X,x)$. The total $k$-jet bundle is then defined as the disjoint union
\[
J_k(X) := \coprod_{x \in X} J_k(X)_x.
\]

Together with the natural projection
\[
\pi_k \colon J_k(X) \longrightarrow X, \qquad j_k(f) \longmapsto x = f(0),
\]
the set $J_k(X)$ carries the structure of a holomorphic fibre bundle over $X$. 

When $k = 1$, the fibre $J_1(X)_x$ is naturally identified with the tangent space $T_xX$, and $J_1(X)$ is canonically isomorphic to the holomorphic tangent bundle $T_X$. For $k \geqslant 2$, however, the fibres $J_k(X)_x$ no longer carry a uniform vector space structure that varies holomorphically with $x$; consequently, $J_k(X)$ is in general only a (non‑linear) fibre bundle, not a vector bundle.

Let \(U \subset X\) be an open subset. Given a holomorphic \(1\)-form \(\omega \in H^0(U, T_X^*)\) and a \(k\)-jet \(j_k(f) \in J_k(X)|_U\), the pullback \(f^*\omega\) can be written as \(A(z)\,dz\) for some holomorphic function \(A\). Because the derivatives \(A^{(j)}\) (\(0 \leqslant j \leqslant k-1\)) depend only on the equivalence class \(j_k(f)\) and not on the choice of representative \(f\), the form \(\omega\) induces a well-defined holomorphic map
\begin{equation}\label{trivialization-jet}
\tilde{\omega} \colon J_k(X)|_U \longrightarrow \mathbb{C}^k,\qquad
j_k(f) \longmapsto \bigl(A(0), A'(0), \dots, A^{(k-1)}(0)\bigr).
\end{equation}

Now suppose \(\omega_1, \dots, \omega_n\) is a local holomorphic coframe on \(U\), i.e. \(\omega_1 \wedge \cdots \wedge \omega_n \neq 0\). Using the maps \(\tilde{\omega}_i\) defined as in~\eqref{trivialization-jet}, we obtain a trivialisation
\[
J_k(X)|_{U} \;\longrightarrow\; U \times (\mathbb{C}^k)^n,
\]
\[
\sigma \;\longmapsto\; \bigl(\pi_k \circ \sigma;\; \tilde{\omega}_1 \circ \sigma, \dots, \tilde{\omega}_n \circ \sigma\bigr).
\]

The components \(x_i^{(j)}\) (\(1 \leqslant i \leqslant n,\; 1 \leqslant j \leqslant k\)) of the vectors \(\tilde{\omega}_i \circ \sigma\) are called \emph{jet-coordinates}.

Now let \(D\) be a normal crossing divisor on \(X\). Following Iitaka~\cite{Iitaka1982}, the \emph{logarithmic cotangent bundle of \(X\) along \(D\)}, denoted by \(T_X^*(\log D)\), is the locally free sheaf that is locally generated by the forms
\[
\frac{\mathrm{d} z_1}{z_1},\dots,\frac{\mathrm{d} z_\ell}{z_\ell}, \mathrm{d} z_{\ell+1},\dots, \mathrm{d} z_n,
\]
where \(z_1,\dots,z_\ell,z_{\ell+1},\dots,z_n\) are local coordinates centred at a point \(x\) in which \(D\) is given by
\[
D = \{z_1 \cdots z_\ell = 0\}.
\]
Here \(\ell = \ell(x)\) denotes the number of local branches of \(D\) passing through \(x\).

A holomorphic section \(s \in H^0(U, J_k(X))\) over an open set \(U \subset X\) is called a \emph{logarithmic \(k\)-jet field} if, for every open subset \(U' \subset U\) and every logarithmic \(1\)-form \(\omega \in H^0\bigl(U', T_X^*(\log D)\bigr)\), the composition \(\tilde{\omega} \circ s\) is holomorphic on \(U'\), where \(\tilde{\omega}\) is the map induced by \(\omega\) as in~\eqref{trivialization-jet}. The collection of all such logarithmic \(k\)-jet fields forms a subsheaf of \(J_k(X)\); this subsheaf is precisely the sheaf of sections of a holomorphic fibre bundle over \(X\), called the \emph{logarithmic \(k\)-jet bundle of \(X\) along \(D\)} and denoted by \(J_k(X,-\log D)\) \cite{Noguchi1986}.


A \emph{jet differential} (resp. \emph{logarithmic jet differential}) of order \(k\) and weighted degree \(m\) at a point \(x \in X\) is a polynomial \(Q(f', \dots, f^{(k)})\) on the fibre of \(J_k(X)\) (resp. \(J_k(X, -\log D)\)) over \(x\) that is homogeneous with respect to the \(\mathbb{C}^*\)-action
\[
Q\bigl( \lambda f', \lambda^2 f'', \dots, \lambda^k f^{(k)} \bigr) = \lambda^m Q\bigl( f', f'', \dots, f^{(k)} \bigr) \qquad (\forall \lambda \in \mathbb{C}^*).
\]

Denote by \(E_{k,m}^{GG}T_X^*\big|_x\) (resp. \(E_{k,m}^{GG}T_X^*(\log D)\big|_x\)) the vector space of such polynomials and set
\[
E_{k,m}^{GG}T_X^* \coloneqq \bigcup_{x \in X} E_{k,m}^{GG}T_X^*\big|_x,
\qquad
E_{k,m}^{GG}T_X^*(\log D) \coloneqq \bigcup_{x \in X} E_{k,m}^{GG}T_X^*(\log D)\big|_x.
\]

By Faà di Bruno’s formula \cite{Constantine1996, Merker2015}, both \(E_{k,m}^{GG}T_X^*\) and \(E_{k,m}^{GG}T_X^*(\log D)\) carry the structure of holomorphic vector bundles over \(X\); they are called the \emph{Green–Griffiths bundle} and the \emph{logarithmic Green–Griffiths bundle}, respectively.

The classical result below states that every (logarithmic) jet differential twisted negatively by an ample line bundle yields an algebraic differential equation satisfied by all entire curves.

\begin{namedthm*}{Fundamental Vanishing Theorem for Entire Curves}
  [{\cite[Section 2]{Green-Griffiths1980}, \cite[Lemma 2.1]{Siu-Yeung1996MathA}, \cite[Corollary 7.9]{Demailly1997}}]

  Let \(X\) be a smooth complex projective variety and \(\mathscr{A}\) an ample line bundle on \(X\).

  \medskip
  \noindent
  \textbf{(1) Compact case.}
  If
  \[
  0 \neq \mathscr{P} \in H^0\bigl( X,\; E_{k,m}^{GG}T_X^* \otimes \mathscr{A}^{-1} \bigr),
  \]
  then every nonconstant holomorphic curve \(f \colon \mathbb{C} \to X\) satisfies
  \[
  \mathscr{P}\bigl(f',\dots,f^{(k)}\bigr) \equiv 0.
  \]

  \medskip
  \noindent
  \textbf{(2) Logarithmic case.}
  Let \(D \subset X\) be a normal crossing divisor. If
  \[
  0 \neq \mathscr{P} \in H^0\bigl( X,\; E_{k,m}^{GG}T_X^*(\log D) \otimes \mathscr{A}^{-1} \bigr),
  \]
  then every nonconstant holomorphic curve \(f \colon \mathbb{C} \to X \setminus D\) satisfies
  \[
  \mathscr{P}\bigl(f',\dots,f^{(k)}\bigr) \equiv 0.
  \]
\end{namedthm*}

\subsection{Invariant Jet Differentials}
\label{subsection:Invariant Jet}

Geometrically, only the image of an entire curve in \(X\) matters, not its specific parametrization. This motivates the notion of \emph{invariant jets differentials}, introduced by Demailly \cite{Demailly1997} in the compact case and later extended to the logarithmic setting by Dethloff--Lu \cite{Dethloff-Lu2001}.

Let \(\mathbb{G}_k\) be the group of germs of \(k\)-jets of biholomorphisms of \((\mathbb{C},0)\):
\begin{equation}\label{reparameterization group G_k}
\mathbb{G}_k \coloneqq 
\Bigl\{\varphi \colon t \longmapsto \sum_{j=1}^{k} c_j t^j \;\Bigm|\; 
c_1 \in \mathbb{C}^*,\; c_j \in \mathbb{C}\;(j \geqslant 2) \Bigr\},
\end{equation}
where the group operation is composition modulo \(t^{k+1}\). The group \(\mathbb{G}_k\) acts on \(J_k(X)\) by reparametrization on the right:
\[
\varphi \cdot j_k(f) \coloneqq j_k(f \circ \varphi).
\]

A (logarithmic) $k$-jet differential \(P\) of weight degree $m$ over an open set \(U \subset X\) is called \emph{invariant under \(\mathbb{G}_k\)} if for every \(\varphi \in \mathbb{G}_k\) it satisfies
\begin{equation}\label{def-invariant-jet-diff}
P\bigl( (f \circ \varphi)', \dots, (f \circ \varphi)^{(k)} \bigr) 
= (\varphi')^{m} \, P\bigl(f',\dots,f^{(k)}\bigr).
\end{equation}
This invariance under reparametrization is the crucial property that distinguishes (Demailly--Semple) invariant jet differentials from Green--Griffiths jet differentials.

The subbundle
\[
E_{k,m}T_X^* \subset E_{k,m}^{GG}T_X^* \qquad 
\bigl(\text{resp. } E_{k,m}T_X^*(\log D) \subset E_{k,m}^{GG}T_X^*(\log D)\bigr)
\]
whose fibres consist of all such invariant polynomials is called the \emph{Demailly--Semple bundle} (resp. the \emph{logarithmic Demailly--Semple bundle}). A key advantage of these bundles, especially in dimension two, is their favourable positivity properties. For example, they allow one to apply Riemann--Roch techniques at lower degrees, which in turn yields the existence of negatively twisted global sections that are crucial for establishing hyperbolicity.

Throughout this article we work exclusively with jets of order \(k = 2\). When \(\dim_{\mathbb{C}} X = 2\), the bundle of invariant (logarithmic) 2‑jet differentials admits a filtration whose associated graded bundle is given by~\cite[Section~12]{Demailly1997}:
\begin{align}
    \operatorname{Gr}^{\bullet} E_{2,m} T_X^{*}
    &= \bigoplus_{0 \leqslant j \leqslant \lfloor m/3 \rfloor}
       S^{m-3j} T_X^{*} \otimes K_X^{j}, \label{filtration-of-E2m, cpt}
    \\[2mm]
    \operatorname{Gr}^{\bullet} E_{2,m} T_X^{*}(\log D)
    &= \bigoplus_{0 \leqslant j \leqslant \lfloor m/3 \rfloor}
       S^{m-3j} T_X^{*}(\log D) \otimes \overline{K}_X^{\,j}, \label{filtration-of-E2m, log}
\end{align}
where \(K_X\) is the canonical bundle and \(\overline{K}_X := K_X \otimes \mathcal{O}_X(D)\) is the logarithmic canonical bundle. 

This grading reflects the fact that an invariant 2‑jet differential \(P(f',f'')\) (or its logarithmic counterpart outside \(D\)) can be written locally as
\begin{equation}\label{local-expansion}
P(f',f'')
 = \sum_{0 \leqslant j \leqslant \lfloor m/3 \rfloor}\;
   \sum_{\substack{\alpha_1+\alpha_2 = m-3j \\ \alpha_1,\alpha_2 \geqslant 0}}
   R_{\alpha_1,\alpha_2,j}(f_1,f_2)\,
   (f_1')^{\alpha_1}(f_2')^{\alpha_2}
   \bigl(f_1' f_2'' - f_2' f_1''\bigr)^{j},
\end{equation}
where \(f = (f_1,f_2)\) are local coordinates of the holomorphic germ, and the coefficients \(R_{\alpha_1,\alpha_2,j}\) are holomorphic functions.  
 For higher-order invariant jet differentials, analogous local expressions exist, though they are considerably more complicated, and the bases are related by intricate syzygies; see \cite{Merker2008b}.

\subsection{Demailly-Semple Tower}

We now recall the construction of the Demailly–Semple tower \cite[Section 6]{Demailly1997}. 

A \emph{directed manifold} is a pair \((X,V)\) where \(X\) is a complex manifold and \(V \subset T_X\) is a holomorphic sub‑bundle of the tangent bundle.

In the compact two‑dimensional setting, we begin with the directed manifold \((X_0,V_0) \coloneqq (X,T_X)\), where \(X\) is a generic algebraic surface in \(\mathbb{P}^3\). Set
\[
X_1 \coloneqq \mathbb{P}(V_0),
\]
equipped with the natural projection \(\pi_{1,0} \colon X_1 \to X_0\). The sub‑bundle \(V_1 \subset T_{X_1}\) is defined fibre‑wise by
\begin{equation}\label{def-V1}
V_{1,(x,[v])} \coloneqq 
\Bigl\{\, \xi \in T_{X_1,(x,[v])} \;\Bigm|\; 
(\pi_{1,0})_*\xi \in \mathbb{C} \cdot v \,\Bigr\},
\end{equation}
where \(\mathbb{C}\cdot v \subset V_{0,x}=T_{X,x}\). This yields the first level \((X_1,V_1)\) of the Demailly–Semple tower over \((X_0,V_0)\).

Iterating the same construction gives the second level, which fits into the tower
\[
(X_2,V_2) \longrightarrow (X_1,V_1) \longrightarrow (X_0,V_0),
\]
with projections
\(\pi_{2,1} \colon X_2 \to X_1\) and
\(\pi_{2,0} \coloneqq \pi_{1,0} \circ \pi_{2,1} \colon X_2 \to X_0\).
  
Let \(\mathcal{O}_{X_{1}}(-1)\) denote the tautological line bundle of \(X_{1} = \mathbb{P}(V_0)\); thus
\[
\mathcal{O}_{X_{1}}(-1)_{(x,[v])} = \mathbb{C} \cdot v .
\]

From the definition of \(V_1\), we obtain the short exact sequence
\begin{equation}\label{SES:V_1}
    0 \longrightarrow T_{X_{1}/X_{0}}
    \longrightarrow V_{1}
    \xrightarrow{(\pi_{1,0})_{*}}
    \mathcal{O}_{X_{1}}(-1) \longrightarrow 0,
\end{equation}
where \(T_{X_{1}/X_{0}}\) is the relative tangent bundle of the fibration
\(\pi_{1,0}\colon X_{1} \to X_{0}\). Over the fibre \(\mathbb{P}(V_{0,x})\) this construction restricts to the usual Euler sequence, giving the relative Euler sequence
\begin{equation}\label{ses:relative Euler sequence}
    0 \longrightarrow \mathcal{O}_{X_{1}}
    \longrightarrow \pi_{1,0}^{*}V_{0} \otimes \mathcal{O}_{X_{1}}(1)
    \longrightarrow T_{X_{1}/X_{0}}
    \longrightarrow 0.
\end{equation}

At the second level there is a canonical inclusion
\(\mathcal{O}_{X_{2}}(-1) \hookrightarrow \pi_{2,1}^{*}V_{1}\).
Composing it with \(\pi_{2,1}^{*} (\pi_{1,0})_{*}\) yields a morphism of line bundles
\begin{equation}\label{morphism-OX2}
    \mathcal{O}_{X_{2}}(-1)
    \;\longrightarrow\;
    \pi_{2,1}^{*}\mathcal{O}_{X_{1}}(-1),
\end{equation}
whose zero divisor is the hyperplane sub‑bundle
\begin{equation}\label{equ:def of Gamma_2}
    \Gamma_2 \coloneqq \mathbb{P}\bigl(T_{X_{1}/X_{0}}\bigr)
    \subset \mathbb{P}(V_1) = X_{2}.
\end{equation}
Consequently,
\begin{equation}\label{gamma_2=O(-1, 1)}
    \mathcal{O}_{X_{2}}(\Gamma_{2})
    \;\cong\;
    \pi_{2,1}^{*}\mathcal{O}_{X_{1}}(-1) \otimes \mathcal{O}_{X_{2}}(1).
\end{equation}

In the logarithmic case, following Dethloff--Lu \cite{Dethloff-Lu2001}, we consider the notion of a \emph{log‑directed manifold}, a triple \((X,D,V)\) where \(X\) is a complex manifold, \(D\) is a simple normal crossing divisor, and \(V\) is a sub‑bundle of the logarithmic tangent bundle \(T_X(-\log D)\).

We start with the initial log‑directed manifold
\[
(X_0,D_0,V_0) \coloneqq \bigl(X,\;D,\;T_X(-\log D)\bigr).
\]
Set \(X_1 \coloneqq \mathbb{P}(V_0)\) with the natural projection \(\pi_1\colon X_1 \to X_0\),
and define \(D_1 \coloneqq \pi_1^* D_0\), so that \(\pi_1\) becomes a log‑morphism.
The sub‑bundle \(V_1 \subset T_{X_1}(-\log D_1)\) is then given fibre‑wise by
\[
V_{1,(x,[v])} \coloneqq 
\bigl\{\xi \in T_{X_1,(x,[v])}(-\log D_1) \;\big|\;
(\pi_1)_{*}\xi \in \mathbb{C} \cdot v \bigr\},
\]
where \(\mathbb{C}\cdot v \subset V_{0,x}=T_{X,x}(-\log D_0)\). 
Again, let \(\mathcal{O}_{X_1}(-1)\) be the tautological line bundle of \(\mathbb{P}(V_0)\), so that
\(\mathcal{O}_{X_1}(-1)_{(x,[v])} = \mathbb{C}\cdot v\).

Iterating the construction produces the \emph{logarithmic Demailly--Semple tower}
\[
(X_2,D_2,V_2) \longrightarrow (X_1,D_1,V_1) \longrightarrow (X_0,D_0,V_0),
\]
together with projections \(\pi_{k,0} \colon X_k \to X_0\) for \(k = 1,2\).

The same short exact sequences \eqref{SES:V_1} and \eqref{ses:relative Euler sequence} remain valid in the logarithmic setting. Moreover, with a slight abuse of notation, the divisor \(\Gamma_2\) on the second level \(X_2\) of the logarithmic tower can be defined exactly as in \eqref{equ:def of Gamma_2}.

From the construction, the direct image sheaf on $X$ (resp. $(X,D)$)
\[
(\pi_{2,0})_*\mathcal{O}_{X_2}(m),
\qquad
\mathcal{O}_{X_2}(m) \coloneqq \mathcal{O}_{X_2}(1)^{\otimes m},
\]
coincides precisely with the bundle \(E_{2,m}T_X^*\) (resp. \(E_{2,m}T_X^*(\log D)\)). Consequently, we obtain the following direct‑image formula.

\begin{namedthm*}{Direct image formula}
  [cf.~{\cite[Theorem 6.8]{Demailly1997}}, {\cite[Proposition 3.9]{Dethloff-Lu2001}}]
  
  Let \(\mathcal{A}\) be an ample line bundle on \(X\).
  \begin{enumerate}
    \item If \(X_2\) denotes the second level of the Demailly–Semple tower for \((X,T_X)\), then
      \begin{equation}\label{direct image formula, cpt}
      H^0\bigl(X,\;E_{2,m}T_{X}^{*}\otimes\mathcal{A}^{-1}\bigr)
      \;\cong\;
      H^0\bigl(X_2,\;\mathcal{O}_{X_2}(m)\otimes\pi_{2,0}^{*}\mathcal{A}^{-1}\bigr).
      \end{equation}
      
    \item If \(X_2\) denotes the second level of the logarithmic Demailly–Semple tower for \((X,D,T_X(-\log D))\), then
      \[
      H^0\bigl(X,\;E_{2,m}T_{X}^{*}(\log D)\otimes\mathcal{A}^{-1}\bigr)
      \;\cong\;
      H^0\bigl(X_2,\;\mathcal{O}_{X_2}(m)\otimes\pi_{2,0}^{*}\mathcal{A}^{-1}\bigr).
      \]
  \end{enumerate}
\end{namedthm*}

\subsection{Existence of Negatively Twisted Invariant \texorpdfstring{$2$}{2}-Jet Differentials}\label{subsec:existence of 2-jet differ.}

Classical vanishing theorems in algebraic geometry imply that, for any smooth hypersurface in \(\mathbb{P}^n\) with \(n\geqslant 3\), there exist no nonzero negatively twisted symmetric differentials (i.e. the case \(k=1\)), regardless of the degree~\cite{Sakai1979}. This forces one to consider jet differentials of order \(k \geqslant 2\).

For smooth hypersurfaces of general type in \(\mathbb{P}^n\) (i.e., degree \( \geqslant n+2\)), the existence of such negatively twisted jet differentials has been established for large orders \(k \gg 1\) by Merker~\cite{Merker2015} (see also \cite{Demailly2011, Cadorel2019}). However, for \(k \geqslant 3\) the control of the base locus remains a serious technical obstacle, rendering these higher-order differentials unsuitable for sharpening the hyperbolicity degree bounds in \(\mathbb{P}^3\).

In this paper we therefore focus on the case \(k=2\), which provides a viable intermediate regime — the first-order theory being obstructed, and higher orders being computationally inaccessible. We now specialize to surfaces in \(\mathbb{P}^3\).

\begin{pro}[{\cite[Corollary 13.9]{Demailly1997}}]\label{existence of 2 jets, compact case}
Let \(X\) be a smooth projective surface of general type and \(A\) an ample line bundle on \(X\). Then
\[
h^0\bigl(X,\;E_{2,m}T_X^*\otimes\mathcal{O}(-A)\bigr)
\;\geqslant\;
\frac{m^{4}}{648}\,(13c_1^{2}-9c_2)-O(m^{3}),
\]
where \(c_i\) are the Chern numbers of \(X\). \qed
\end{pro}

The proof of Proposition~\ref{existence of 2 jets, compact case} relies on a Riemann--Roch calculation based on the filtration~\eqref{filtration-of-E2m, cpt} and the following Bogomolov vanishing theorem.

\begin{thm}\label{thm:Bogomolov vanishing theorem}
Let \(X\) be a smooth projective surface. Assume that \(X\) is of general type (i.e. \(K_{X}\) is big) and that \(K_{X}\) is nef. Then
\[
H^0\bigl(X,\;S^{p}T_{X} \otimes \mathcal{O}_X(qK_{X})\bigr)=0,
\]
for all integers $p,q$  satisfying $p>2q$. 
\qed
\end{thm}

Bogomolov~\cite{Bogomolov1978} proved the original vanishing theorem under the assumption that the tangent bundle \(T_X\) is semistable. Enoki~\cite{Enoki1987} later showed that for a minimal Kähler space \(X\) with big canonical bundle \(K_X\), its tangent bundle \(T_X\) is \(K_{X}\)-semistable. Independently, Tsuji~\cite{Tsuji1988} established the \(K_{X}\)-semistability of \(T_X\) for smooth minimal algebraic varieties (here ``minimal'' means that \(K_X\) is nef). Combining these results yields a Bogomolov-type vanishing theorem for projective varieties whose canonical bundle \(K_X\) is both big and nef; see \cite[Theorem 14.1]{Demailly1997}.

\medskip

If \(X\) is a smooth surface in \(\mathbb{P}^3\) of degree \(d\), one can compute the Chern numbers:
\begin{equation}\label{Chern numbers, cpt case}
    c_1^{2}=d(d-4)^{2},\qquad c_2=d(d^{2}-4d+6).
\end{equation}
Thus the leading coefficient satisfies
\[
13c_1^{2}-9c_2 = d\bigl(4d^{2}-68d+154\bigr),
\]
which is positive for every \(d \geqslant 15\). Hence Proposition~\ref{existence of 2 jets, compact case} yields:

\begin{cor}[\cite{Demailly1997}]\label{cor:existence of 2 jets, compact case}
For a smooth surface \(X\subset\mathbb{P}^3\) of degree \(d \geqslant 15\) and any ample line bundle \(A\) on \(X\),
\[
H^0\bigl(X,\;E_{2,m}T_X^*\otimes\mathcal{O}(-A)\bigr) \neq 0
\]
for all sufficiently large \(m \gg 1\). \qed
\end{cor}

Similar statements hold in the logarithmic setting by essentially the same arguments (cf.~\cite{ElGoul2003}).

\begin{pro}[{\cite[Theorem 1.2.1]{ElGoul2003}}]\label{existence of 2 jets, logarithmic case}
Let \(X\) be a smooth projective surface and \(\mathcal{C} \subset X\) a simple normal crossing divisor such that \((X,\mathcal{C})\) is of log‑general type. Let \(A\) be an ample line bundle on \(X\). Then
\[
h^0\bigl(X,\;E_{2,m}T_X^*(\log \mathcal{C})\otimes\mathcal{O}(-A)\bigr)
\;\geqslant\;
\frac{m^{4}}{648}\,(13\bar{c}_1^{2}-9\bar{c}_2)-O(m^{3}),
\]
where \(\bar{c}_i\) are the logarithmic Chern numbers of \((X,\mathcal{C})\). \qed
\end{pro}

When \(X = \mathbb{P}^2\) and \(\mathcal{C}\subset\mathbb{P}^2\) is a smooth plane curve of degree \(d\),  one can compute the logarithmic Chern numbers
\begin{equation}\label{Chern numbers, log case}
    \bar{c}_1^{2}=(d-3)^{2},\qquad \bar{c}_2=d^{2}-3d+3.
\end{equation}
Hence the leading coefficient satisfies
\[
13\bar{c}_1^{2}-9\bar{c}_2 = 4d^{2}-51d+90 > 0
\qquad\text{for } d \geqslant 11.
\]

Consequently, we obtain

\begin{cor}
For a smooth algebraic curve \(\mathcal{C}\subset\mathbb{P}^2\) of degree \(d \geqslant 11\) and any ample line bundle \(A\) on \(\mathbb{P}^2\),
\[
H^0\bigl(\mathbb{P}^2,\;E_{2,m}T_{\mathbb{P}^2}^*(\log\mathcal{C})\otimes\mathcal{O}(-A)\bigr) \neq 0
\]
for all sufficiently large \(m \gg 1\). \qed
\end{cor}

\subsection{Nevanlinna Theory}\label{Nevanlinna theory notation}

We begin by recalling the standard notation of Nevanlinna theory.  
Let \(\mathcal{C} \subset \mathbb{P}^{2}\) be an algebraic curve and let \(f \colon \mathbb{C} \to \mathbb{P}^{2}\) be an entire curve with \(f(\mathbb{C}) \not\subset \mathcal{C}\). For \(k\in \mathbb{N}\cup \{\infty\}\), the \emph{\(k\)-truncated counting function} of \(f\) with respect to \(\mathcal{C}\) is defined as
\[
N_f^{[k]}(r,\mathcal{C}) \coloneqq \int_{1}^{r} 
\sum_{|z|<t} \min\bigl\{k,\;\operatorname{ord}_{z}f^{*}\mathcal{C}\bigr\}\,
\frac{\mathrm{d}t}{t}\qquad (r>1),
\]
which measures the frequency of intersections \(f(\mathbb{C}) \cap \mathcal{C}\).  
The \emph{order function} of \(f\) is given by
\[
T_{f}(r) \coloneqq \int_{1}^{r}\frac{\mathrm{d}t}{t}
                     \int_{|z|<t} f^{*}\omega_{\mathrm{FS}}\qquad (r>1),
\]
where \(\omega_{\mathrm{FS}}\) denotes the Fubini–Study metric on \(\mathbb{P}^{2}\). This function captures the average growth of the area of the holomorphic discs \(f(\{|z|<t\})\) with respect to \(\omega_{\mathrm{FS}}\) as \(t\) tends to infinity.

\medskip

The pursuit of a Second Main Theorem faces a fundamental obstacle: entire curves exhibit an inherent and wild flexibility in large-scale geometry, as evidenced by the exotic examples constructed in \cite{HX2021, Xie2024, WX2025, CFX2025}. Our strategy to overcome this difficulty is to constrain this flexibility by leveraging additional algebro-geometric structures, thereby extracting the underlying rigidity of entire curves. The algebraically nondegenerate branch of the proof relies on the following two results.

The first result can be viewed as a quantitative refinement of the Fundamental Vanishing Theorem for entire curves (cf. \cite[Section 2]{Green-Griffiths1980}, \cite[Lemma 1.2]{Siu-Yeung1996MathA}, and \cite[Corollary 7.9]{Demailly1997}).

\begin{thm}[{\cite[Theorem 3.1]{Huynh-Vu-Xie-2017}}]
\label{smt-form-logarithmic-diff-jet}
Let \(\mathcal{C} = \sum_{i=1}^{q} \mathcal{C}_{i}\) be a simple normal crossing divisor in \(\mathbb{P}^{2}\).  
For any logarithmic jet differential
\[
\omega \in H^{0}\!\bigl(\mathbb{P}^{2},\; 
        E_{k,m}^{GG}T_{\mathbb{P}^{2}}^{*}(\log\mathcal{C}) \otimes \mathcal{O}_{\mathbb{P}^{2}}(-t)\bigr)
\qquad (k,m,t \geqslant 1),
\]
and any entire holomorphic curve \(f \colon \mathbb{C} \to \mathbb{P}^{2}\) with \(f^{*}\omega \not\equiv 0\), the following Second Main Theorem type estimate holds:
\begin{equation}\label{SMT2}
T_{f}(r) \leqslant \frac{m}{t} \sum_{i=1}^{q} N_{f}^{[1]}(r,\mathcal{C}_{i}) 
                + O\bigl(\log T_{f}(r) + \log r\bigr) \quad\parallel.
\end{equation}
\end{thm}
Here, the symbol ``\(\parallel\)'' indicates that the inequality~\eqref{SMT2} holds for all \(r>1\) outside a set of finite Lebesgue measure. This notation will be used consistently throughout this paper.

The second key ingredient is a theorem of McQuillan (see also \cite{Rousseau2010}), which plays a crucial role in the proof of Theorem~\ref{SMT for curves of d geqslant 12}.

\begin{thm}[McQuillan \cite{Mcquillan1998}]
\label{McQuillan's result}
Let \(\mathcal{C}=\sum_{i=1}^{q}\mathcal{C}_{i}\) be a simple normal crossing divisor in \(\mathbb{P}^{2}\), where each \(\mathcal{C}_{i}\) is a smooth algebraic curve of degree \(d_{i}\) and \(d=\sum_{i=1}^{q}d_{i}\geqslant 4\).  
Let \(f \colon \mathbb{C} \to \mathbb{P}^{2}\) be an algebraically nondegenerate entire curve.  
If the lift \(f_{[1]} \colon \mathbb{C} \to X_{1}= \mathbb{P}\,T_{\mathbb{P}^{2}}(-\log\mathcal{C})\) is algebraically degenerate, then
\[
(d-3)\,T_{f}(r) \leqslant N^{[1]}_{f}(r,\mathcal{C}) + o\bigl(T_{f}(r)\bigr) \quad\parallel.
\]
\end{thm}

To prove the assertion of Theorem~\ref{SMT for curves of d geqslant 12}
for very generic plane curves $\mathcal{C}$, we use the following lemma
to treat algebraically degenerate entire curves
$f\colon\mathbb{C}\to\mathbb{P}^2$ whose images are not contained in
$\mathcal{C}$. The lemma follows from Chen's logarithmic genus--degree
inequality~\cite[Theorem~1.7]{Chen2004}
(see also~\cite[Corollary~4]{PacienzaRousseau2007}),
together with one-dimensional Nevanlinna theory.

\begin{lem}
\label{lem:SMT-algebraically-degenerate}
Let $\mathcal C\subset\mathbb P^2$ be a very generic smooth
plane curve of degree $d\geqslant5$.
Then every nonconstant algebraically degenerate entire curve
$f:\mathbb C\to\mathbb P^2$ with
$f(\mathbb C)\not\subset\mathcal C$ satisfies
\[
(d-4)T_f(r)
\leqslant
N_f^{[1]}(r,\mathcal C)+o\bigl(T_f(r)\bigr)
\quad\parallel.
\]
\end{lem}

\begin{proof}
Let
\[
B=\overline{f(\mathbb C)}^{\,\mathrm{Zar}}
\]
with its reduced structure.
Since $f$ is nonconstant and algebraically degenerate, $B$
is an irreducible algebraic curve.
Let
\[
\nu:X\longrightarrow B\hookrightarrow\mathbb P^2
\]
be its normalization.
The map $f$ lifts to a nonconstant holomorphic map
$h:\mathbb C\to X$ such that $f=\nu\circ h$.
Set
\[
m=\deg B,\qquad
g=g(X),\qquad
D=(\nu^*\mathcal C)_{\mathrm{red}},\qquad
s=\deg D.
\]
The divisor $D$ is well defined because $B\not\subset\mathcal C$.
By Chen's logarithmic genus--degree inequality
\cite[Theorem~1.7]{Chen2004} (see also \cite[Corollary~4]{PacienzaRousseau2007}),
\begin{equation}
\label{eq:algdeg-Chen}
2g-2+s\geqslant(d-4)m.
\end{equation}

Choose a smooth positive $(1,1)$-form $\omega$ on $X$
normalized by $\int_X\omega=1$, and put
\[
T_{h,\omega}(r)
=\int_1^r\frac{d\rho}{\rho}
  \int_{|z|<\rho}h^*\omega.
\]
For every line bundle $L$ on $X$, equipped with a smooth
Hermitian metric, one has
\begin{equation}
\label{eq:algdeg-degree-comparison}
T_{h,L}(r)=\deg(L)T_{h,\omega}(r)+O(1).
\end{equation}
Indeed, its curvature form differs from
$\deg(L)\omega$ by $dd^c$ of a smooth function on the
compact curve $X$, so the corresponding characteristics
differ by a bounded term.
Consequently, functoriality gives
\begin{equation}
\label{eq:algdeg-characteristic}
T_f(r)=m\,T_{h,\omega}(r)+O(1).
\end{equation}
Moreover, the reduced pullbacks of $\mathcal C$ under $f$
and of $D$ under $h$ coincide, and hence
\begin{equation}
\label{eq:algdeg-counting}
N_f^{[1]}(r,\mathcal C)=N_h^{[1]}(r,D).
\end{equation}

We next establish the truncated Second Main Theorem on $X$
in the form
\begin{equation}
\label{eq:algdeg-SMT-on-X}
(2g-2+s)T_{h,\omega}(r)
\leqslant
N_h^{[1]}(r,D)+o\bigl(T_{h,\omega}(r)\bigr)
\quad\parallel.
\end{equation}
Choose a finite morphism
$\pi:X\to\mathbb P^1$ of degree $k$.
Let $Q$ be a reduced divisor on $\mathbb P^1$ whose support
contains every branch value of $\pi$ and every point of
$\pi(\operatorname{Supp}D)$.
Adding further points if necessary, we may assume
$q:=\deg Q\geqslant3$.
Set
\[
E=(\pi^*Q)_{\mathrm{red}},
\qquad e=\deg E.
\]
Then $D\leqslant E$.
Since every ramification point of $\pi$ lies over
$\operatorname{Supp}Q$, the Riemann--Hurwitz formula gives
\begin{equation}
\label{eq:algdeg-Riemann-Hurwitz}
e=kq-\deg R_\pi
  =kq-(2g-2+2k)
  =k(q-2)-(2g-2),
\end{equation}
where $R_\pi$ is the ramification divisor of $\pi$.

The composition $u=\pi\circ h$ is a nonconstant meromorphic
function on $\mathbb C$.
The classical truncated Second Main Theorem gives
\begin{equation}
\label{eq:algdeg-SMT-on-P1}
(q-2)T_u(r)
\leqslant
N_u^{[1]}(r,Q)+o\bigl(T_u(r)\bigr)
\quad\parallel.
\end{equation}

By \eqref{eq:algdeg-degree-comparison} and functoriality,
\[
T_u(r)=k\,T_{h,\omega}(r)+O(1),
\qquad
N_u^{[1]}(r,Q)=N_h^{[1]}(r,E).
\]
The second equality holds because truncation removes all
multiplicities arising from $\pi$ and $h$.
Since $u$ is nonconstant, $T_{h,\omega}(r)\to\infty$.
Applying the First Main Theorem to the effective divisor
$E-D$ and using \eqref{eq:algdeg-degree-comparison}, we obtain
\[
N_h^{[1]}(r,E-D)
\leqslant N_h(r,E-D)
\leqslant(e-s)T_{h,\omega}(r)+O(1).
\]
Therefore \eqref{eq:algdeg-SMT-on-P1} yields
\[
\begin{aligned}
k(q-2)T_{h,\omega}(r)
&\leqslant N_h^{[1]}(r,E)+o\bigl(T_{h,\omega}(r)\bigr)\\
&=N_h^{[1]}(r,D)+N_h^{[1]}(r,E-D)
  +o\bigl(T_{h,\omega}(r)\bigr)\\
&\leqslant N_h^{[1]}(r,D)+(e-s)T_{h,\omega}(r)
  +o\bigl(T_{h,\omega}(r)\bigr)
\quad\parallel.
\end{aligned}
\]
Using \eqref{eq:algdeg-Riemann-Hurwitz} and rearranging
proves \eqref{eq:algdeg-SMT-on-X}.

Finally, combining
\eqref{eq:algdeg-Chen},
\eqref{eq:algdeg-characteristic},
\eqref{eq:algdeg-counting}, and
\eqref{eq:algdeg-SMT-on-X}, we obtain
\[
\begin{aligned}
(d-4)T_f(r)
&=(d-4)m\,T_{h,\omega}(r)+O(1)\\
&\leqslant(2g-2+s)T_{h,\omega}(r)+O(1)\\
&\leqslant N_h^{[1]}(r,D)+o\bigl(T_{h,\omega}(r)\bigr)\\
&=N_f^{[1]}(r,\mathcal C)+o\bigl(T_f(r)\bigr)
\quad\parallel.
\end{aligned}
\]
\end{proof}

\medskip

\section{\bf Proof of Key Vanishing Lemma \ref{lem-1.2}}
\label{sect:proof of lemma 2}
In this section, we prove Lemma~\ref{lem-1.2} by executing the three-step strategy outlined in the introduction.

\subsection{Local Description of Invariant Logarithmic Jet Differentials}
Now, working in affine coordinates $(x,y)$ on $\mathbb{P}^2$ with
\[
[T:X:Y] \in \mathbb{P}^2,\qquad
[1:x:y] := [1:\tfrac{X}{T}:\tfrac{Y}{T}],
\]
consider an algebraic curve $\mathcal{C} \subset \mathbb{P}^2$ defined by
\[
R(x,y) = 0,
\]
where $R \in \mathbb{C}[x,y]$ is a polynomial of degree
\[
\deg R = d \qquad (d \in \mathbb{N}).
\]
We take $R$ to be a perturbation of the Fermat curve:
\[
R(x,y) = 1^{d} + x^{d} + y^{d} + \cdots,
\]
with generic coefficients, so that in particular $\mathcal{C}$ is smooth at every point of $\mathbb{P}^2$. Consequently, in the affine $(x,y)$-chart we have
\[
\emptyset = \{R=0\} \cap \{R_x=0\} \cap \{R_y=0\}.
\]

On the open set $\{R_y \neq 0\}$ (understood to be taken within the affine chart $\mathbb{C}^2$), the implicit function theorem allows us to solve locally for $y$ as a function of $x$. On $\{R_x \neq 0\}$, similarly, $x$ can be solved locally as a function of $y$.

In $\{R_y \neq 0\}$, where $\{R=0\}$ is the divisor that entire curves $f: \mathbb{C}\rightarrow \mathbb{P}^2\setminus\mathcal{C}$ must avoid, the fibers of the logarithmic Green–Griffiths $1$-jet bundle are polynomially generated by
\[
\big\langle x',\; (\log R)' \big\rangle
= \Big\langle x',\; \frac{R'}{R} \Big\rangle,
\]
with
\[
R' := x' R_x + y' R_y.
\]

Likewise, on $\{R_x \neq 0\}$, the logarithmic $1$-jets are generated by
\[
\big\langle (\log R)',\; y' \big\rangle
= \Big\langle \frac{R'}{R},\; y' \Big\rangle.
\]

Concerning logarithmic $2$-jets, the generators in $\{R_y \neq 0\}$ are:
\[
\big\langle x',\; x'',\; (\log R)',\; (\log R)'' \big\rangle
= \Big\langle x',\; x'',\; \frac{R'}{R},\; \frac{R''}{R} - \frac{R' R'}{R R} \Big\rangle,
\]
where
\[
R'' := x'' R_x + y'' R_y + x'x' R_{xx} + 2x'y' R_{xy} + y'y' R_{yy}.
\]

Similarly, the generators in $\{R_x \neq 0\}$ are:
\[
\big\langle (\log R)',\; (\log R)'',\; y',\; y'' \big\rangle
= \Big\langle \frac{R'}{R},\; \frac{R''}{R} - \frac{R' R'}{R R},\; y',\; y'' \Big\rangle.
\]

For the invariant logarithmic $2$-jet bundles, instead of plain $2$-jets, one considers the logarithmic Wronskian. In $\{R_y \neq 0\}$ it is given by
\[
\Delta_{xR}'''
\,:=\,
\begin{vmatrix}
x' & (\log\,R)'
\\
x'' & (\log\,R)''
\end{vmatrix}
\,=\,
\def\arraystretch{1.25}
\left\vert\!
\begin{array}{ll}
x' & \frac{R'}{R}
\\
x'' & \frac{R''}{R}-\frac{{R'}^2}{R^2}
\end{array}
\!\right\vert,
\]
and in $\{R_x \neq 0\}$ by
\[
\Delta_{Ry}'''
\,:=\,
\begin{vmatrix}
(\log\,R)' & y'
\\
(\log\,R)'' & y''
\end{vmatrix}
\,=\,
\def\arraystretch{1.25}
\left\vert\!
\begin{array}{ll}
\frac{R'}{R} & y'
\\
\frac{R''}{R}-\frac{{R'}^2}{R^2} & y''
\end{array}
\!\right\vert.
\]

The transition from $\{R_x \neq 0\}$ to $\{R_y \neq 0\}$ consists in solving for $y'$ and $y''$ from the relations obtained by differentiating the function $R(x(\zeta),y(\zeta))$ with respect to $\zeta$ once and twice:
\begin{equation}\label{relations of log jet}
\begin{aligned}
0 &= -R' + x' R_x + y' R_y,\\
0 &= -R'' + x'' R_x + y'' R_y + x'x' R_{xx} + 2x'y' R_{xy} + y'y' R_{yy}.
\end{aligned}
\end{equation}
Solving the first equation yields
\begin{equation}\label{log y'}
    y'=-x'\frac{R_{x}}{R_{y}} +\frac{R'}{R_{y}}.
\end{equation}
Substituting this into the second equation and solving for $y''$ gives
\[
\aligned
y''
&
\,=\,
-x'' \frac{R_{x}}{R_{y}} +\frac{R''}{R_{y}}
\\
&
\ \ \ \ \
-\ x'x'\ \frac{R_{x}^{2} \ R_{yy}}{R_{y}^{3}} +2\ x'\ R'\ \frac{R_{x} \ R_{yy}}{R_{y}^{3}} -R'R'\ \frac{R_{yy}}{R_{y}^{3}} +2\ x'x'\ \frac{R_{x} \ R_{xy}}{R_{y}^{2}} -2\ x'R'\ \frac{R_{xy}}{R_{y}^{2}} -x'x'\ \frac{R_{xx}}{R_{y}}.
\endaligned
\]
The transition of a jet differential is then obtained by replacing $y'$ and $y''$ with these expressions.

In particular, the logarithmic Wronskian $\Delta_{Ry}'''$ transforms under the change of chart as follows:
\begin{equation}\label{log Wronskian Ry}
\aligned
\Delta_{Ry}'''
&
\,=\,
\frac{R_x}{R_y}\,
\Delta_{xR}'''
+
\frac{R}{R_y^3}
(
{R_y^2}
-
{R_{yy}}{R}
)
\Big(\frac{R'}{R}\Big)^3
+
\frac{2 R}{R_y^3}
(
{R_x\,R_{yy}}
-
{R_{xy}\,R_y}
)
x' \Big(\frac{R'}{R}\Big)^2 
\\
&
\ \ \ \ \ \ \ \ \ \ \ \ \ \ \ \ \ \ \ \ \ \ \ \ \ \ \ \ \ \ \ \ \ \ \
\ \ \ \ \ \ \ \ \ \ \ \ \
+
\frac{1}{R_y^3}
(
-
{R_x^2\,R_{yy}}
+
2\,{R_x\,R_{xy}\,R_y}
-
{R_{xx}\,R_y^2}
)
{x'}^2 \frac{R'}{R}
.
\endaligned
\end{equation}

Now consider a global section of the invariant logarithmic $2$-jet differential bundle $E_{2,m}T^*_{\mathbb{P}^2}(\log\mathcal{C})$ of weighted degree $m$. By the GAGA principle, such a section is algebraic, and it is a basic fact of algebraic geometry that on the open set $\{R_y \neq 0\}$ it takes the general form
\begin{equation}\label{J_xR of weighted degree m}
J_{xR} \coloneqq \sum_{\substack{0 \leqslant k \leqslant \lfloor m/3 \rfloor \\ 0 \leqslant j \leqslant m-3k}} 
\frac{F_{j,k}}{R_y^{f_{j,k}}}\, (x')^{m-3k-j}\Bigl(\frac{R'}{R}\Bigr)^{\!j}\,
\bigl(\Delta_{xR}'''\bigr)^k,
\end{equation}
where $F_{j,k} \in \mathbb{C}[x,y]$ are polynomials and $f_{j,k} \in \mathbb{N}$ are nonnegative integers indicating the admissible powers of $R_y$ in the denominator (division by $R_y$ being permissible on $\{R_y \neq 0\}$).

\begin{Question}
{\sl What are the maximal possible values of $f_{j,k} \in \mathbb{N}$?}
\end{Question}

The following proposition answers this question by showing that any global section of the invariant logarithmic $2$-jet differential bundle must take a prescribed finite form, with undetermined polynomial coefficients.

\begin{pro}\label{prop:4.2}
Let $\widehat{R}(T,X,Y)\in \mathbb{C}[T, X, Y]$ be a homogeneous polynomial defining a smooth curve $\mathcal{C} \subset \mathbb{P}^2$, and assume that the divisors in $\mathbb{P}^2$ given by $\widehat{R}$ and its partial derivatives $\widehat{R}_T, \widehat{R}_X, \widehat{R}_Y$ are in general position (i.e., any intersection of them has the expected codimension).  
Dehomogenizing by setting $T = 1$ yields an affine polynomial $R(x,y) \in \mathbb{C}[x,y]$ such that $\mathcal{C} \cap \{T \neq 0\} = \{R=0\}$.  
For any weighted order $m$, every global section of the invariant logarithmic $2$-jet differential bundle $E_{2,m}T^*_{\mathbb{P}^2}(\log\mathcal{C})$ on $\mathbb{P}^2$, when restricted to the affine open set $\{R_y \neq 0\}\subset \{T\neq 0\}$, admits the general expression
\[
J_{xR} = \sum_{\substack{0 \leqslant k \leqslant \lfloor m/3 \rfloor \\ 0 \leqslant j \leqslant m-3k}} 
\frac{F_{j,k}}{R_y^{\,m-2k}}\, (x')^{m-3k-j}\Bigl(\frac{R'}{R}\Bigr)^{\!j}\,
\bigl(\Delta_{xR}'''\bigr)^k,
\]
where $F_{j,k} \in \mathbb{C}[x,y]$ are polynomials.
\end{pro}

\begin{proof}
In the standard affine coordinates \((x,y) = (X/T, Y/T)\) on \(\{T \neq 0\} = \mathbb{C}^2\), the invariant logarithmic $2$-jet differential is algebraic on \(\{R_x \neq 0\}\). By the same reasoning as in~\eqref{J_xR of weighted degree m}, on this open set it admits an expression of the form
\begin{equation}\label{J_Ry of weighted degree m}
J_{Ry} \coloneqq \sum_{\substack{0 \leqslant q \leqslant \lfloor m/3 \rfloor \\ 0 \leqslant p \leqslant m-3q}} 
\overline{F}_{p,q}\, \Bigl(\frac{R'}{R}\Bigr)^{p}\,(y')^{m-3q-p}\,
\bigl(\Delta_{Ry}'''\bigr)^q,
\end{equation}
where the coefficients \(\overline{F}_{p,q}\) are rational functions whose denominators may involve only powers of \(R_x\).

We now examine the compatibility of the two expressions~\eqref{J_xR of weighted degree m} and~\eqref{J_Ry of weighted degree m} for the jet differential. Substituting the transition formulas~\eqref{log y'} and~\eqref{log Wronskian Ry} into~\eqref{J_Ry of weighted degree m} and comparing with~\eqref{J_xR of weighted degree m}, we find that the coefficients in~\eqref{J_xR of weighted degree m} become rational expressions whose denominators are products of powers of \(R_x\) (which must ultimately cancel) and powers of \(R_y\). The exponent of \(R_y\) in the denominator arises from two sources:
\begin{itemize}
    \item Each of the \(m-3q-p\) factors of \(y'\) contributes at most one factor of \(R_y^{-1}\) via \eqref{log y'}.
    \item Each of the \(q\) factors of the Wronskian \(\Delta_{Ry}'''\) contributes at most one factor of \(R_y^{-1}\) from the leading term in \eqref{log Wronskian Ry}, and at most three factors of \(R_y^{-1}\) from the remaining cubic terms.
\end{itemize}

Moreover, the term \((x')^{m-3k-j}\bigl(\frac{R'}{R}\bigr)^{\!j}\bigl(\Delta_{xR}'''\bigr)^k\) in expression~\eqref{J_xR of weighted degree m} arises from expanding terms of the form
\begin{equation}\label{terms in J_Ry with q>=k, p<=j}
    \Bigl(\frac{R'}{R}\Bigr)^{p}\,(y')^{m-3q-p}\,\bigl(\Delta_{Ry}'''\bigr)^q \quad\text{with } q \geqslant k.
\end{equation}
When expanding such a term, \(\bigl(\Delta_{Ry}'''\bigr)^q\) contributes a factor of \(\bigl(\Delta_{xR}'''\bigr)^k\) together with \(k\) powers of \(R_y^{-1}\) from the leading term in \eqref{log Wronskian Ry} and an additional \(q-k\) powers of \(R_y^{-3}\) from the remaining cubic terms. The factor \((y')^{m-3q-p}\) contributes \(m-3q-p\) powers of \(R_y^{-1}\) via \eqref{log y'}. Hence, the total exponent of \(R_y\) in the denominator of the coefficient multiplying \((x')^{m-3k-j}\bigl(\frac{R'}{R}\bigr)^{\!j}\bigl(\Delta_{xR}'''\bigr)^k\) is at most
\[
k + 3(q-k) + (m-3q-p) = m - 2k - p \leqslant m - 2k.
\]

This upper bound \(m-2k\) can be achieved by expanding the term in \eqref{terms in J_Ry with q>=k, p<=j} corresponding to \(q = k\) and \(p = 0\), namely \((y')^{m-3k}\,\bigl(\Delta_{Ry}'''\bigr)^k\).

By our assumption on \(\widehat{R}\), the divisors defined by \(R\) and its partial derivatives \(R_x, R_y\) are in general position. Under this genericity condition, after cancelling common factors between numerator and denominator, the maximal possible pole order \(f_{j,k}\) of the coefficients in \eqref{J_xR of weighted degree m} with respect to \(R_y\) satisfies
\[
f_{j,k} \leqslant m-2k.
\]
Moreover, since the expression~\eqref{J_xR of weighted degree m} must be regular on \(\{R_y \neq 0\}\), any factor of \(R_x\) in the denominator must cancel. Consequently, the coefficients \(F_{j,k}\) may be taken as polynomials, and we may set \(f_{j,k} = m-2k\) as the maximal admissible exponent, as claimed.
\end{proof}


\subsection{Denominator Exponent Constraints from Holomorphicity}

On $\mathbb{P}^2$, the affine coordinates $(x,y)$ arise from homogeneous coordinates $[T:X:Y]$ by setting $T=1$:
\[
(1,x,y) = \left(\frac{T}{T},\frac{X}{T},\frac{Y}{T}\right).
\]

Two further affine charts are obtained by normalizing with $X$ or $Y$ instead of $T$. Normalizing by $X$ yields coordinates $(\tilde{t}, \tilde{y})$ via $[\tilde{t}:1:\tilde{y}] = [T:X:Y] = \bigl[\frac{T}{X}:1:\frac{Y}{X}\bigr]$, namely
\begin{equation}\label{to-xy-chart-log}
\tilde{t} := \frac{1}{x},\qquad 
\tilde{y} := \frac{y}{x},
\end{equation}
with inverse transformation
\leqnomode\usetagform{default}
\begin{align}
\label{1-x-chart-log}
x = \frac{1}{\tilde{t}},\qquad 
y = \frac{\tilde{y}}{\tilde{t}}.
\end{align}

To study the behavior along the line at infinity $\mathbb{P}_\infty^1 := \{T=0\}$, we work in this new chart, where $\mathbb{P}_\infty^1$ is locally given by $\{\tilde{t}=0\}$. This chart covers the open subset $\mathbb{P}_\infty^1 \setminus \mathbb{P}_\infty^0$, with $\mathbb{P}_\infty^0 := \{T=0,\; X=0\}$ a single point of $\mathbb{P}_\infty^1$. The remaining part of $\mathbb{P}_\infty^1$ would be covered by normalizing with $Y$; however, since all objects considered are polynomial or rational, their vanishing or pole order along $\mathbb{P}_\infty^1$ is already determined on the dense open set $\mathbb{P}_\infty^1 \setminus \mathbb{P}_\infty^0$. It therefore suffices to examine the chart transitions~\eqref{to-xy-chart-log} and \eqref{1-x-chart-log}.

Under the change $(x,y) \mapsto (\tilde{t},\tilde{y})$ given by \eqref{1-x-chart-log}, the $1$-jets and $2$-jets transform according to
\begin{equation}\label{1 jets from xy to tilde_ty log}
\begin{aligned}
x' &= -\frac{\tilde{t}'}{\tilde{t}^2}, &
y' &= \frac{\tilde{y}'}{\tilde{t}} - \frac{\tilde{y}\,\tilde{t}'}{\tilde{t}^2}, \\[2mm]
x'' &= -\frac{\tilde{t}''}{\tilde{t}^2} + 2\frac{(\tilde{t}')^2}{\tilde{t}^3}, &
y'' &= \frac{\tilde{y}''}{\tilde{t}} - 2\frac{\tilde{t}'\,\tilde{y}'}{\tilde{t}^2} - \frac{\tilde{y}\,\tilde{t}''}{\tilde{t}^2} + 2\frac{\tilde{y}\,(\tilde{t}')^2}{\tilde{t}^3}.
\end{aligned}
\end{equation}

To understand its behavior along the line at infinity $\mathbb{P}_\infty^1$, given locally by $\{\tilde{t}=0\}$, we examine its transformation under~\eqref{1-x-chart-log}. The polynomial $R$ transforms as
\[
R(x,y) = R\!\left(\frac{1}{\tilde{t}},\frac{\tilde{y}}{\tilde{t}}\right) = \frac{1}{\tilde{t}^{\,d}}\, R^\ast(\tilde{t},\tilde{y}),
\]
where $d = \deg R$ and $R^\ast(\tilde{t},\tilde{y}) := \tilde{t}^{\,d} R(1/\tilde{t}, \tilde{y}/\tilde{t})$ is a polynomial with $R^\ast(0,\tilde{y}) \not\equiv 0$ for generic $R$.

Taking the logarithmic derivative of both sides  yields the transformation rules
\begin{equation}\label{log jets from R to tilde_t R^* log}
(\log R)' = \bigl(\log R^{*}\bigr)' - d\,\frac{\tilde{t}'}{\tilde{t}},
\end{equation}
and consequently for the second derivative,
\[
(\log R)'' = \bigl(\log R^{*}\bigr)'' - d\,\frac{\tilde{t}''}{\tilde{t}} + d\,\frac{(\tilde{t}')^{2}}{\tilde{t}^{2}}.
\]

Finally, the logarithmic Wronskian transforms as
\begin{equation}\label{log Wronskian from x R to tilde_t R^*}
\Delta_{xR} =
\begin{vmatrix}
x' & (\log R)' \\
x'' & (\log R)''
\end{vmatrix}
= -\frac{1}{\tilde{t}^{2}}
   \begin{vmatrix}
   \tilde{t}' & \bigl(\log R^{*}\bigr)' \\
   \tilde{t}'' & \bigl(\log R^{*}\bigr)''
   \end{vmatrix}
   + d\frac{(\tilde{t}')^{3}}{\tilde{t}^{4}}
   - 2\frac{(\tilde{t}')^{2}}{\tilde{t}^{3}}\bigl(\log R^{*}\bigr)'.
\end{equation}

Similarly, under the inverse change of chart \((\tilde{t},\tilde{y}) \mapsto (x,y)\) given by \eqref{to-xy-chart-log}, the \(1\)-jets and \(2\)-jets transform according to the following formulas:
\begin{equation}\label{1 jets from tilde_ty to xy log}
\begin{aligned}
\tilde{t}' &= -\frac{x'}{x^2}, &
\tilde{y}' &= \frac{y'}{x} - \frac{y\,x'}{x^2}, \\[2mm]
\tilde{t}'' &= -\frac{x''}{x^2} + 2\frac{(x')^2}{x^3}, &
\tilde{y}'' &= \frac{y''}{x} - 2\frac{x'\,y'}{x^2} - \frac{y\,x''}{x^2} + 2\frac{y\,(x')^2}{x^3}.
\end{aligned}
\end{equation}

The logarithmic derivative transforms as
\begin{equation}\label{log jets from R^* to xR log}
\begin{aligned}
\bigl(\log R^{*}\bigr)' &= (\log R)' - d \frac{x'}{x},\\[2mm]
\bigl(\log R^{*}\bigr)'' &= (\log R)'' - d \frac{x''}{x} + d\frac{(x')^2}{x^2}.
\end{aligned}
\end{equation}

Consequently, the logarithmic Wronskian transforms as
\begin{equation}\label{log Wronskian from tilde_t R^* to x R}
\begin{vmatrix}
\tilde{t}' & \bigl(\log R^{*}\bigr)' \\[2mm]
\tilde{t}'' & \bigl(\log R^{*}\bigr)''
\end{vmatrix}
= -\frac{1}{x^{2}}
   \begin{vmatrix}
   x' & (\log R)' \\[2mm]
   x'' & (\log R)''
   \end{vmatrix}
   + d\frac{(x')^{3}}{x^{4}}
   - 2\frac{(x')^{2}}{x^{3}}(\log R)'.
\end{equation}

To simplify computations, we now make a specific choice for the polynomial defining the curve \(\mathcal{C} \subset \mathbb{P}^2\). We take its homogeneous equation to be a perturbation of the Fermat curve:
\[
\widehat{R} \coloneqq T^{d} + X^{d} + Y^{d} + P(T,X) = 0,
\]
where the perturbation term \(P(T,X) \in \mathbb{C}[T,X]\) is homogeneous of degree \(d\) and involves only the variables \(T\) and \(X\). In the affine chart \(T = 1\), this becomes
\[
R(x,y) = 1 + x^{d} + y^{d} + S(x),
\]
with \(S(x) \in \mathbb{C}[x]\) a polynomial of degree at most \(d-1\), to be chosen later in a suitably generic way.

\begin{Convention}\label{convention log}
In what follows, we work under the following generic assumptions:\footnote{The equation \(\widehat{R}\) to be used later, with \(P(T, X) = X^6 T^{d-6}\) for \(d = 12, 13\), satisfies Convention~\ref{convention log}. The associated Maple code is available at \url{https://github.com/LeiHou-code/Hyperbolicity-Dimension-2-Conventions-Facts-2026}.}
\begin{itemize}
    \item The curve \(\mathcal{C}\) defined by \(\widehat{R} = 0\) is smooth.
    \item The divisors defined by \(\widehat{R}\) and its partial derivatives \(\widehat{R}_T, \widehat{R}_X, \widehat{R}_Y\) are in general position; i.e., any intersection of them has the expected (co)dimension.
    \item The coordinate line \(\{Y = 0\}\) and \(\{T=0\}\) intersect \(\mathcal{C}\) transversely.
\end{itemize}
\end{Convention}

The advantage of this choice is that, up to the nonzero constant factor $\frac{1}{d}$ (which causes no essential complication, especially for computer algebra calculations), we have simply
\[
\tfrac{1}{d} R_y = y^{d-1}.
\]

Recall that for an algebraic invariant logarithmic $2$-jet differential initially defined on $\{R_y \neq 0\}$ to be also holomorphic on $\{R_x \neq 0\}$, certain divisibility conditions by $R_y$ must be satisfied. With the above choice $R_y = d y^{d-1}$, these divisibility conditions translate directly into the requirement that, after expanding with respect to $y$, the polynomial coefficients of $y^0, y^1, \dots, y^{d-2}$ all vanish identically.

In contrast, if one starts with a truly generic polynomial $R(x,y)$ of degree $d$, its partial derivative $R_y$ would contain many monomials involving $y$, and the required divisibility by $R_y$ would involve a Euclidean division — a computationally intensive task even on powerful electronic equipment. It is therefore highly advantageous for us to have made the preliminary choice
\[
 R_y =d y^{d-1}.
\]

On $\{ T \neq 0 \} \cap \{ \widehat{R}_Y \neq 0 \}$, every invariant logarithmic $2$-jet differential of weighted order $m$ on $\mathbb{P}^2$ admits the general expression
\begin{equation}\label{J_X widehatR}
J_{X \widehat{R}} \coloneqq \sum_{\substack{0 \leqslant k \leqslant \lfloor m/3 \rfloor \\ 0 \leqslant j \leqslant m-3k}}
\frac{\widehat{G}_{j,k}}{T^{g_{j,k}} \widehat{R}_Y^{m-2k}}\,
(x')^{m-3k-j}\Bigl(\frac{R'}{R}\Bigr)^{\!j}\,
\bigl(\Delta_{xR}'''\bigr)^k,
\end{equation}
where $\widehat{G}_{j,k}(T,X,Y)$ are homogeneous polynomials in $\mathbb{C}[T,X,Y]$ satisfying
\begin{equation*}
    \deg \widehat{G}_{j,k} = g_{j,k} + (m-2k)(d-1),
\end{equation*}
and $\widehat{R}_Y$ is the partial derivative of the homogeneous polynomial $\widehat{R}$ as in Convention \ref{convention log}. Here the exponent of \(\widehat{R}_Y\) in the denominator is taken to be \(m-2k\), which coincides with the maximal admissible exponent of \(R_y\) in the denominators of the coefficients \(F_{j,k} / R_y^{m-2k}\) from Proposition~\ref{prop:4.2}.

Similarly, on \(\{ X \neq 0 \} \cap \{ \widehat{R}_Y \neq 0 \}\), every invariant logarithmic \(2\)-jet differential of weighted order \(m\) on \(\mathbb{P}^2\) admits the general expression
\begin{equation}\label{J_T widehatR}
J_{T \widehat{R}} \coloneqq 
\sum_{\substack{0 \leqslant q \leqslant \lfloor m/3 \rfloor \\ 0 \leqslant p \leqslant m-3q}} 
\frac{\widehat{E}_{p,q}}{X^{e_{p,q}} \widehat{R}_Y^{m-2q}}
(\tilde{t}')^{m-3q-p}
\bigl((\log R^{*})'\bigr)^{p}
\begin{vmatrix}
\tilde{t}' & (\log R^{*})' \\
\tilde{t}'' & (\log R^{*})''
\end{vmatrix}^q,
\end{equation}
where \(\widehat{E}_{p,q}(T,X,Y)\) are homogeneous polynomials in \(\mathbb{C}[T,X,Y]\) satisfying
\[
\deg \widehat{E}_{p,q} = e_{p,q} + (m-2q)(d-1).
\]

\begin{Question}
{\sl What happens to the general form \(J_{T \widehat{R}}\) given in \eqref{J_T widehatR} after the change of affine chart from \(\{ X \neq 0 \} \cap \{ \widehat{R}_Y \neq 0 \}\) to \(\{ T \neq 0 \} \cap \{ \widehat{R}_Y \neq 0 \}\)?}
\end{Question}

Substituting \eqref{1 jets from tilde_ty to xy log}, \eqref{log jets from R^* to xR log}, and \eqref{log Wronskian from tilde_t R^* to x R} into \eqref{J_T widehatR} yields the following transformation for the term:
\begin{align*}
&
\frac{\widehat{E}_{p,q}}{X^{e_{p,q}} \widehat{R}_{Y}^{m-2q}}
(\tilde{t}')^{m-3q-p}
((\log R^{*})')^{p}
\begin{vmatrix}
\tilde{t} ' & \left(\log R^{*}\right) '\\
\tilde{t} '' & \left(\log R^{*}\right) ''
\end{vmatrix}^q
\\
=&
\frac{\widehat{E}_{p,q}}{X^{e_{p,q}} \widehat{R}_{Y}^{m-2q}}
(-\frac{x'}{x^2})^{m-3q-p}
((\log R)' - d \frac{x'}{x})^{p}
\\
&
\cdot
\left(
-\frac{1}{x^{2}}
\begin{vmatrix}
x' & (\log R) '\\
x'' & (\log R) ''
\end{vmatrix}
+
d\frac{1}{x^{4}} (x')^{3} 
-
2\frac{1}{x^{3}} (x')^{2}\left(\log R\right)'
\right)^q
\\
=&
\frac{( -1)^{m-3q-p}}{x^{2( m-3q-p) +p+4q}}
\frac{\hat{E}_{p,q}}{X^{e_{p,q}}\hat{R}_{Y}^{m-2q}}
(x')^{m-3q-p}
(x(\log R)'-dx')^{p}
\\
&\cdot
\left( 
-x^{2}
\begin{vmatrix}
x' & (\log R)'\\
x'' & (\log R)''
\end{vmatrix} 
+
d(x')^{3} 
-
2x(x')^{2}(\log R)'
\right)^{q},
\end{align*}
where the factor $(x(\log R)'-dx')^{p}$ can be expanded as
\begin{equation*}
    (x(\log R)'-dx')^{p} = \sum _{r=0}^{p}( -1)^{p-r}\binom{p}{r} d^{p-r} x^{r}( x')^{p-r}( (\log R)')^{r}.
\end{equation*}
The last factor expands as
\begin{align*}
&
\left( 
-x^{2}
\begin{vmatrix}
x' & (\log R)'\\
x'' & (\log R)''
\end{vmatrix} 
+
d(x')^{3} 
-
2x(x')^{2}(\log R)'
\right)^{q}
\\
=&
\sum_{\alpha +\beta +\gamma =q} 
\frac{q!}{\alpha ! \beta ! \gamma !}
d^{\alpha } (x')^{3\alpha }( -1)^{\beta } 2^{\beta } x^{\beta }( x')^{2\beta }( (\log R)')^{\beta }( -1)^{\gamma } x^{2\gamma }\begin{vmatrix}
x' & (\log R)'\\
x'' & (\log R)''
\end{vmatrix}^{\gamma}
\\
=&
\sum_{\alpha +\beta +\gamma =q}
(-1)^{\beta + \gamma} 
d^{\alpha} 2^{\beta} \frac{q!}{\alpha ! \beta ! \gamma !}
x^{\beta + 2\gamma}
(x')^{3\alpha +2\beta}
((\log R)')^{\beta}
\begin{vmatrix}
x' & (\log R)'\\
x'' & (\log R)''
\end{vmatrix}^{\gamma}
\\
=&
\sum_{\gamma =0}^{q} \sum_{\beta =0}^{q-\gamma} 
(-1)^{\beta + \gamma} 
d^{q- \beta - \gamma} 2^{\beta} \frac{q!}{\beta ! \gamma ! (q- \beta - \gamma)!}
x^{\beta + 2\gamma}
(x')^{3(q- \beta - \gamma) +2\beta}
((\log R)')^{\beta}
\begin{vmatrix}
x' & (\log R)'\\
x'' & (\log R)''
\end{vmatrix}^{\gamma}.
\end{align*}

Combining the expansions, the term becomes
\begin{align*}
&
\frac{\widehat{E}_{p,q}}{X^{e_{p,q}} \widehat{R}_{Y}^{m-2q}}
(\tilde{t}')^{m-3q-p}
\bigl((\log R^{*})'\bigr)^{p}
\begin{vmatrix}
\tilde{t}' & (\log R^{*})' \\
\tilde{t}'' & (\log R^{*})''
\end{vmatrix}^q
\\
=& \sum_{r=0}^{p} \sum_{\gamma=0}^{q} \sum_{\beta=0}^{q-\gamma}
   (-1)^{m-3q-r+\beta+\gamma} 2^{\beta}
   d^{\,p-r+q-\beta-\gamma}
   \binom{p}{r} \frac{q!}{\beta!\,\gamma!\,(q-\beta-\gamma)!}
\\
&\cdot
   \frac{x^{r+\beta+2\gamma}}{x^{2(m-3q-p)+p+4q}}
   \frac{\widehat{E}_{p,q}}{X^{e_{p,q}} \widehat{R}_{Y}^{m-2q}}
   (x')^{m-3q-r+3(q-\beta-\gamma)+2\beta}
   \bigl((\log R)'\bigr)^{r+\beta}
   \begin{vmatrix}
   x' & (\log R)' \\
   x'' & (\log R)''
   \end{vmatrix}^{\gamma}.
\end{align*}

Introducing new indices \(k := \gamma\) and \(j := r + \beta\), we rewrite the expression as
\begin{align*}
&
\frac{\widehat{E}_{p,q}}{X^{e_{p,q}} \widehat{R}_{Y}^{m-2q}}
(\tilde{t}')^{m-3q-p}
\bigl((\log R^{*})'\bigr)^{p}
\begin{vmatrix}
\tilde{t}' & (\log R^{*})' \\
\tilde{t}'' & (\log R^{*})''
\end{vmatrix}^q
\\
=& \sum_{r=0}^{p} \sum_{k=0}^{q} \sum_{j=r}^{q+r-k}
   (-1)^{m-3q-2r+j+k} 2^{j-r}
   d^{\,p+q-j-k}
   \binom{p}{r} \frac{q!}{(j-r)!\,k!\,(q+r-j-k)!}
\\
&\cdot
   \frac{x^{j+2k}}{x^{2(m-3q-p)+p+4q}}
   \frac{\widehat{E}_{p,q}}{X^{e_{p,q}} \widehat{R}_{Y}^{m-2q}}
   (x')^{m-3k-j}
   \bigl((\log R)'\bigr)^{j}
   \begin{vmatrix}
   x' & (\log R)' \\
   x'' & (\log R)''
   \end{vmatrix}^{k}.
\end{align*}

Consequently, the full expression \(J_{T \widehat{R}}\) transforms as
\begin{align*}
J_{T \widehat{R}}
=&
\sum_{q=0}^{\lfloor m/3 \rfloor} \sum_{p=0}^{m-3q} \sum_{r=0}^{p} \sum_{k=0}^{q} \sum_{j=r}^{q+r-k} 
(-1)^{m-3q-2r+j+k} 2^{j-r} d^{p+q-j-k}
\binom{p}{r} \frac{q!}{(j-r)!k!(q+r-j-k)!}
\\
&\cdot
\frac{x^{j+2k}}{x^{2(m-3q-p)+p+4q}} 
\frac{\hat{E}_{p,q}}{X^{e_{p,q}}\hat{R}_{Y}^{m-2q}}
(x')^{m-3k-j}
((\log R)')^{j}
\begin{vmatrix}
x' & (\log R)'\\
x'' & (\log R)''
\end{vmatrix}^{k}
\\
=&
\sum_{k=0}^{\lfloor m/3\rfloor} \sum_{j=0}^{m-3k} \sum_{q=k}^{\lfloor m/3\rfloor} \sum_{r=\max (0,\ j+k-q)}^{\min (j,\ m-3q)} \sum_{p=r}^{m-3q}
(-1)^{m-3q-2r+j+k} 2^{j-r} d^{p+q-j-k}
\binom{p}{r} \frac{q!}{(j-r)!k!(q+r-j-k)!}
\\
&\cdot
\frac{x^{j+2k+p+2q}}{x^{2m}} 
\frac{\hat{E}_{p,q}}{X^{e_{p,q}}\hat{R}_{Y}^{m-2q}}
(x')^{m-3k-j}
((\log R)')^{j}
\begin{vmatrix}
x' & (\log R)'\\
x'' & (\log R)''
\end{vmatrix}^{k}
\\
=&
\sum_{k=0}^{\lfloor m/3\rfloor} \sum_{j=0}^{m-3k} 
\frac{T^{m-j-2k}}{X^{m-j-2k}}
\sum_{q=k}^{\lfloor m/3\rfloor} \sum_{r=\max (0,\ j+k-q)}^{\min (j,\ m-3q)} \sum_{p=r}^{m-3q}
\\
&\cdot
(-1)^{m-3q-2r+j+k} 2^{j-r} d^{p+q-j-k}
\binom{p}{r} \frac{q!}{(j-r)!k!(q+r-j-k)!}
\frac{T^{m-p-2q}}{X^{m-p-2q}}
\frac{\hat{E}_{p,q}}{X^{e_{p,q}}\hat{R}_{Y}^{m-2q}}
\\
&\cdot
(x')^{m-3k-j}
((\log R)')^{j}
\begin{vmatrix}
x' & (\log R)'\\
x'' & (\log R)''
\end{vmatrix}^{k}.
\end{align*}

Observe that \(p+2q \leqslant m-3q+2q = m-q \leqslant m-k\), which implies \(m-p-2q \geqslant k\). Hence we may write
\begin{align*}
J_{T \widehat{R}}
=&
\sum_{k=0}^{\lfloor m/3\rfloor} \sum_{j=0}^{m-3k} 
\frac{T^{m-j-k}}{X^{m-j-k}}
\sum_{q=k}^{\lfloor m/3\rfloor} \sum_{r=\max (0,\ j+k-q)}^{\min (j,\ m-3q)} \sum_{p=r}^{m-3q}
\\
&\cdot
(-1)^{m-3q-2r+j+k} 2^{j-r} d^{p+q-j-k}
\binom{p}{r} \frac{q!}{(j-r)!k!(q+r-j-k)!}
\frac{T^{m-k-p-2q}}{X^{m-k-p-2q}}
\frac{\hat{E}_{p,q}}{X^{e_{p,q}}\hat{R}_{Y}^{m-2q}}
\\
&\cdot
(x')^{m-3k-j}
((\log R)')^{j}
\begin{vmatrix}
x' & (\log R)'\\
x'' & (\log R)''
\end{vmatrix}^{k}.
\end{align*}

The coefficients
\begin{equation*}
\frac{T^{m-j-k}}{X^{m-j-k}}
(-1)^{m-3q-2r+j+k} 2^{j-r} d^{p+q-j-k}
\binom{p}{r} \frac{q!}{(j-r)!k!(q+r-j-k)!}
\frac{T^{m-k-p-2q}}{X^{m-k-p-2q}}
\frac{\hat{E}_{p,q}}{X^{e_{p,q}}\hat{R}_{Y}^{m-2q}}
\end{equation*}
in the last line, as rational functions, not only have no pole along the divisor $\{ T=0 \}$, but also vanish with multiplicity at least $m-j-k$.

Comparing with the coefficients of \eqref{J_X widehatR}, it follows that, on $\{ T \neq 0 \} \cap \{ \widehat{R}_Y \neq 0 \}$, every invariant logarithmic $2$-jet differential of weighted order $m$ on $\mathbb{P}^2$ admits the general expression
\begin{equation}\label{J_X widehatR with bounded T}
J_{X \widehat{R}} = \sum_{\substack{0 \leqslant k \leqslant \lfloor m/3 \rfloor \\ 0 \leqslant j \leqslant m-3k}}
\frac{T^{m-j-k} \widehat{H}_{j,k}}{\widehat{R}_Y^{m-2k}}\,
(x')^{m-3k-j}\Bigl(\frac{R'}{R}\Bigr)^{\!j}\,
\bigl(\Delta_{xR}'''\bigr)^k,
\end{equation}
where $\widehat{H}_{j,k}(T,X,Y)$ are homogeneous polynomials in $\mathbb{C}[T,X,Y]$ satisfying
\begin{equation*}
    \deg \widehat{H}_{j,k} = (m-2k)(d-1)-m+j+k.
\end{equation*}

If we now require the logarithmic jet differential as~\eqref{J_X widehatR with bounded T} to vanish to order at least $t \geqslant 1$ along the divisor $\{ X=0 \}$, i.e.,
\[
\Twist := -t \leqslant -1,
\]
then this corresponds to considering sections of the jet differential bundle twisted by $\mathcal{O}_{\P^2}(-t)$, i.e.,
\[
(\cdot) \otimes \mathcal{O}_{\P^2}(-t).
\]

Therefore, on \(\{ T \neq 0 \} \cap \{ \widehat{R}_Y \neq 0 \}\), every invariant logarithmic \(2\)-jet differential of weighted order \(m\) and vanishing order at least \(t\) along \(\{ X=0 \}\) on \(\mathbb{P}^2\) admits the general expression
\begin{equation}\label{best J_X widehatR}
J_{X \widehat{R}} = \sum_{\substack{0 \leqslant k \leqslant \lfloor m/3 \rfloor \\ 0 \leqslant j \leqslant m-3k}}
\frac{T^{m-j-k} X^{t} \widehat{K}_{j,k}}{\widehat{R}_Y^{m-2k}}\,
(x')^{m-3k-j}\Bigl(\frac{R'}{R}\Bigr)^{\!j}\,
\bigl(\Delta_{xR}'''\bigr)^k,
\end{equation}
where \(\widehat{K}_{j,k}(T,X,Y)\) are homogeneous polynomials in \(\mathbb{C}[T,X,Y]\) satisfying
\[
\deg \widehat{K}_{j,k} = (m-2k)(d-1) - m + j + k - t.
\]

When the coefficients of the above expression are expressed as regular functions in the affine coordinates \((x,y)\), we obtain
\begin{equation}\label{best J_X widehatR inhomogeneous}
J_{x R} 
\coloneqq 
x^{t} \sum_{\substack{0 \leqslant k \leqslant \lfloor m/3 \rfloor \\ 0 \leqslant j \leqslant m-3k}}
\frac{K_{j,k}}{R_y^{m-2k}}\,
(x')^{m-3k-j}\Bigl(\frac{R'}{R}\Bigr)^{\!j}\,
\bigl(\Delta_{xR}'''\bigr)^k,
\end{equation}
where \(K_{j,k}(x,y)\) are polynomials in \(\mathbb{C}[x,y]\) satisfying
\[
\deg K_{j,k} \leqslant (m-2k)(d-1) - m + j + k - t.
\]

To prove Lemma~\ref{lem-1.2}, we focus on the range \(m \leqslant 14\). For these low weights, the above expression simplifies considerably. Indeed, the exponent \(k\) of the logarithmic Wronskian \(\Delta_{xR}'''\) satisfies \(k \leqslant \lfloor m/3 \rfloor \leqslant 4\). Terms involving a power of the logarithmic Wronskian exceeding \(\lfloor m/3 \rfloor\) are understood to be identically zero.


It remains to check the holomorphicity of the jet differential~\eqref{best J_X widehatR} along the divisor \(\{T = 0\}\) at infinity.

\begin{Question}
{\sl What happens to the general form \(J_{X \widehat{R}}\) given in \eqref{best J_X widehatR} after the change of affine chart~\eqref{1-x-chart-log}?}
\end{Question}

Substituting \eqref{1 jets from xy to tilde_ty log}, \eqref{log jets from R to tilde_t R^* log}, and \eqref{log Wronskian from x R to tilde_t R^*} into \eqref{best J_X widehatR}, the term transforms as follows:
\begin{align*}
&
\frac{T^{m-j-k} X^{t} \widehat{K}_{j,k}}{\widehat{R}_Y^{m-2k}}\,
(x')^{m-3k-j}\Bigl(\frac{R'}{R}\Bigr)^{\!j}\,
\bigl(\Delta_{xR}'''\bigr)^k
\\
=&
\frac{T^{m-j-k} X^{t} \widehat{K}_{j,k}}{\widehat{R}_Y^{m-2k}}
(-\frac{\tilde{t}'}{\tilde{t}^2})^{m-3k-j}
((\log R^{*})' - d \frac{\tilde{t}'}{\tilde{t}})^{j}
\\
&
\cdot
\left(
-\frac{1}{\tilde{t}^{2}}
\begin{vmatrix}
\tilde{t} ' & \left(\log R^{*}\right) '\\
\tilde{t} '' & \left(\log R^{*}\right) ''
\end{vmatrix} 
+
d\frac{1}{\tilde{t}^{4}} (\tilde{t}')^{3} 
-
2\frac{1}{\tilde{t}^{3}} (\tilde{t}')^{2}\left(\log R^{*}\right)'
\right)^k
\\
=&
\frac{(-1)^{m-3k-j}}{\tilde{t}^{2(m-3k-j)+j+4k}}
\frac{T^{m-j-k} X^{t} \widehat{K}_{j,k}}{\widehat{R}_Y^{m-2k}}
(\tilde{t}')^{m-3k-j}
(\tilde{t}(\log R^{*})'-d \tilde{t}')^{j}
\\
&\cdot
\left( 
-{\tilde{t}^{2}}
\begin{vmatrix}
\tilde{t} ' & \left(\log R^{*}\right) '\\
\tilde{t} '' & \left(\log R^{*}\right) ''
\end{vmatrix} 
+
d (\tilde{t}')^{3} 
-
2 \tilde{t} (\tilde{t}')^{2}\left(\log R^{*}\right)'
\right)^{k},
\end{align*}
where the factor $(\tilde{t}(\log R^{*})'-d \tilde{t}')^{j}$ can be expanded as
\begin{equation*}
    (\tilde{t}(\log R^{*})'-d \tilde{t}')^{j} = \sum _{r=0}^{j}(-1)^{j-r}\binom{j}{r} d^{j-r} \tilde{t}^{\,r}(\tilde{t}')^{j-r}( (\log R^{*})')^{r}.
\end{equation*}
The last factor expands as
\begin{align*}
&
\left( 
-{\tilde{t}^{2}}
\begin{vmatrix}
\tilde{t} ' & \left(\log R^{*}\right) '\\
\tilde{t} '' & \left(\log R^{*}\right) ''
\end{vmatrix} 
+
d (\tilde{t}')^{3} 
-
2 \tilde{t} (\tilde{t}')^{2}\left(\log R^{*}\right)'
\right)^{k}
\\
=&
\sum_{\alpha +\beta +\gamma =k} 
\frac{k!}{\alpha ! \beta ! \gamma !}
d^{\alpha } (\tilde{t}')^{3\alpha }( -1)^{\beta } 2^{\beta } \tilde{t}^{\beta }( \tilde{t}')^{2\beta }( (\log R^{*})')^{\beta }( -1)^{\gamma } \tilde{t}^{2\gamma }
\begin{vmatrix}
\tilde{t}' & (\log R^{*})'\\
\tilde{t}'' & (\log R^{*})''
\end{vmatrix}^{\gamma}
\\
=&
\sum_{\alpha +\beta +\gamma =k}
(-1)^{\beta + \gamma} 
d^{\alpha} 2^{\beta} \frac{k!}{\alpha ! \beta ! \gamma !}
\tilde{t}^{\beta + 2\gamma}
(\tilde{t}')^{3\alpha +2\beta}
((\log R^{*})')^{\beta}
\begin{vmatrix}
\tilde{t}' & (\log R^{*})'\\
\tilde{t}'' & (\log R^{*})''
\end{vmatrix}^{\gamma}
\\
=&
\sum_{\gamma =0}^{k} \sum_{\beta =0}^{k-\gamma} 
(-1)^{\beta + \gamma} 
d^{k- \beta - \gamma} 2^{\beta} \frac{k!}{\beta ! \gamma ! (k- \beta - \gamma)!}
\tilde{t}^{\beta + 2\gamma}
(\tilde{t}')^{3(k- \beta - \gamma) +2\beta}
((\log R^{*})')^{\beta}
\begin{vmatrix}
\tilde{t}' & (\log R^{*})'\\
\tilde{t}'' & (\log R^{*})''
\end{vmatrix}^{\gamma}.
\end{align*}

Combining the expansions, the term becomes
\begin{align*}
&
\frac{T^{m-j-k} X^{t} \widehat{K}_{j,k}}{\widehat{R}_Y^{m-2k}}\,
(x')^{m-3k-j}\Bigl(\frac{R'}{R}\Bigr)^{\!j}\,
\bigl(\Delta_{xR}'''\bigr)^k
\\
=&
\sum_{r=0}^{j} \sum_{\gamma =0}^{k} \sum_{\beta =0}^{k-\gamma }
(-1)^{m-3k-r+\beta +\gamma} 2^{\beta} d^{j-r+k-\beta -\gamma}
\binom{j}{r} \frac{k!}{\beta !\gamma !(k-\beta -\gamma)!}
\frac{\tilde{t}^{\, r+\beta +2\gamma }}{\tilde{t}^{\, 2(m-3k-j)+j+4k}}
\frac{T^{m-j-k} X^{t} \widehat{K}_{j,k}}{\widehat{R}_Y^{m-2k}}
\\
&\cdot
(\tilde{t}')^{m-3k-r+3(k-\beta -\gamma )+2\beta}
((\log R^{*})')^{r+\beta}
\begin{vmatrix}
\tilde{t}' & (\log R^{*})'\\
\tilde{t}'' & (\log R^{*})''
\end{vmatrix}^{\gamma}
\\
=&
\sum_{r=0}^{j} \sum_{q=0}^{k} \sum_{p=r}^{k+r-q}
(-1)^{m-3k-2r+p+q} 2^{p-r} d^{j+k-p-q}
\binom{j}{r} \frac{k!}{(p-r)!q!(k+r-p-q)!}
\\
&\cdot
\frac{\tilde{t}^{\, p+2q}}{\tilde{t}^{\, 2(m-3k-j)+j+4k}} 
\frac{T^{m-j-k} X^{t} \widehat{K}_{j,k}}{\widehat{R}_Y^{m-2k}}
(\tilde{t}')^{m-3q-p}
((\log R^{*})')^{p}
\begin{vmatrix}
\tilde{t}' & (\log R^{*})'\\
\tilde{t}'' & (\log R^{*})''
\end{vmatrix}^{q},
\end{align*}
where the last line is obtained by setting $q \coloneqq \gamma$ and $p \coloneqq r + \beta$.

Consequently, the full expression transforms as
\begin{align*}
J_{T \widehat{R}}
=&
\sum_{k=0}^{\lfloor m/3 \rfloor} \sum_{j=0}^{m-3k} \sum_{r=0}^{j} \sum_{q=0}^{k} \sum_{p=r}^{k+r-q} 
(-1)^{m-3k-2r+p+q} 2^{p-r} d^{j+k-p-q}
\binom{j}{r} \frac{k!}{(p-r)!q!(k+r-p-q)!}
\\
&\cdot
\frac{\tilde{t}^{\, p+2q}}{\tilde{t}^{\, 2(m-3k-j)+j+4k}} 
\frac{T^{m-j-k} X^{t} \widehat{K}_{j,k}}{\widehat{R}_Y^{m-2k}}
(\tilde{t}')^{m-3q-p}
((\log R^{*})')^{p}
\begin{vmatrix}
\tilde{t}' & (\log R^{*})'\\
\tilde{t}'' & (\log R^{*})''
\end{vmatrix}^{q}
\\
=&
\sum_{q=0}^{\lfloor m/3\rfloor} \sum_{p=0}^{m-3q} \sum_{k=q}^{\lfloor m/3\rfloor} \sum_{r=\max (0,\ p+q-k)}^{\min (p,\ m-3k)} \sum_{j=r}^{m-3k}
(-1)^{m-3k-2r+p+q} 2^{p-r} d^{j+k-p-q}
\binom{j}{r} \frac{k!}{(p-r)!q!(k+r-p-q)!}
\\
&\cdot
\frac{\tilde{t}^{\, j+2k+p+2q}}{\tilde{t}^{\, 2m}} 
\frac{T^{m-j-k} X^{t} \widehat{K}_{j,k}}{\widehat{R}_Y^{m-2k}}
(\tilde{t}')^{m-3q-p}
((\log R^{*})')^{p}
\begin{vmatrix}
\tilde{t}' & (\log R^{*})'\\
\tilde{t}'' & (\log R^{*})''
\end{vmatrix}^{q}
\\
=&
\sum_{q=0}^{\lfloor m/3\rfloor} \sum_{p=0}^{m-3q} 
\sum_{k=q}^{\lfloor m/3\rfloor} \sum_{r=\max (0,\ p+q-k)}^{\min (p,\ m-3k)} \sum_{j=r}^{m-3k}
\\
&
(-1)^{m-3k-2r+p+q} 2^{p-r} d^{j+k-p-q}
\binom{j}{r} \frac{k!}{(p-r)!q!(k+r-p-q)!}
\\
&\cdot
X^{m-j-2k} \frac{X^{m-p-2q}}{T^{m-p-2q}} T^{k}
\frac{X^{t} \widehat{K}_{j,k}}{\widehat{R}_Y^{m-2k}}
(\tilde{t}')^{m-3q-p}
((\log R^{*})')^{p}
\begin{vmatrix}
\tilde{t}' & (\log R^{*})'\\
\tilde{t}'' & (\log R^{*})''
\end{vmatrix}^{q}
\\
=&
\sum_{q=0}^{\lfloor m/3\rfloor} \sum_{p=0}^{m-3q} 
\frac{X^{m-p-q} X^{t}}{\widehat{R}_Y^{m-2q} T^{m-3q-p}}
\sum_{k=q}^{\lfloor m/3\rfloor} \sum_{r=\max (0,\ p+q-k)}^{\min (p,\ m-3k)} \sum_{j=r}^{m-3k}
\\
&
(-1)^{m-3k-2r+p+q} 2^{p-r} d^{j+k-p-q}
\binom{j}{r} \frac{k!}{(p-r)!q!(k+r-p-q)!}
{X^{m-q-j-2k} T^{k-q} \widehat{R}_Y^{2(k-q)} \widehat{K}_{j,k}} 
\\
&\cdot
(\tilde{t}')^{m-3q-p}
((\log R^{*})')^{p}
\begin{vmatrix}
\tilde{t}' & (\log R^{*})'\\
\tilde{t}'' & (\log R^{*})''
\end{vmatrix}^{q},
\end{align*}
where the exponent \(m - q - j - 2k \geqslant 0\) since \(j + 2k \leqslant (m-3k) + 2k = m - k\) and \(q \leqslant k\).

To guarantee the holomorphicity of the jet differential~\eqref{best J_X widehatR} along the divisor at infinity \(\{T = 0\}\), each coefficient
\begin{align*}
&
\frac{X^{m-p-q} X^{t}}{\widehat{R}_Y^{m-2q} T^{m-3q-p}}
\sum_{k=q}^{\lfloor m/3\rfloor} \sum_{r=\max(0,\,p+q-k)}^{\min(p,\,m-3k)} \sum_{j=r}^{m-3k}
\\
&
(-1)^{m-3k-2r+p+q} 2^{p-r} d^{j+k-p-q}
\binom{j}{r} \frac{k!}{(p-r)!q!(k+r-p-q)!}
{X^{m-q-j-2k} T^{k-q} \widehat{R}_Y^{2(k-q)} \widehat{K}_{j,k}} 
\\
=&
(-1)^{m+p+q} 
\frac{X^{m-p-q} X^{t}}{\widehat{R}_Y^{m-2q} T^{m-3q-p}}
\sum_{k=q}^{\lfloor m/3\rfloor} \sum_{r=\max(0,\,p+q-k)}^{\min(p,\,m-3k)} \sum_{j=r}^{m-3k} 
\\
&
(-1)^{-3k} 2^{p-r} d^{j+k-p-q}
\binom{j}{r} \frac{k!}{(p-r)!\,q!\,(k+r-p-q)!}
\, X^{m-q-j-2k} T^{k-q} \widehat{R}_Y^{2(k-q)} \widehat{K}_{j,k},
\end{align*}
where the last equality uses \((-1)^{2r} = 1\) for any integer \(r\), must be regular on \(\{ X \neq 0 \} \cap \{ \widehat{R}_Y \neq 0 \}\). Consequently, the holomorphicity conditions translate into the following divisibility requirements in the polynomial ring \(\mathbb{C}[T,X,Y]\):
\begin{align}\label{divisibility by the powers of T log}
T^{m-3q-p} 
&
\; \Big\vert \;
\sum_{k=q}^{\lfloor m/3\rfloor} \sum_{r=\max (0,\ p+q-k)}^{\min (p,\ m-3k)} \sum_{j=r}^{m-3k}
\\
&
(-1)^{-3k} 2^{p-r} d^{j+k-p-q}
\binom{j}{r} \frac{k!}{(p-r)!\,q!\,(k+r-p-q)!}
\, X^{m-q-j-2k} T^{k-q} \widehat{R}_Y^{2(k-q)} \widehat{K}_{j,k}, \notag
\end{align}
for every $q=0, \dots , \lfloor \frac{m}{3} \rfloor$ and $p=0, \dots , m-3q$.

\subsection{End of the Proof of the Key Vanishing Lemma~\ref{lem-1.2}}\label{Subsec-logarithmic-vanishing-theorem}

For a smooth curve $\mathcal{C} \subset \mathbb{P}^2$, we consider the bundle $E_{2,m}T_{\mathbb{P}^2}^\ast(\log\mathcal{C})$ of invariant logarithmic $2$-jet differentials and its negative twists:
\[
E_{2,m}T_{\mathbb{P}^2}^\ast(\log\mathcal{C}) \otimes \mathcal{O}_{\mathbb{P}^2}(-t),
\]
with integer $t \geqslant 1$.

We deliberately choose the affine equation of the curve $\mathcal{C}$ to be
\[
R := (x^d + y^d + 1) + x^6 \in \mathbb{C}[x,y],
\]
where the extra monomial $x^6$ in this deformed Fermat-type polynomial is determined experimentally: lower-order perturbations such as $x^1, x^2, \dots$ fail to yield the desired vanishing for certain pairs $(m,t)$ (nontrivial solutions persist). 
Moreover, its projectivization $\widehat{R} \coloneqq T^d + X^d + Y^d + X^6 T^{d-6}$ has degree \(d\).
One can verify the smoothness of $\mathcal{C}$ as well as the general position of the four divisors $R, R_X, R_Y, R_T$; we omit these routine computations.

For instance, when \(m=3\), by Proposition~\ref{prop:4.2}, an invariant logarithmic \(2\)-jet differential of weighted order \(3\) in $E_{2,m}T^*_{\mathbb{P}^2}(\log\mathcal{C})$ takes the following general local form on $\{R_y \neq 0\}$:
\begin{equation*}
J_{xR}
= 
\sum_{0\leqslant j\leqslant 3}
   \frac{A_j}{R_y^3}\, (x')^{3-j}\Big(\frac{R'}{R}\Big)^j 
+ 
\frac{B_0}{R_y}\,\Delta_{xR}''' .
\end{equation*}

The jet bundle transition from $\{R_y \neq 0\}$ to $\{R_x \neq 0\}$ consists in solving for $x'$ and $x''$ from the relations~\eqref{relations of log jet} which gives the solutions:
\begin{equation}\label{equ:x' and x''}
\aligned
x'
&
\,=\,
-\,y'\,\frac{R_y}{R_x}
+
\frac{R'}{R_x},
\\
x''
&
\,=\,
-
y''\,\frac{R_y}{R_x}
+
\frac{R''}{R_x}
\\
&
\ \ \ \ \
-\,
y'y'\,
\frac{R_y^2\,R_{xx}}{R_x^3}
+
2\,y'\,R'\,
\frac{R_y\,R_{xx}}{R_x^3}
-
R'R'\,
\frac{R_{xx}}{R_x^3}
+
2\,y'y'\,
\frac{R_y\,R_{xy}}{R_x^2}
-
2\,y'R'\,
\frac{R_{xy}}{R_x^2}
-
y'y'\,
\frac{R_{yy}}{R_x},
\endaligned
\end{equation}
and in replacing \(x'\) and \(x''\) in the considered jet differential $J_{xR}$.

Performing the transition from 
$\{R_y \neq 0\}$
to
$\{R_x \neq 0\}$, we obtain the transformed expression:
\[
\aligned
J_{xR}
&
\,=\,
\Delta_{Ry}'''\,
\frac{B_0}{R_x}
\\
&
\ \ \ \ \
-\,
{y'}^3\,
\frac{A_0}{R_x^3}
\\
&
\ \ \ \ \
+
R'{y'}^2
\bigg(
\frac{3\,A_0}{R_y R_x^3}
+
\frac{A_1}{R_y R\,R_x^2}
+
B_0
\bigg[
\frac{R_{xx} R_y}{R\,R_x^3}
-
\frac{2\,R_{xy}}{R\,R_x^2}
+
\frac{R_{yy}}{R_y R\,R_x}
\bigg]
\bigg)
\\
&
\ \ \ \ \
+
{R'}^2y'
\bigg(
-\,
\frac{3\,A_0}{R_y^{2}R_x^3}
-
\frac{2\,A_1}{R_y^{2}R\,R_x^2}
-
\frac{A_2}{R_y^{2}R^2\,R_x}
+
B_0
\bigg[
\frac{2\,R_{xy}}{R_y R\,R_x^2}
-
\frac{2\,R_{xx}}{R\,R_x^3}
\bigg]
\bigg)
\\
&
\ \ \ \ \
+
{R'}^3
\bigg(
\frac{A_0}{R_y^{3}R_x^3}
+
\frac{A_1}{R_y^{3}R\,R_x^2}
+
\frac{A_2}{R_y^{3}R^2\,R_x}
+
\frac{A_3}{R_y^{3}R^3}
+
B_0
\bigg[
\frac{R_{xx}}{R_y R\,R_x^3}
-
\frac{1}{R_y R^2\,R_x}
\bigg]
\bigg).
\endaligned
\]

In the first two lines, only \(R_x\) appears in denominators — which is allowed on the open set \(\{R_x \neq 0\}\) — so no constraints are imposed on \(B_0\) or \(A_0\) from these terms.

On the third line, however, \(R_y\) appears to the first power in denominators. More precisely, the coefficient of \(R'{y'}^2\) in \(J_{xR}\), which we denote by \(\big[ R'{y'}^2 \big] \big( J_{xR} \big)\), factors as:
\[
\big[ R'{y'}^2 \big] \big( J_{xR} \big)
= 
\frac{
3\,A_0\,R
+
A_1\,R_x
+
B_0
\big[
R_y^2\,R_{xx}
-
2\,R_y\,R_x\,R_{xy}
+
R_x^2\,R_{yy}
\big]
}
{
R R_y R_x^3
}.
\]

To ensure holomorphicity after transitioning to the chart \(\{R_x \neq 0\}\) and after clearing denominators involving \(R_x\) and $R$, it is necessary that \(R_y\) divides the numerator in \(\mathbb{C}[x,y]\):
\[
R_y \;\Big\vert\; 
3\,A_0\,R
+
A_1\,R_x
+
B_0
\big[
R_y^2\,R_{xx}
-
2\,R_y\,R_x\,R_{xy}
+
R_x^2\,R_{yy}
\big]
\quad \text{in } \mathbb{C}[x,y].
\]

Since the first two terms inside the brackets are already divisible by \(R_y\), this condition reduces to
\[
R_y
\,\,\Big\vert\,\,
3\,A_0\,R
+
A_1\,R_x
+
B_0\,
R_x^2\,R_{yy}
\quad \text{in } \mathbb{C}[x,y].
\]

Next, the coefficients of \({R'}^2 y'\) and \({R'}^3\) in \(J_{xR}\) are respectively:
\[
\begin{aligned}
\big[ {R'}^2 y' \big] \big( J_{xR} \big)
&= 
\frac{
-\,3\,A_0\,R^2
-
2\,A_1\,R\,R_x
-
A_2\,R_x^2
+
2\,B_0
\big[
R_y R R_x R_{xy}
-
R_y^2 R R_{xx}
\big]
}{
R^2 R_y^2 R_x^3
}, \\[4mm]
\big[ {R'}^3 \big] \big( J_{xR} \big)
&= 
\frac{
A_0\,R^3
+
A_1\,R^2 R_x
+
A_2\,R R_x^2
+
A_3\,R_x^3
+
B_0
\big[
R_y^2 R^2 R_{xx}
-
R_y^2 R R_x^2
\big]
}{
R^3 R_y^3 R_x^3
}.
\end{aligned}
\]

After clearing the admissible powers of \(R_x\) and $R$ from the denominators, the holomorphicity conditions on the chart \(\{R_x \neq 0\}\) translate into the following divisibility requirements in \(\mathbb{C}[x,y]\):
\[
\begin{aligned}
R_y^2 &\;\Big\vert\; 
-\,3\,A_0\,R^2
-
2\,A_1\,R\,R_x
-
A_2\,R_x^2
+
2\,B_0
\big[
R_y R R_x R_{xy}
-
R_y^2 R R_{xx}
\big]
\quad \text{in } \mathbb{C}[x,y], \\[4mm]
R_y^3 &\;\Big\vert\; 
A_0\,R^3
+
A_1\,R^2 R_x
+
A_2\,R R_x^2
+
A_3\,R_x^3
+
B_0
\big[
R_y^2 R^2 R_{xx}
-
R_y^2 R R_x^2
\big]
\quad \text{in } \mathbb{C}[x,y].
\end{aligned}
\]

In general, for any \(m \geqslant 3\), returning to the expression of \(J_{xR}\) given in \eqref{best J_X widehatR inhomogeneous} and applying the jet bundle transition — which replaces \(x'\) and $x''$ via formula~\eqref{equ:x' and x''}  — we obtain an expansion of the form
\[
J_{xR}
= x^t \sum_{i+j+3k=m}
    \Bigl(\frac{R'}{R}\Bigr)^{i}\,(y')^{j}\,
    \bigl(\Delta_{Ry}'''\bigr)^k
    \frac{
      \Eq_{i,j,k}\bigl( K_{j,k}; j_{x,y}^m R \bigr)
        }{
        (R_y)^i \; R_x^{\,m-2k}
        },
\]
where \(\Eq_{i,j,k}\) are certain complicated polynomials in the coefficient functions \(K_{j,k}\) appearing in the initial expression of \(J_{xR}\) (see \eqref{best J_X widehatR inhomogeneous}) and in the \(m\)-jet \(j_{x,y}^m R(x,y)\).

Since division by \(R_x^{\,m-2k}\) is allowed on the open set \(\{R_x \neq 0\}\), the only remaining constraints arise from the powers \((R_y)^i\) appearing in the denominator. Holomorphicity on \(\{R_x \neq 0\}\) therefore requires, for every triple \((i,j,k) \in \mathbb{N}^3\) with \(i+j+3k = m\),
\[
(R_y)^i \;\Big\vert\; \Eq_{i,j,k}\bigl( K_{j,k}; j_{x,y}^m R \bigr) \quad \text{in } \mathbb{C}[x,y].
\]

Having chosen the defining polynomial to be of the Fermat-type perturbation
\[
R := x^{d} + y^{d} + x^{6} + 1,
\]
for which \(R_y = d y^{d-1}\), these divisibility conditions become
\[
(y^{d-1})^i \;\Big\vert\; \Eq_{i,j,k}\bigl( K_{j,k}; j_{x,y}^m R \bigr) \quad \text{in } \mathbb{C}[x,y].
\]

We then expand each \(\Eq_{i,j,k}\) as a polynomial in \(y\) and discard all terms of degree at least \((d-1)i\), keeping only the remainder modulo \(y^{(d-1)i}\). This yields the truncated expression
\[
\Eq_{i,j,k}^{\text{(trunc)}} := \Eq_{i,j,k} \pmod{y^{(d-1)i}},
\]
i.e., the polynomial obtained by removing any factor of \(y^{(d-1)i}\) from \(\Eq_{i,j,k}\). All remaining terms must vanish identically for the original divisibility condition to hold, thereby converting the geometric constraints into an explicit linear system.

Next, we proceed inductively with respect to \(i = 1, 2, \dots, m\) (starting naturally from \(i = 1\)). For each fixed \(i\), we collect all divisibility equations sharing the same power of \(R_y\), namely:
\begin{equation}\label{divisibility equations log}
    (y^{d-1})^i \;\Big\vert\; \Eq_{i,j,k}^{\text{(trunc)}} \quad \text{in } \mathbb{C}[x,y], \qquad \text{for all } (j,k) \text{ with } j+3k = m-i.
\end{equation}

The computer then extracts, for each triple \((i,j,k) \in \mathbb{N}^3\) with \(j+3k = m-i\), all coefficients
\[
\bigl[ x^p y^q \bigr] \bigl( \Eq_{i,j,k}^{\text{(trunc)}} \bigr), \qquad q \leqslant (d-1)i-1,
\]
and sets them to zero in order to satisfy the divisibility conditions~\eqref{divisibility equations log}. Here, the notation \([x^p y^q](\cdot)\) denotes the coefficient of the monomial \(x^p y^q\) in the polynomial expansion. These vanishing conditions yield a linear system that must be satisfied by the coefficients of the polynomials \(\{K_{j,k}\}\) appearing in the local expression of the jet differential.

For each fixed \(i\), this linear system is solved symbolically. The solutions are stored in memory before proceeding to the next value of \(i\), up to \(i=m\). At each stage, intermediate feedback — such as the number of nonzero terms remaining in the full expansions of \(\{K_{j,k}\}\) — is monitored to track progress.

It remains to check the holomorphicity of the jet differential~\eqref{best J_X widehatR inhomogeneous} along the divisor \(\{T = 0\}\) at infinity.

To this end, we rewrite the polynomials \(K_{j,k} \in \mathbb{C}[x,y]\) appearing in \eqref{best J_X widehatR inhomogeneous} as homogeneous polynomials $\widehat{K}_{j,k} \in \mathbb{C}[T,X,Y]$, where each \(\widehat{K}_{j,k}(T,X,Y)\) satisfies
\[
\deg \widehat{K}_{j,k} = (m-2k)(d-1) - m + j + k - t,
\]
for \(m \leqslant 14\), \(k = 0, \dots, 4\), and \(j = 0, \dots, m-3k\).

Then, as in \eqref{divisibility by the powers of T log}, we must verify the following divisibility conditions in the polynomial ring \(\mathbb{C}[T,X,Y]\):
\begin{equation*}
T^{m-3q-p} \;\Big\vert\; \HEq_{p,q}(\widehat{K}_{j,k}, T, X)
\quad \text{in } \mathbb{C}[T,X,Y], \qquad \text{for every } q = 0, \dots, 4 \text{ and } p = 0, \dots, m-3q,
\end{equation*}
where \(\HEq_{p,q}(\widehat{K}_{j,k}, T, X)\) are certain complicated polynomials in the coefficient functions \(\widehat{K}_{j,k}\) (from the initial expression of \(J_{X \widehat{R}}\) in \eqref{best J_X widehatR}) and in \(T, X\). Explicitly, as in \eqref{divisibility by the powers of T log},
\begin{align*}
\HEq_{p,q}(\widehat{K}_{j,k}, T, X) 
=&
\sum_{k=q}^{\lfloor m/3\rfloor} \sum_{r=\max (0,\ p+q-k)}^{\min (p,\ m-3k)} \sum_{j=r}^{m-3k}
\\
&
(-1)^{-3k} 2^{p-r} d^{j+k-p-q}
\binom{j}{r} \frac{k!}{(p-r)!\,q!\,(k+r-p-q)!}
\, X^{m-q-j-2k} T^{k-q} \widehat{R}_Y^{2(k-q)} \widehat{K}_{j,k}.
\end{align*}


We then expand each \(\HEq_{p,q}\) as a polynomial in \(T\) and discard all terms of degree at least \(m-3q-p\), keeping only the remainder modulo \(T^{m-3q-p}\). This yields the truncated expression
\[
\HEq_{p,q}^{\text{(trunc)}} := \HEq_{p,q} \pmod{T^{m-3q-p}},
\]
i.e., the polynomial obtained from \(\HEq_{p,q}\) by removing any factor of \(T^{m-3q-p}\). For the original divisibility condition to hold, all remaining terms must vanish identically, thereby converting the geometric constraints into another explicit linear system.

Next, we proceed inductively with respect to \(i = 1, 2, \dots, m\) (starting naturally from \(i = 1\)). For each fixed \(i\), we collect all divisibility equations sharing the same power of \(T\), namely:
\begin{equation}\label{divisibility by the powers of T equations log}
T^{i} \;\Big\vert\; \HEq_{p,q}^{\text{(trunc)}} \quad \text{in } \mathbb{C}[T,X,Y], \qquad \text{for all } (p,q) \text{ with } p + 3q = m - i.
\end{equation}

The computer then extracts, for each triple \((i,p,q) \in \mathbb{N}^3\) with \(p+3q = m-i\), all coefficients
\[
\bigl[ T^{\alpha} X^{\beta} Y^{\gamma} \bigr] \bigl( \HEq_{p,q}^{\text{(trunc)}} \bigr), \qquad \alpha \leqslant m-3q-p-1 = i-1,
\]
and sets them to zero in order to satisfy the divisibility conditions~\eqref{divisibility by the powers of T equations log}. Here, the notation \([T^{\alpha} X^{\beta} Y^{\gamma}](\cdot)\) denotes the coefficient of the monomial \(T^{\alpha} X^{\beta} Y^{\gamma}\) in the polynomial expansion. These vanishing conditions yield a linear system that must be satisfied by the coefficients of the polynomials \(\{\widehat{K}_{j,k}\}\) in \eqref{best J_X widehatR} (which correspond to the coefficients of \(\{K_{j,k}\}\) in \eqref{best J_X widehatR inhomogeneous}) appearing in the local expression of the jet differential.

Finally, combining the two linear systems obtained from the divisibility conditions in \(y\) and in \(T\), a computer algebra verification shows that all coefficients of \(\{K_{j,k}\}\) are forced to vanish. This confirms that for the specified pairs \((m,t)\) with \(d = 12, 13\), only the trivial solution exists. The Key Vanishing Lemma~\ref{lem-1.2} for a generic curve \(\mathcal{C} \subset \mathbb{P}^2\) then follows by a standard semicontinuity argument.

For execution, a single Maple file handles all cases in the logarithmic setting.\footnote{The Maple code for the logarithmic case is available at \url{https://xiesongyan.github.io/}. The parameters \((m, t)\) and \(d\) at the beginning of the code can be adjusted as needed to verify the required cases.} \qed

\begin{rem}
\label{surprise-1}
\rm
  Our computer-assisted computations yield several nonvanishing examples that further confirm the reliability of the code.
  For the parameter \((m,t)=(3,1)\), the dimension of
  \(H^{0}\bigl(\mathbb{P}^{2},\;E_{2,3}T_{\mathbb{P}^{2}}^{*}(\log \mathcal{C})\otimes\mathcal{O}_{\mathbb{P}^{2}}(-1)\bigr)\) is:
  \begin{itemize}
    \item \(1\) for the degree \(7\) Fermat curve \(X^{7}+Y^{7}+Z^{7}=0\);
    \item \(2\) for the degree \(12\) curve \(X^{12}+Y^{12}+Z^{12}+X^{6}Z^{6}=0\);
    \item \(1\) for the degree \(11\) curve \(X^{11}+Y^{11}+Z^{11}+X^{5}Z^{6}=0\).
  \end{itemize}
  These non‑trivial sections are striking at first sight, but we have inspected them and verified that they satisfy all required conditions without error.
  The existence of such sections supports the correctness of our algorithm: if every test had returned zero, one might have suspected a systematic mistake.
  Instead, the code correctly detects negatively twisted invariant jet differentials whenever they are present.
\end{rem}

\section{\bf Proof of Key Vanishing Lemma \ref{lem-1.1}}
\label{sect:3}

\subsection{Changes of Affine Chart in Jet Bundles of Orders $1$ and $2$}

Let $X \subset \mathbb{P}^3$ be a smooth, generic algebraic surface of degree $d$. 
We fix homogeneous coordinates $[T:X:Y:Z]$ on $\mathbb{P}^3$. 
Although the surface and a coordinate share the symbol $X$, 
this standard convention causes no ambiguity and is maintained throughout. 
In the standard affine chart $\{T \neq 0\} \cong \mathbb{C}^3$ with local coordinates $(x, y, z) = (X/T, Y/T, Z/T)$, 
the surface $X$ is defined by the zero set of a dehomogenized polynomial $R(x, y, z) \in \mathbb{C}[x, y, z]$ of degree $d$.

The complement of this affine chart is the hyperplane at infinity $\mathbb{P}^2_\infty := \{T=0\}$. To understand the behavior of jets on this divisor, one must consider alternative affine charts covering it. These are obtained by dividing by the other homogeneous coordinates; for instance, in the chart where $X \neq 0$, we set $X=1$ and introduce new coordinates
\begin{equation}\label{X=1 chart}
\tilde{t} := \frac{T}{X} = \frac{1}{x},\qquad
\tilde{y} := \frac{Y}{X} = \frac{y}{x},\qquad
\tilde{z} := \frac{Z}{X} = \frac{z}{x},
\end{equation}
so that the transformation from this new chart $(\tilde{t},\tilde{y},\tilde{z})$ to the original chart $(x,y,z)$ is given by
\begin{align}\label{1-x-chart}
x = \frac{1}{\tilde{t}},\qquad 
y = \frac{\tilde{y}}{\tilde{t}},\qquad 
z = \frac{\tilde{z}}{\tilde{t}}.
\end{align}
The charts $Y=1$ and $Z=1$ yield analogous transformations.

In the affine chart $\mathbb{C}^3 \subset \mathbb{P}^3$, the smoothness of $X \subset \mathbb{P}^3$ is equivalent to the condition $$\{R=0\} \cap \{R_x=0\} \cap \{R_y=0\} \cap \{R_z=0\} = \emptyset.$$ In addition to the smoothness condition, we require $R$ to be such that these four divisors are in general position, i.e., any  intersection of them has the expected (co)dimension.

In this work, we will focus on jets of order $2$. For an (entire) holomorphic curve parametrized by
\[
\zeta \longmapsto \big( x(\zeta),\, y(\zeta),\, z(\zeta) \big),
\]
the associated jet coordinates are denoted by
\[
x',\, y',\, z',\qquad x'',\, y'',\, z''.
\]

Since the curve lies in the surface $X$, substituting its parametrization into the defining equation yields the identity
\[
R\big(x(\zeta),\,y(\zeta),\,z(\zeta)\big) \equiv 0.
\]
Differentiating this relation with respect to $\zeta$ provides one constraint among first-order jets:
\begin{equation}
    \label{relation1}
x' R_x + y' R_y + z' R_z = 0,
\end{equation}
and a second differentiation gives one constraint among second-order jets:
\begin{equation}
    \label{relation2}
\begin{split}
0 = &\; x'' R_x + y'' R_y + z'' R_z \\
    &\; + {x'}^2 R_{xx}
     + 2 x' y' R_{xy}
     + 2 x' z' R_{xz}
     + {y'}^2 R_{yy} \\
    &\; + 2 y' z' R_{yz}
     + {z'}^2 R_{zz}.
\end{split}
\end{equation}

Solving the first relation~\eqref{relation1} for $y'$ yields:
\leqnomode\usetagform{default}
\begin{align}
\label{y-prime-formula}
y'
:= -\,\frac{R_x}{R_y}\,x'
   -\,\frac{R_z}{R_y}\,z'.
\end{align}

Solving the second relation~\eqref{relation2} for $y''$ gives:
\[
\begin{split}
y''
&= -\,\frac{R_x}{R_y}\,x''
   -\,\frac{R_z}{R_y}\,z'' \\
  &\quad -\,\frac{R_{xx}}{R_y}\,{x'}^2
     - 2\,\frac{R_{xy}}{R_y}\,x'y'
     - 2\,\frac{R_{xz}}{R_y}\,x'z' \\
  &\quad - \frac{R_{yy}}{R_y}\,{y'}^2
     - 2\,\frac{R_{yz}}{R_y}\,y'z'
     - \frac{R_{zz}}{R_y}\,{z'}^2,
\end{split}
\]
and after substituting the expression~\eqref{y-prime-formula} for $y'$, this becomes:
\leqnomode\usetagform{default}
\begin{align}
\label{y-second-formula}
\begin{split}
y''
&= -\,\frac{R_x}{R_y}\,x''
   -\,\frac{R_z}{R_y}\,z'' \\[2mm]
  &\quad + {x'}^2\,
     \bigg(
        -\,\frac{R_{xx}}{R_y}
        + 2\,\frac{R_x}{R_y}\,\frac{R_{xy}}{R_y}
        - \Bigl(\frac{R_x}{R_y}\Bigr)^{\!2} \frac{R_{yy}}{R_y}
     \bigg) \\[2mm]
  &\quad + 2 x' z'\,
     \bigg(
        -\,\frac{R_{xz}}{R_y}
        + \frac{R_z}{R_y}\,\frac{R_{xy}}{R_y}
        + \frac{R_x}{R_y}\,\frac{R_{yz}}{R_y}
        - \frac{R_x}{R_y}\,\frac{R_z}{R_y}\,\frac{R_{yy}}{R_y}
     \bigg) \\[2mm]
  &\quad + {z'}^2\,
     \bigg(
        -\,\frac{R_{zz}}{R_y}
        + 2\,\frac{R_z}{R_y}\,\frac{R_{yz}}{R_y}
        - \Bigl(\frac{R_z}{R_y}\Bigr)^{\!2} \frac{R_{yy}}{R_y}
     \bigg).
\end{split}
\end{align}

These two formulas for $y'$ and $y''$ admit a geometric interpretation. Denote by $J^2$ the bundle of $2$-jets of local holomorphic curves $\mathbb{C} \to X$. On the open subset of the affine chart where the partial derivative with respect to $y$ does not vanish,
\[
X \cap \mathbb{C}^3 \cap \{R_y \neq 0\},
\]
the holomorphic implicit function theorem allows us to solve locally for $y$ as a function of $(x,z)$. Hence, over this set, the surface $X$ is locally graphed as a holomorphic map $(x,z) \mapsto (x, \function(x,z), z)$ over the $(x,z)$-plane, where we deliberately omit introducing a specific symbol for this resolution $y = \function(x,z)$.

Consequently, in this open set $X \cap \mathbb{C}^3 \cap \{R_y \neq 0\}$, the $2$-jet bundle $J^2$ can be coordinatized by
\[
\big( x,\, z,\; x',\, z',\; x'',\, z'' \big).
\]

Likewise, on the complementary open set $X \cap \mathbb{C}^3 \cap \{R_z \neq 0\}$, one can solve locally for $z$ in terms of $(x,y)$. Accordingly, $J^2$ is coordinatized there by
\[
\big( x,\, y,\; x',\, y',\; x'',\, y'' \big).
\]
In a similar vein, we can also obtain the following formulas:
\begin{equation}\label{cpt z'}
    z'=- \frac{R_{x}}{R_{z}}  x'- \frac{R_{y}}{R_{z}}  y',
\end{equation}
\begin{align}\label{cpt z''}
\begin{split}
z''
&= -\,\frac{R_x}{R_z}\,x''
   -\,\frac{R_y}{R_z}\,y'' \\[2mm]
  &\quad + {x'}^2\,
     \bigg(
        -\,\frac{R_{xx}}{R_z}
        + 2\,\frac{R_x}{R_z}\,\frac{R_{xz}}{R_z}
        - \Bigl(\frac{R_x}{R_z}\Bigr)^{\!2} \frac{R_{zz}}{R_z}
     \bigg) \\[2mm]
  &\quad + 2 x' y'\,
     \bigg(
        -\,\frac{R_{xy}}{R_z}
        + \frac{R_y}{R_z}\,\frac{R_{xz}}{R_z}
        + \frac{R_x}{R_z}\,\frac{R_{yz}}{R_z}
        - \frac{R_x}{R_z}\,\frac{R_y}{R_z}\,\frac{R_{zz}}{R_z}
     \bigg) \\[2mm]
  &\quad + {y'}^2\,
     \bigg(
        -\,\frac{R_{yy}}{R_z}
        + 2\,\frac{R_y}{R_z}\,\frac{R_{yz}}{R_z}
        - \Bigl(\frac{R_y}{R_z}\Bigr)^{\!2} \frac{R_{zz}}{R_z}
     \bigg).
\end{split}
\end{align}

For brevity, in what follows we will omit the explicit intersection with $\mathbb{C}^3$ and simply write $X \cap \{R_y \neq 0\}$ and $X \cap \{R_z \neq 0\}$, it being understood that all considerations are restricted to the affine chart $\mathbb{C}^3 \subset \mathbb{P}^3$.

\begin{Observation}
The two formulas~\eqref{y-prime-formula} and~\eqref{y-second-formula} express the change of coordinates for the $2$-jet bundle $J^2$ between the two charts:
\[
\xymatrix{
J^2\big\vert_{\{R_z\neq 0\}} \ar[rr] & & J^2\big\vert_{\{R_y\neq 0\}}.
}
\]

Conversely, the formulas~\eqref{cpt z'} and~\eqref{cpt z''} give the inverse change of coordinates, expressing the transition from the \(\{R_y \neq 0\}\) chart back to the \(\{R_z \neq 0\}\) chart:
\[
\xymatrix{
J^2\big\vert_{\{R_y\neq 0\}} \ar[rr] & & J^2\big\vert_{\{R_z\neq 0\}}.
}
\]
\end{Observation}


The Wronskian $y'x'' - y''x'$ expressed in the chart $X\cap \{R_z \neq 0\}$ transforms to the chart $X\cap \{R_y \neq 0\}$ by substituting $y'$ from~\eqref{y-prime-formula} and $y''$ from~\eqref{y-second-formula}:
\[
\aligned
\begin{vmatrix}
y' & x'
\\
y'' & x''
\end{vmatrix}
&
\,=\,
-\,
\def\arraystretch{1.25}
\begin{vmatrix}
\frac{R_x}{R_y}x'+\frac{R_z}{R_y}z' & x'
\\
\frac{R_x}{R_y}x''+\frac{R_z}{R_y}z'' & x''
\end{vmatrix}
\\
&
\ \ \ \ \
+
\def\arraystretch{1.25}
\begin{vmatrix}
0
& 
x'
\\
{x'}^2
\big(
-\,\frac{R_{xx}}{R_y}
+
2\,\frac{R_x}{R_y}\,\frac{R_{xy}}{R_y}
-
\frac{R_x}{R_y}\,\frac{R_x}{R_y}\,\frac{R_{yy}}{R_y}
\big)
& 
x''
\end{vmatrix}
\\
&
\ \ \ \ \
+
\def\arraystretch{1.25}
\begin{vmatrix}
0
& 
x'
\\
2x'z'
\big(
-\,\frac{R_{xz}}{R_y}
+
\frac{R_z}{R_y}\,\frac{R_{xy}}{R_y}
+
\frac{R_x}{R_y}\,\frac{R_{yz}}{R_y}
-
\frac{R_x}{R_y}\,\frac{R_z}{R_y}\,\frac{R_{yy}}{R_y}
\big)
& 
x''
\end{vmatrix}
\\
&
\ \ \ \ \
+
\def\arraystretch{1.25}
\begin{vmatrix}
0
& 
x'
\\
z'z'
\big(
-\,\frac{R_{zz}}{R_y}
+
2\,\frac{R_z}{R_y}\,\frac{R_{yz}}{R_y}
-
\frac{R_z}{R_y}\,\frac{R_z}{R_y}\,\frac{R_{yy}}{R_y}
\big)
& 
x''
\end{vmatrix},
\endaligned
\]
by expanding, by reorganizing, and by letting the
new Wronskian $z'x'' - z''x'$ appear:
\leqnomode\usetagform{default}
\begin{align}
\label{Delta-yx-formula}
\begin{vmatrix}
y' & x'
\\
y'' & x''
\end{vmatrix}
&
\,=\,
-\,
\frac{R_z}{R_y}\,
\begin{vmatrix}
z' & x'
\\
z'' & x''
\end{vmatrix}
\\
&
\ \ \ \ \
+
z'z'x'\,
\bigg(
\frac{R_{zz}}{R_y}
-
2\,\frac{R_z}{R_y}\,\frac{R_{yz}}{R_y}
+
\frac{R_z}{R_y}\,\frac{R_z}{R_y}\,\frac{R_{yy}}{R_y}
\bigg)
\notag
\\
&
\ \ \ \ \
+
z'x'x'\,
\bigg(
2\,\frac{R_{xz}}{R_y}
-
2\,\frac{R_z}{R_y}\,\frac{R_{xy}}{R_y}
-
2\,\frac{R_x}{R_y}\,\frac{R_{yz}}{R_y}
+
2\,\frac{R_x}{R_y}\,\frac{R_z}{R_y}\,\frac{R_{yy}}{R_y}
\bigg)
\notag
\\
&
\ \ \ \ \
+
x'x'x'\,
\bigg(
\frac{R_{xx}}{R_y}
-
2\,\frac{R_x}{R_y}\,\frac{R_{xy}}{R_y}
+
\frac{R_x}{R_y}\,\frac{R_x}{R_y}\,\frac{R_{yy}}{R_y}
\bigg).
\notag
\end{align}

This formula expresses the Wronskian $y'x''-y''x'$ in the $X\cap \{R_z \neq 0\}$ chart as a combination of the new Wronskian $z'x''-z''x'$ and purely algebraic terms (involving only first-order jets) that are holomorphic in the $X\cap \{R_y \neq 0\}$ chart. The key observation is that the coefficient of the new Wronskian contains a factor $R_z/R_y$, while the remaining terms have denominators involving only $R_y$.
\subsection{Local Description of Invariant Jet Differentials}\label{subsec:local form of 2-jet differentials}

Recall that \(R(x,y,z)\) is the defining polynomial of a generic smooth surface \(X \subset \mathbb{P}^3\) in the affine chart \(\mathbb{C}^3\). By the GAGA principle, a nonzero global holomorphic section of the bundle of \(2\)-jet differentials of weighted degree \(m\) on the surface \(X\) is algebraic, and it is a basic fact of algebraic geometry that, when restricted to the Zariski open set \(X \cap \{R_z \neq 0\}\), such a section admits a local expression of the form
\begin{equation}\label{J_yx of weighted degree m}
J_{yx} \coloneqq \sum_{\substack{0 \leqslant k \leqslant \lfloor m/3 \rfloor \\ 0 \leqslant j \leqslant m-3k}}
\frac{F_{j,k}}{R_z^{f_{j,k}}}\,
(y')^{j}\,(x')^{m-3k-j}\,
\begin{vmatrix} y' & x' \\ y'' & x'' \end{vmatrix}^k,
\end{equation}
where \(R_z = \partial R/\partial z\) is nonvanishing on this set, and \(F_{j,k} \in \mathbb{C}[x,y,z]\) are polynomials. All denominators and numerators are considered modulo \(R\), i.e., as elements of the quotient ring
\begin{equation}\label{mathfrak R}
\mathfrak{R} \coloneqq \mathbb{C}[x,y,z]/(R),
\end{equation}
and consequently the rational functions are regarded as elements of its fraction field. The nonnegative integers \(f_{j,k}\) indicate the order of the factor \(R_z\) in the denominator. 


\begin{Question}
{\sl Can one obtain upper bounds for the integers \(f_{j,k}\)?}
\end{Question}

The answer is affirmative. The following proposition reveals that every global section of the invariant $2$-jet differential bundle necessarily conforms to a prescribed finite form, with undetermined polynomial coefficients.

\begin{pro}\label{Prp-compact-Jyx}
Let $\widehat{R}(T,X,Y,Z)\in \mathbb{C}[T, X, Y, Z]$ be a homogeneous polynomial defining a smooth surface $X \subset \mathbb{P}^3$, and assume that the divisors in $\mathbb{P}^3$ given by $\widehat{R}$ and its partial derivatives $\widehat{R}_T, \widehat{R}_X, \widehat{R}_Y, \widehat{R}_Z$ are in general position. Dehomogenizing with respect to $T=1$ yields an affine polynomial $R(x,y,z) \in \mathbb{C}[x,y,z]$ such that $X \cap \{T \neq 0\} = \{R=0\}$. For any weighted order $m$, every invariant $2$-jet differential of weighted order $m$ on $X$ admits, on $X \cap \{T \neq 0\} \cap \{R_z \neq 0\}$, the general expression
\[
J_{yx} = \sum_{\substack{0 \leqslant k \leqslant \lfloor m/3 \rfloor \\ 0 \leqslant j \leqslant m-3k}}
\frac{F_{j,k}}{R_z^{m-2k}}\,
(y')^{j}\,(x')^{m-3k-j}\,
\begin{vmatrix} y' & x' \\ y'' & x'' \end{vmatrix}^k,
\]
where $F_{j,k} \in \mathbb{C}[x,y,z]$ are polynomials, viewed as elements of the quotient ring $\mathfrak{R} = \mathbb{C}[x,y,z]/(R)$ (hence as regular functions on $X \cap \{T \neq 0\}$).
\end{pro}

\begin{proof}
On the affine part \(X \cap \{R_y \neq 0\}\), by the same reasoning as in~\eqref{J_yx of weighted degree m}, an invariant \(2\)-jet differential of weighted order \(m\) can be expressed as
\begin{equation}\label{J_zx of weighted degree m}
J_{zx} \coloneqq \sum_{\substack{0 \leqslant q \leqslant \lfloor m/3 \rfloor \\ 0 \leqslant p \leqslant m-3q}}
\overline{F}_{p,q}\,
(z')^{p}\,(x')^{m-3q-p}\,
\begin{vmatrix} z' & x' \\ z'' & x'' \end{vmatrix}^q,
\end{equation}
where all coefficients are rational functions whose denominators may contain only powers of \(R_y\). 

We now examine the compatibility of the two expressions~\eqref{J_yx of weighted degree m} and~\eqref{J_zx of weighted degree m} for the jet differential by substituting the transition formulas~\eqref{cpt z'} and~\eqref{cpt z''} into the alternative expression~\eqref{J_zx of weighted degree m}. Under this substitution, the Wronskian \(z'x'' - z''x'\) — initially expressed in the chart \(X \cap \{R_y \neq 0\}\) — transforms to the chart \(X \cap \{R_z \neq 0\}\) after expanding and reorganizing terms so that the Wronskian \(y'x'' - y''x'\) emerges explicitly:
\begin{align}\label{Wronskian_zx}
\begin{vmatrix}
z' & x' \\
z'' & x''
\end{vmatrix}
&= -\frac{R_y}{R_z}\,
   \begin{vmatrix}
   y' & x' \\
   y'' & x''
   \end{vmatrix} \\[2mm]
&\quad + \frac{1}{R_z^3}\,
   \bigl(R_{yy}R_z^2 - 2R_zR_yR_{yz} + R_y^2R_{zz}\bigr)
   y'y'x' \notag \\[2mm]
&\quad + \frac{2}{R_z^3}\,
   \bigl(R_{xy}R_z^2 - R_zR_yR_{xz} - R_xR_zR_{yz} + R_xR_yR_{zz}\bigr)
   y'x'x' \notag \\[2mm]
&\quad + \frac{1}{R_z^3}\,
   \bigl(R_{xx}R_z^2 - 2R_xR_zR_{xz} + R_x^2R_{zz}\bigr)
   x'x'x'. \notag
\end{align}
Comparing corresponding terms, we observe that the coefficients in \eqref{J_yx of weighted degree m} can be expressed as rational functions whose denominators involve powers of \(R_y\) (which must ultimately cancel) and powers of \(R_z\). The exponents of \(R_z\) appearing in the denominators arise from the following sources:
\begin{itemize}
    \item Each of the \(p\) factors of \(z'\) contributes at most one factor of \(R_z^{-1}\) via \eqref{cpt z'}.
    \item Each of the \(q\) factors of the Wronskian \(\big\vert \begin{smallmatrix} z' & x' \\ z'' & x'' \end{smallmatrix} \big\vert\) contributes at most one factor of \(R_z^{-1}\) from the leading term in \eqref{Wronskian_zx} or, when expanded further, contributes at most one factor of \(R_z^{-3}\) from the cubic terms.
\end{itemize}

Moreover, the term \((y')^{j}\,(x')^{m-3k-j}\,
\big\vert\begin{smallmatrix} y' & x' \\ y'' & x'' \end{smallmatrix}\big\vert^k\) in expression~\eqref{J_yx of weighted degree m} arises from expanding terms of the form
\begin{equation}\label{terms in J_zx with q>=k, m-3q-p<=m-3k-j}
    (z')^{p}\,(x')^{m-3q-p}\,
    \begin{vmatrix} z' & x' \\ z'' & x'' \end{vmatrix}^q \quad\text{with } q \geqslant k.
\end{equation}
When expanding such a term, \(\big\vert \begin{smallmatrix} z' & x' \\ z'' & x'' \end{smallmatrix} \big\vert^q\) contributes a factor of \(\big\vert \begin{smallmatrix} y' & x' \\ y'' & x'' \end{smallmatrix} \big\vert^k\) together with \(k\) powers of \(R_z^{-1}\) from the leading term in \eqref{Wronskian_zx} and an additional \(q-k\) powers of \(R_z^{-3}\) from the remaining cubic terms. The factor \((z')^{p}\) contributes \(p\) powers of \(R_z^{-1}\) via \eqref{cpt z'}. Hence, the total exponent of \(R_z\) in the denominator of the coefficient multiplying \((y')^{j}\,(x')^{m-3k-j}\,
\big\vert\begin{smallmatrix} y' & x' \\ y'' & x'' \end{smallmatrix}\big\vert^k\) is at most
\[
k + 3(q-k) + p = 3q + p - 2k \leqslant m - 2k.
\]

This upper bound \(m-2k\) can be achieved by expanding the term in \eqref{terms in J_zx with q>=k, m-3q-p<=m-3k-j} corresponding to \(q = k\) and \(p = m-3k\), namely \((z')^{m-3k}\,\big\vert\begin{smallmatrix} z' & x' \\ z'' & x'' \end{smallmatrix}\big\vert^k\).

By our assumption on \(\widehat{R}\), the divisors defined by \(R\) and its partial derivatives \(R_x, R_y, R_z\) are in general position. Under this condition, after cancelling common factors between numerator and denominator, the maximal possible pole order \(f_{j,k}\) of the coefficients in \eqref{J_yx of weighted degree m} with respect to \(R_z\) satisfies
\[
f_{j,k} \leqslant m-2k.
\]
Moreover, since the expression~\eqref{J_yx of weighted degree m} must be regular on \(X \cap \{R_z \neq 0\}\), any factor of \(R_y\) in the denominator must cancel. Consequently, the coefficients \(F_{j,k}\) may be taken as polynomials, and we may set \(f_{j,k} = m-2k\) as the maximal admissible exponent.
\end{proof}

\subsection{Choice of the Surface Equation}\label{subsec:Choice of the Surface Equation}

We now set the degree of the smooth surface \(X \subset \mathbb{P}^3\) to be
\[
\deg X = d,
\]
with \(d \in \mathbb{N}\). Its defining equation is taken as a perturbation of the Fermat hypersurface:
\begin{equation}\label{homogeneous equation}
    \widehat{R} \coloneqq X^d + Y^d + Z^d + T^d + X^a T^b = 0,
\end{equation}
where the positive integers \(a, b\) satisfy \(a + b = d\) and will be chosen later with \(\lvert a-b \rvert\) small (e.g. \(a-b = 0\) or \(1\)). The specific choice of the exponents \(a, b\) will be guided by computer experiments in Maple.

\begin{Convention}\label{convention cpt}
In the remainder of the proof, we work under the following generic assumptions:\footnote{The equation \(\widehat{R}\) to be adopted later has parameters \(a = 8\) and \(d = 17\), satisfying Convention~\ref{convention cpt}. The associated Maple code is available at \url{https://github.com/LeiHou-code/Hyperbolicity-Dimension-2-Conventions-Facts-2026}.}
\begin{itemize}
    \item The surface \(X\) defined by \(\widehat{R} = 0\) is smooth.
    \item The five divisors defined by the homogeneous polynomial $\widehat{R}$ and its partial derivatives \(\widehat{R}_X, \widehat{R}_Y, \widehat{R}_Z, \widehat{R}_T\) are in general position; i.e., any intersection of them has the expected (co)dimension.
    \item The coordinate hyperplanes \(\{Z=0\}\) and \(\{T=0\}\) intersect \(X\) transversely.
\end{itemize}
\end{Convention}

In the affine chart \(T = 1\), this defining equation becomes
\[
R(x,y,z) = 1 + x^{d} + y^{d} + z^{d} + x^a,
\]
where the exponent \(a\) will eventually be chosen as \(\lfloor d/2 \rfloor\). This particular form yields two crucial simplifications. First, the relevant partial derivative simplifies to
\[
R_z = d\,z^{d-1},
\]
which is independent of \(x\) and \(y\) up to the harmless constant factor \(d\). Second, since the extra term \(X^a T^b\) involves no variable \(Y\), we obtain the relation
\[
y^d = -1 - x^d - z^d - x^a,
\]
which allows for an efficient reduction procedure in the quotient ring \(\mathfrak{R} = \mathbb{C}[x,y,z]/(R)\).

Recall that for an algebraic invariant \(2\)-jet differential initially defined on \(X\cap\{R_z \neq 0\}\) to be also holomorphic on \(X\cap\{R_y \neq 0\}\), certain divisibility conditions by \(R_z\) must be satisfied — constraints already encountered in the transition formulas. These conditions are understood in the quotient ring \(\mathfrak{R}\). A polynomial \(F \in \mathbb{C}[x,y,z]\) is said to be divisible by \(R_z\) in this quotient ring if there exists \(G \in \mathbb{C}[x,y,z]\) such that \(F \equiv R_z \cdot G \pmod{R}\), i.e., their difference lies in the ideal generated by \(R\). Geometrically, this means that as a rational function on the surface \(X\), \(F\) vanishes on \(X \cap \{R_z = 0\}\) with at least the same multiplicity as \(R_z\).

To test divisibility by \(R_z = d z^{d-1}\) in \(\mathfrak{R}\), we employ an algorithmic reduction based on the relation \(y^d = -1 - x^d - z^d - x^a\). Every element of \(\mathfrak{R}\) can be uniquely represented by a polynomial in \(\mathbb{C}[x,y,z]\) with \(\deg_y < d\). Indeed, for any monomial \(y^k\) with \(k \geqslant d\), writing \(k = dq + r\) where \(0 \leqslant r < d\) via Euclidean division, we may replace
\[
y^k = (y^d)^q \cdot y^r = (-1 - x^d - z^d - x^a)^q \cdot y^r,
\]
which involves only powers of \(y\) with exponent \(r < d\). Applying this substitution to each monomial yields a reduced representative \(\operatorname{red}(F) \in \mathbb{C}[x,y,z]\) satisfying \(\deg_y \operatorname{red}(F) < d\) and \(\operatorname{red}(F) \equiv F \pmod{R}\).

The uniqueness of this reduced form follows from the observation that if two polynomials \(F_1, F_2\) with \(\deg_y < d\) satisfy \(F_1 \equiv F_2 \pmod{R}\), then their difference is a multiple of \(R\). Any nonzero multiple of \(R\) must have \(y\)-degree at least \(d\), since \(R\) itself contains the monomial \(y^d\) with coefficient \(1\) and no other term can cancel it after multiplication. Hence \(F_1 - F_2 = 0\), establishing uniqueness.

\begin{lem}\label{Lemma-divisibility}
Let $R(x,y,z) = 1 + x^d + y^d + z^d + S(x)$, where $S(x) \in \mathbb{C}[x]$ is a polynomial of degree strictly less than $d$. Assume that the surface $X \subset \mathbb{P}^3$ defined by the homogenization of $R$ is smooth and that the coordinate hyperplane $\{Z=0\}$ intersects $X$ transversely. For any polynomial $F \in \mathbb{C}[x,y,z]$, let $\operatorname{red}(F)$ denote its unique normal form satisfying $\deg_y \operatorname{red}(F) < d$, obtained by Euclidean reduction using the relation $y^d = -1 - x^d - z^d - S(x)$. Then, for any positive integer $i$, the polynomial $F$ is divisible by $(R_z)^i = (d z^{d-1})^i$ in the quotient ring $\mathfrak{R} = \mathbb{C}[x,y,z]/(R)$ if and only if its reduction $\operatorname{red}(F)$ is divisible by $(z^{d-1})^i$ in $\mathbb{C}[x,y,z]$.
\end{lem}

\begin{proof}
Since $\{Z=0\}$ intersects $X$ transversely, $z$ is not a zero divisor in $\mathfrak{R}$.

($\Leftarrow$) If $\operatorname{red}(F) = (z^{d-1})^i H$ for some $H \in \mathbb{C}[x,y,z]$, then 
$F \equiv \operatorname{red}(F) \equiv (z^{d-1})^i H \pmod{R}$. 
Since $(R_z)^i = d^i (z^{d-1})^i$, we have $(z^{d-1})^i H = (R_z)^i \cdot \bigl(\frac{1}{d^i} H\bigr)$, so 
$F \equiv (R_z)^i \cdot \bigl(\frac{1}{d^i} H\bigr) \pmod{R}$, proving that $F$ is divisible by $(R_z)^i$ in $\mathfrak{R}$.

($\Rightarrow$) Suppose $F \equiv (R_z)^i G \pmod{R}$ for some $G \in \mathbb{C}[x,y,z]$. 
Let $G_0 = \operatorname{red}(G)$ be the normal form of $G$. Then $(R_z)^i G \equiv (R_z)^i G_0 \pmod{R}$, 
and $(R_z)^i G_0 = d^i (z^{d-1})^i G_0$ is already in normal form (since $\deg_y G_0 < d$ and multiplication by powers of $z$ does not affect the $y$-degree). 
By uniqueness of the normal form, we must have $\operatorname{red}(F) = d^i (z^{d-1})^i G_0 = (z^{d-1})^i (d^i G_0)$, 
which is clearly divisible by $(z^{d-1})^i$ in $\mathbb{C}[x,y,z]$.
\end{proof}

Consequently, the divisibility conditions required for holomorphicity on $X\cap\{R_y \neq 0\}$ translate into an explicit algebraic constraint: after reducing each numerator polynomial $F$ to its normal form $\operatorname{red}(F)$, the coefficients of $z^0, z^1, \dots, z^{d-2}$ in $\operatorname{red}(F)$ must all vanish identically.

This reduction crucially relies on our specific choice of the defining equation. By contrast, if one starts with a genuinely generic homogeneous polynomial of degree $d$ — rather than the special form $X^d + Y^d + Z^d + T^d + X^a T^b$ — the quotient ring $\mathfrak{R}$ becomes significantly more complicated. Its structure no longer admits a simple degree reduction in any variable, and checking divisibility by $R_z$ would require computing normal forms with respect to a Gröbner basis. Such computations are notoriously intensive and quickly become infeasible even on powerful computer algebra systems.

\subsection{Denominator Exponent Constraints from Holomorphicity}\label{subsec:Pole Orders at Infinity, cpt}

To analyze the behavior of jets on the plane at infinity $\mathbb{P}_\infty^2 := \{T=0\}$, we will work in the affine chart $\{X\neq 0\}$. Introducing the coordinates $(\tilde{t}, \tilde{y}, \tilde{z})$ as in~\eqref{1-x-chart}, the divisor $\mathbb{P}_\infty^2$ is then locally defined by $\{\tilde{t}=0\}$.

Recall that the jets under consideration are jets of entire holomorphic curves $\mathbb{C} \to \mathbb{P}^3$, or of local holomorphic curves $\mathbb{D} \to \mathbb{C}^3 \subset \mathbb{P}^3$, where $\mathbb{D} \subset \mathbb{C}$ is the unit disc:
\[
\mathbb{D} \ni \zeta \longmapsto \big( x(\zeta),\, y(\zeta),\, z(\zeta) \big) \in \mathbb{C}^3.
\]

Under the change of chart $(\tilde{t},\tilde{y},\tilde{z}) \mapsto (x,y,z)$ given by \eqref{X=1 chart}, 
the $1$-jets and $2$-jets transform according to the following formulas (expressing the original jet coordinates in terms of the new ones):
\begin{equation}\label{jets from tilde(t,y,z) to (x,y,z)}
\begin{aligned}
\tilde{t}' &= -\frac{x'}{x^2}, &
\tilde{y}' &= \frac{y'}{x} - \frac{y\,x'}{x^2}, 
\\
\tilde{t}'' &= -\frac{x''}{x^2} + 2\frac{(x')^2}{x^3}, &
\tilde{y}'' &= \frac{y''}{x} - 2\frac{x'\,y'}{x^2} - \frac{y\,x''}{x^2} + 2\frac{y\,(x')^2}{x^3}.
\end{aligned}
\end{equation}
And the Wronskian transforms as:
\begin{equation}\label{Wronskian from tilde(y,t) to (y,x)}
\begin{vmatrix}
\tilde{y}' & \tilde{t}' \\
\tilde{y}'' & \tilde{t}''
\end{vmatrix}
=
-\frac{1}{x^3}\,
\begin{vmatrix}
y' & x' \\
y'' & x''
\end{vmatrix}.
\end{equation}

Similarly, under the change of chart $(x,y,z) \mapsto (\tilde{t},\tilde{y},\tilde{z})$ given by \eqref{1-x-chart}, 
the $1$-jets and $2$-jets transform according to the following formulas:
\begin{equation}\label{jets from (x,y,z) to tilde(t,y,z)}
\begin{aligned}
x' &= -\frac{\tilde{t}'}{\tilde{t}^2}, &
y' &= \frac{\tilde{y}'}{\tilde{t}} - \frac{\tilde{y}\,\tilde{t}'}{\tilde{t}^2}, 
\\
x'' &= -\frac{\tilde{t}''}{\tilde{t}^2} + 2\frac{(\tilde{t}')^2}{\tilde{t}^3}, &
y'' &= \frac{\tilde{y}''}{\tilde{t}} - 2\frac{\tilde{t}'\,\tilde{y}'}{\tilde{t}^2} - \frac{\tilde{y}\,\tilde{t}''}{\tilde{t}^2} + 2\frac{\tilde{y}\,(\tilde{t}')^2}{\tilde{t}^3}.
\end{aligned}
\end{equation}

Consequently, the Wronskian transforms as:
\begin{equation}\label{Wronskian from (y,x) to tilde(y,t)}
\begin{vmatrix}
y' & x' \\
y'' & x''
\end{vmatrix}
= -\frac{1}{\tilde{t}^3}\,
\begin{vmatrix}
\tilde{y}' & \tilde{t}' \\
\tilde{y}'' & \tilde{t}''
\end{vmatrix}.
\end{equation}

\medskip
Now consider a general term appearing in $J_{yx}$ as described above in Proposition~\ref{Prp-compact-Jyx}.
Since the polynomials $F_{j,k}$ above will be restricted to the surface defined by $R=0$, by the same reduction argument preceding Lemma~\ref{Lemma-divisibility}, we may assume without loss of generality the following:

\begin{Convention}
\label{convention} All polynomials $F_{j,k}$ are assumed to have $y$-degree strictly less than $d$; equivalently, every monomial appearing in them satisfies $\deg_y < d$.
\end{Convention}

Furthermore, on $X \cap \{ T \neq 0 \} \cap \{ \widehat{R}_Z \neq 0 \}$, every invariant $2$-jet differential of weighted order $m$ on $X$ admits the general expression
\begin{equation}\label{J_YX}
J_{YX} \coloneqq \sum_{\substack{0 \leqslant k \leqslant \lfloor m/3 \rfloor \\ 0 \leqslant j \leqslant m-3k}}
\frac{\widehat{G}_{j,k}}{T^{g_{j,k}} \widehat{R}_Z^{m-2k}}\,
(y')^{j}\,(x')^{m-3k-j}\,
\begin{vmatrix} y' & x' \\ y'' & x'' \end{vmatrix}^k,
\end{equation}
where $\widehat{G}_{j,k}(T,X,Y,Z)$ are homogeneous polynomials in $\mathbb{C}[T,X,Y,Z]$ satisfying
\begin{equation*}
    \deg \widehat{G}_{j,k} = g_{j,k} + (m-2k)(d-1)
    \quad
    \text{and}
    \quad
    \deg_Y \widehat{G}_{j,k} < d,
\end{equation*}
and $\widehat{R}_Z$ is the partial derivative of the homogeneous polynomial $\widehat{R}$ as in Convention \ref{convention cpt}. Here the exponent of \(\widehat{R}_Z\) in the denominator is taken to be \(m-2k\), which coincides with the maximal admissible exponent of \(R_z\) in the denominators of the coefficients \(F_{j,k} / R_z^{m-2k}\) from Proposition~\ref{Prp-compact-Jyx}.

Similarly, on $X \cap \{ X \neq 0 \} \cap \{ \widehat{R}_Z \neq 0 \}$,  every invariant $2$-jet differential of weighted order $m$ on $X$ admits the general expression
\begin{equation}\label{J_TY}
J_{TY}
= 
\sum_{\substack{0 \leqslant k \leqslant \lfloor m/3 \rfloor \\ 0 \leqslant \ell \leqslant m-3k}} 
\frac{\widehat{E}_{\ell ,k}}{X^{e_{\ell ,k}} \widehat{R}_{Z}^{m-2k}}
(\tilde{t}')^{m-3k-\ell}
(\tilde{y}')^{\ell}
\begin{vmatrix}
\tilde{y}' & \tilde{t}' \\
\tilde{y}'' & \tilde{t}''
\end{vmatrix}^k,
\end{equation}
where $\widehat{E}_{\ell,k}(T,X,Y,Z)$ are homogeneous polynomials in $\mathbb{C}[T,X,Y,Z]$ satisfying
\begin{equation*}
    \deg \widehat{E}_{\ell,k} = e_{\ell,k} + (m-2k)(d-1)
    \quad
    \text{and}
    \quad
    \deg_Y \widehat{E}_{\ell,k} < d.
\end{equation*}

\begin{Question}
{\sl What happens to the general form $J_{TY}$ given in \eqref{J_TY} after the change of affine chart from $X \cap \{ X \neq 0 \} \cap \{ \widehat{R}_Z \neq 0 \}$ to $X \cap \{ T \neq 0 \} \cap \{ \widehat{R}_Z \neq 0 \}$?}
\end{Question}

Substituting \eqref{jets from tilde(t,y,z) to (x,y,z)} and \eqref{Wronskian from tilde(y,t) to (y,x)} into \eqref{J_TY}, the term transforms as follows:
\[
\begin{aligned}
&
\frac{\widehat{E}_{\ell ,k}}{X^{e_{\ell ,k}} \widehat{R}_{Z}^{m-2k}}
(\tilde{t}')^{m-3k-\ell}
(\tilde{y}')^{\ell}
\begin{vmatrix}
\tilde{y}' & \tilde{t}' \\
\tilde{y}'' & \tilde{t}''
\end{vmatrix}^k
\\
=
&
\frac{\widehat{E}_{\ell ,k}}{X^{e_{\ell ,k}} \widehat{R}_{Z}^{m-2k}}
\big(-\frac{x'}{x^2}\big)^{m - 3k - \ell}\,
\big(\frac{y'}{x} - \frac{y\,x'}{x^2}\big)^{\ell}
\big(-\frac{1}{x^3}\big)^k
\begin{vmatrix} y' & x' \\ y'' & x'' \end{vmatrix}^k
\\
=
&
\frac{1}{x^{2(m-3k)+3k}} 
\sum_{j=0}^{\ell}
(-1)^{m-2k-j}
\frac{\widehat{E}_{\ell ,k}}{X^{e_{\ell ,k}} \widehat{R}_{Z}^{m-2k}}
\binom{\ell }{j} x^{j} y^{\ell -j}
(y')^{j}\,(x')^{m-3k-j}\,
\begin{vmatrix} y' & x' \\ y'' & x'' \end{vmatrix}^k,
\end{aligned}
\]
where the last equality is obtained by expanding $(xy'-yx')^{\ell}$ 
and simplifying. 

Consequently, the general form transforms as
\begin{align*}
J_{TY}
&= 
\sum_{k=0}^{\lfloor m/3 \rfloor} \sum_{\ell =0}^{m-3k}
\frac{1}{x^{2(m-3k)+3k}} 
\sum_{j=0}^{\ell}
(-1)^{m-2k-j}
\frac{\widehat{E}_{\ell ,k}}{X^{e_{\ell ,k}} \widehat{R}_{Z}^{m-2k}}
\binom{\ell }{j} x^{j} y^{\ell -j}
(y')^{j}\,(x')^{m-3k-j}\,
\begin{vmatrix} y' & x' \\ y'' & x'' \end{vmatrix}^k
\\
&=
\sum_{k=0}^{\lfloor m/3 \rfloor} \sum_{j=0}^{m-3k} 
\frac{(-1)^{m-2k-j}}{x^{2(m-3k)+3k-j}}
\sum_{\ell=j}^{m-3k} 
\frac{\widehat{E}_{\ell ,k}}{X^{e_{\ell ,k}} \widehat{R}_{Z}^{m-2k}}
\binom{\ell }{j} y^{\ell -j}
(y')^{j}\,(x')^{m-3k-j}\,
\begin{vmatrix} y' & x' \\ y'' & x'' \end{vmatrix}^k
\\
&=
\sum_{k=0}^{\lfloor m/3 \rfloor}
\sum_{j=0}^{m-3k} 
(-1)^{m-2k-j} 
\frac{T^{2(m-3k)+3k-j}}{X^{2(m-3k)+3k-j}}
\sum_{\ell=j}^{m-3k} 
\frac{\widehat{E}_{\ell ,k}}{X^{e_{\ell ,k}} \widehat{R}_{Z}^{m-2k}}
\binom{\ell }{j} \frac{Y^{\ell -j}}{T^{\ell -j}}
(y')^{j}\,(x')^{m-3k-j}\,
\begin{vmatrix} y' & x' \\ y'' & x'' \end{vmatrix}^k 
\\
&=
\sum_{k=0}^{\lfloor m/3 \rfloor}
\frac{T^{m}}{X^{m}}
\sum_{j=0}^{m-3k} 
(-1)^{m-2k-j} 
\sum_{\ell=j}^{m-3k} 
\frac{\widehat{E}_{\ell ,k}}{X^{e_{\ell ,k}} \widehat{R}_{Z}^{m-2k}}
\binom{\ell }{j} \frac{T^{m-3k- \ell} Y^{\ell -j}}{X^{m-3k-j}}
(y')^{j}\,(x')^{m-3k-j}\,
\begin{vmatrix} y' & x' \\ y'' & x'' \end{vmatrix}^k.
\end{align*}
Note that the coefficients
\begin{equation*}
\frac{T^{m}}{X^{m}}
(-1)^{m-2k-j} 
\sum_{\ell=j}^{m-3k} 
\frac{\widehat{E}_{\ell ,k}}{X^{e_{\ell ,k}} \widehat{R}_{Z}^{m-2k}}
\binom{\ell }{j} \frac{T^{m-3k- \ell} Y^{\ell -j}}{X^{m-3k-j}},
\end{equation*}
in the last line, as rational functions, not only have no pole along the divisor $X\cap \{ T=0 \}$, but also vanish with multiplicity at least $m$.

Comparing with the coefficients of \eqref{J_YX}, it follows that, on $X \cap \{ T \neq 0 \} \cap \{ \widehat{R}_Z \neq 0 \}$, every invariant $2$-jet differential of weighted order $m$ on $X$ admits the general expression
\begin{equation*}
J_{YX}
=
\sum_{\substack{0 \leqslant k \leqslant \lfloor m/3 \rfloor \\ 0 \leqslant j \leqslant m-3k}}
\frac{T^{m} \widehat{H}_{j,k}}{\widehat{R}_{Z}^{m-2k}}
(y')^{j}\,(x')^{m-3k-j}\,
\begin{vmatrix} y' & x' \\ y'' & x'' \end{vmatrix}^k,
\end{equation*}
where $\widehat{H}_{j,k}(T,X,Y,Z)$ are homogeneous polynomials in $\mathbb{C}[T,X,Y,Z]$ satisfying
\begin{equation*}
    \deg \widehat{H}_{j,k} = (m-2k)(d-1)-m
    \quad
    \text{and}
    \quad
    \deg_Y \widehat{H}_{j,k} < d.
\end{equation*}

To determine the explicit vanishing order of the jet differential as~\eqref{J_YX} along the divisor $\{ x=0 \} \cap X$, we need the following lemma.

\begin{lem}
The rational function 
$\frac{T^{m} \widehat{H}_{j,k}}{\widehat{R}_{Z}^{m-2k}}$ 
vanishes along the divisor $X \cap \{ x = 0 \}$ with multiplicity at least $t$ 
if and only if 
$\widehat{H}_{j,k} \equiv X^{t} \widehat{K}_{j,k} \pmod{\widehat{R}}$
for some homogeneous polynomial $\widehat{K}_{j,k}$ satisfying $\deg_Y \widehat{K}_{j,k} < d$.
\end{lem}

\begin{proof}
($\Leftarrow$) If $\widehat{H}_{j,k} = X^{t} K_{j,k}$ in $\mathbb{C}[T,X,Y,Z] / (\widehat{R})$, then as a rational function, it clearly vanishes to order at least $t$ along $\{x = 0\}$. By our choice of the surface equation~\eqref{homogeneous equation}, $\widehat{R}_Z$ does not vanish identically on $X \cap \{ x = 0 \}$, so the denominator introduces no poles or zero-canceling factors along $X \cap \{ x = 0 \}$, preserving the vanishing order.

(\(\Rightarrow\)) Conversely, suppose the rational function \({T^{m} \widehat{H}_{j,k}}/{\widehat{R}_{Z}^{m-2k}}\) vanishes to order at least \(t\) along \(X \cap \{ x = 0 \}\). The statement admits a proof by induction on \(t\); therefore, it suffices to establish the base case \(t = 1\).

Since \(T^{m}\) and \(\widehat{R}_Z\) do not vanish identically on \(X \cap \{ x = 0 \}\) (by our choice of the surface equation~\eqref{homogeneous equation}), it follows that \(\widehat{H}_{j,k}\) itself vanishes along \(X \cap \{ x = 0 \}\); i.e., \(\widehat{H}_{j,k} \equiv 0 \pmod{(X, \widehat{R})}\). By Convention~\ref{convention}, \(\widehat{H}_{j,k}\) has a unique representative with \(\deg_Y < d\); we may therefore assume that \(\widehat{H}_{j,k}\) is already in this normal form, satisfying \(\deg_Y \widehat{H}_{j,k} < d\). Consider its restriction \(\widehat{H}_{j,k}(T,0,Y,Z) \in \mathbb{C}[T,Y,Z]\) to the hyperplane \(\{ X = 0 \}\). This restriction satisfies:
\begin{itemize}
    \item \(\widehat{H}_{j,k}(T,0,Y,Z) \equiv \widehat{H}_{j,k} \equiv 0 \pmod{(X, \widehat{R})}\);
    \item \(\deg_Y \widehat{H}_{j,k}(T,0,Y,Z) \leqslant \deg_Y \widehat{H}_{j,k} < d\).
\end{itemize}
By our choice of the surface equation~\eqref{homogeneous equation}, the first condition implies \(\widehat{H}_{j,k}(T,0,Y,Z) \equiv 0 \pmod{(T^d + Y^d + Z^d)}\) in \(\mathbb{C}[T,Y,Z]\). However, the second condition — namely \(\deg_Y < d\) — prevents \(\widehat{H}_{j,k}(T,0,Y,Z)\) from being a nonzero multiple of \(T^d + Y^d + Z^d\), since the latter has \(Y\)-degree exactly \(d\). Hence the restriction must be identically zero: \(\widehat{H}_{j,k}(T,0,Y,Z) = 0\) in \(\mathbb{C}[T,Y,Z]\). It follows that \(\widehat{H}_{j,k} \equiv X \widehat{K}_{j,k} \pmod{\widehat{R}}\) for some homogeneous polynomial \(\widehat{K}_{j,k}\) satisfying \(\deg_Y \widehat{K}_{j,k} < d\).
\end{proof}

If we now require the jet differential as~\eqref{J_YX} to vanish to order at least $t \geqslant 1$ along the divisor $\{ x=0 \} \cap X$ on the surface $X$, i.e.,
\[
\Twist := -t \leqslant -1,
\]
then this corresponds to considering sections of the jet differential bundle twisted by $\mathcal{O}_{X}(-t)$, i.e.,
\[
(\cdot) \otimes \mathcal{O}_{X}(-t).
\]

Therefore, on $X \cap \{ T \neq 0 \} \cap \{ \widehat{R}_Z \neq 0 \}$, every invariant $2$-jet differential of weighted order $m$ and vanishing order at least $t$ along $\{ x=0 \} \cap X$ on the surface $X$ admits the general expression
\begin{equation}\label{best J}
J_{YX}
=
\sum_{\substack{0 \leqslant k \leqslant \lfloor m/3 \rfloor \\ 0 \leqslant j \leqslant m-3k}}
\frac{T^{m} X^{t} \widehat{K}_{j,k}}{\widehat{R}_{Z}^{m-2k}}
(y')^{j}\,(x')^{m-3k-j}\,
\begin{vmatrix} y' & x' \\ y'' & x'' \end{vmatrix}^k,
\end{equation}
where $\widehat{K}_{j,k}(T,X,Y,Z)$ are homogeneous polynomials in $\mathbb{C}[T,X,Y,Z]$ satisfying
\begin{equation*}
    \deg \widehat{K}_{j,k} = (m-2k)(d-1)-m-t
    \quad
    \text{and}
    \quad
    \deg_Y \widehat{K}_{j,k} < d.
\end{equation*}

When the coefficients of the above expression are expressed as regular functions in the affine coordinates \((x,y,z)\), we obtain
\begin{equation}\label{best J_yx}
J_{yx}
\coloneqq
x^{t} 
\sum_{\substack{0 \leqslant k \leqslant \lfloor m/3 \rfloor \\ 0 \leqslant j \leqslant m-3k}}
\frac{K_{j,k}}{R_{z}^{m-2k}}
(y')^{j}\,(x')^{m-3k-j}\,
\begin{vmatrix} y' & x' \\ y'' & x'' \end{vmatrix}^k,
\end{equation}
where $K_{j,k}(x,y,z)$ are polynomials in $\mathbb{C}[x,y,z]$ satisfying
\begin{equation*}
    \deg K_{j,k} \leqslant (m-2k)(d-1)-m-t
    \quad
    \text{and}
    \quad
    \deg_y K_{j,k} < d.
\end{equation*}

For our later applications, we only need weighted orders up to \(m = 11\). For these low weights, the above expression simplifies considerably. Indeed, the exponent \(k\) of the Wronskian \(\big\vert \begin{smallmatrix} y' & x' \\ y'' & x'' \end{smallmatrix} \big\vert\) satisfies \(k \leqslant \lfloor m/3 \rfloor \leqslant 3\). Terms involving a power of the Wronskian greater than \(\lfloor m/3 \rfloor\) are understood to be identically zero. 


It remains to check the holomorphicity of our jet differential~\eqref{best J} along the divisor at infinity $X \cap \{T = 0\}$.

\begin{Question}
{\sl What happens to the general form $J_{YX}$ given in \eqref{best J} after the change of affine chart~\eqref{1-x-chart}?}
\end{Question}

Substituting \eqref{jets from (x,y,z) to tilde(t,y,z)} and \eqref{Wronskian from (y,x) to tilde(y,t)} into \eqref{best J}, the term transforms as follows:
\[
\begin{aligned}
&
\frac{T^{m} X^{t} \widehat{K}_{j,k}}{\widehat{R}_Z^{m-2k}}
(y')^{j}\,(x')^{m-3k-j}\,
\begin{vmatrix} y' & x' \\ y'' & x'' \end{vmatrix}^k
\\
=
&
\frac{T^{m} X^{t} \widehat{K}_{j,k}}{\widehat{R}_Z^{m-2k}}
\,
\big(\frac{\tilde{y}'}{\tilde{t}} - \frac{\tilde{y}\,\tilde{t}'}{\tilde{t}^2}\big)^j\,
\big(-\frac{\tilde{t}'}{\tilde{t}^2}\big)^{m - 3k -j}
\begin{vmatrix} y' & x' \\ y'' & x'' \end{vmatrix}^k 
\\
=
&
\frac{(-1)^k}{\tilde{t}^{2(m-3k) + 3k}}
\frac{T^{m} X^{t} \widehat{K}_{j,k}}{\widehat{R}_Z^{m-2k}}
\,(\tilde{t}\,\tilde{y}' - \tilde{t}'\,\tilde{y})^j\,
(-\tilde{t}')^{m-3k-j}
\begin{vmatrix}
\tilde{y}' & \tilde{t}' \\
\tilde{y}'' & \tilde{t}''
\end{vmatrix}^k
\\
=
&
\frac{1}{\tilde{t}^{2(m-3k) + 3k}}
\sum_{\ell=0}^{j}
(-1)^{m-2k-\ell}
\frac{T^{m} X^{t} \widehat{K}_{j,k}}{\widehat{R}_Z^{m-2k}}
\binom{j}{\ell}
\tilde{t}^{\ell}
\tilde{y}^{j-\ell}
(\tilde{t}')^{m-3k-\ell}
(\tilde{y}')^{\ell}
\begin{vmatrix}
\tilde{y}' & \tilde{t}' \\
\tilde{y}'' & \tilde{t}''
\end{vmatrix}^k,
\end{aligned}
\]
where the last equality is obtained by expanding $(\tilde{t}\,\tilde{y}' - \tilde{t}'\,\tilde{y})^j$ and simplifying. 

Consequently, the general form transforms as
\begin{align*}
J_{yx} 
&= 
\sum_{k=0}^{\lfloor m/3 \rfloor} \sum_{j=0}^{m-3k} 
\frac{1}{\tilde{t}^{2(m-3k) + 3k}}
\sum_{\ell=0}^{j}
(-1)^{m-2k-\ell}
\frac{T^{m} X^{t} \widehat{K}_{j,k}}{\widehat{R}_Z^{m-2k}}
\binom{j}{\ell}
\tilde{t}^{\ell}
\tilde{y}^{j-\ell}
(\tilde{t}')^{m-3k-\ell}
(\tilde{y}')^{\ell}
\begin{vmatrix}
\tilde{y}' & \tilde{t}' \\
\tilde{y}'' & \tilde{t}''
\end{vmatrix}^k
\\
&=
\sum_{k=0}^{\lfloor m/3 \rfloor} \sum_{\ell=0}^{m-3k}  
\frac{(-1)^{m-2k-\ell}}{\tilde{t}^{2(m-3k) + 3k-\ell}}
\sum_{j=\ell}^{m-3k}
\frac{T^{m} X^{t} \widehat{K}_{j,k}}{\widehat{R}_Z^{m-2k}}
\binom{j}{\ell}
\tilde{y}^{j-\ell}
(\tilde{t}')^{m-3k-\ell}
(\tilde{y}')^{\ell}
\begin{vmatrix}
\tilde{y}' & \tilde{t}' \\
\tilde{y}'' & \tilde{t}''
\end{vmatrix}^k
\\
&=
\sum_{k=0}^{\lfloor m/3 \rfloor} \sum_{\ell=0}^{m-3k}  
(-1)^{m-2k-\ell}
\frac{X^{2(m-3k) + 3k-\ell}}{T^{2(m-3k) + 3k-\ell}}
\sum_{j=\ell}^{m-3k}
\frac{T^{m} X^{t} \widehat{K}_{j,k}}{\widehat{R}_Z^{m-2k}}
\binom{j}{\ell}
\frac{Y^{j-\ell}}{X^{j-\ell}}
(\tilde{t}')^{m-3k-\ell}
(\tilde{y}')^{\ell}
\begin{vmatrix}
\tilde{y}' & \tilde{t}' \\
\tilde{y}'' & \tilde{t}''
\end{vmatrix}^k
\\
&=
\sum_{k=0}^{\lfloor m/3 \rfloor}
\sum_{\ell=0}^{m-3k} 
(-1)^{m-2k-\ell} 
\sum_{j=\ell}^{m-3k} 
\frac{X^{m} X^{t} \widehat{K}_{j,k}}{\widehat{R}_Z^{m-2k}}
\binom{j}{\ell} \frac{X^{m-3k-j} Y^{j-\ell}}{T^{m-3k-\ell}}
(\tilde{t}')^{m-3k-\ell}
(\tilde{y}')^{\ell}
\begin{vmatrix}
\tilde{y}' & \tilde{t}' \\
\tilde{y}'' & \tilde{t}''
\end{vmatrix}^k.
\end{align*}

To guarantee the holomorphicity of our jet differential~\eqref{best J} along the divisor at infinity $X \cap \{T = 0\}$, each coefficient
\begin{align*}
&
(-1)^{m-2k-\ell} 
\sum_{j=\ell}^{m-3k} 
\frac{X^{m} X^{t} \widehat{K}_{j,k}}{\widehat{R}_Z^{m-2k}}
\binom{j}{\ell} \frac{X^{m-3k-j} Y^{j-\ell}}{T^{m-3k-\ell}}
\\
=&
(-1)^{m-2k-\ell} 
\frac{X^m X^t}{\widehat{R}_Z^{m-2k}}
\frac{\sum_{j=\ell}^{m-3k} \widehat{K}_{j,k} \binom{j}{\ell} X^{m-3k-j} Y^{j-\ell}}{T^{m-3k-\ell}},
\end{align*}
must be regular on $X \cap \{ X \neq 0 \} \cap \{ \widehat{R}_Z \neq 0 \}$. Then the holomorphicity conditions translate into the following divisibility requirements in the homogeneous quotient ring $\mathbb{C}[T,X,Y,Z] / (\widehat{R})$:
\begin{equation}\label{divisibility by the powers of T cpt}
T^{m-3k-\ell} \; \Big\vert \; \sum_{j=\ell}^{m-3k} \widehat{K}_{j,k} \binom{j}{\ell} X^{m-3k-j} Y^{j-\ell} 
\quad \text{ for each } k=0, \dots , \lfloor \tfrac{m}{3} \rfloor \text{ and } \ell=0, \dots , m-3k.
\end{equation}

\subsection{End of the Proof of the Key Vanishing Lemma~\ref{lem-1.1}}

For a smooth surface $X \subset \mathbb{P}^3$, we consider the bundle $E_{2,m}T_X^\ast$ of invariant $2$-jet differentials and its negative twists:
\[
E_{2,m}T_X^\ast \otimes \mathcal{O}_{X}(-t),
\]
with integer $t \geqslant 1$.

Take the affine equation of the surface $X$ to be a Fermat-type polynomial with a single additional monomial $x^a$:
\[
\begin{aligned}
0 &= z^{d} + y^{d} + x^{d} +1 + x^a
  &=: R(x,y,z),
\end{aligned}
\]
with
\[
d := 17 
, \qquad a = \left\lfloor \tfrac{d}{2} \right\rfloor,
\]
so that its projectivization $\widehat{R}(T,X,Y,Z) \coloneqq X^d + Y^d + Z^d + T^d + X^a T^b$ has degree $d = 17$. One can verify the smoothness of $X$ (including along $X \cap \{T=0\}$ and $X \cap \{X=0\}$) as well as the general position of the five divisors $\widehat{R}, \widehat{R}_X, \widehat{R}_Y, \widehat{R}_Z, \widehat{R}_T$; we omit these routine computations.

For instance, when $m=3$, by Proposition~\ref{Prp-compact-Jyx}, an invariant $2$-jet differential of weighted order $3$ on $X$ takes the following most general local form on \(X \cap \{R_z \neq 0\}\):
\begin{equation*}
J_{yx}
= 
\sum_{0\leqslant j\leqslant 3} \frac{A_j}{R_z^3}\,(y')^{j}\,(x')^{3-j}
+ 
\frac{B_0}{R_z}\,
\begin{vmatrix} y' & x' \\ y'' & x'' \end{vmatrix}.
\end{equation*}
After applying the jet bundle transition — which consists in replacing $y'x'' - y''x'$ by formula~\eqref{Delta-yx-formula} and $y'$ by formula~\eqref{y-prime-formula} — the transition formula becomes:
\[
\begin{aligned}
J_{yx}
&= -\frac{B_0}{R_y}\,
   \begin{vmatrix}
   z' & x' \\
   z'' & x''
   \end{vmatrix} \\
  &\quad + z'z'z'
     \bigg( -\frac{A_3}{R_y^3} \bigg) \\
  &\quad + z'z'x'
     \bigg( -3\,\frac{A_3 R_x}{R_z^1 R_y^3}
           + \frac{A_2}{R_z^1 R_y^2}
           + B_0\,
             \Big[ \frac{R_{zz}}{R_z R_y}
                   - 2\frac{R_z}{R_y^2}\frac{R_{yz}}{R_z}
                   + \frac{R_z}{R_y}\frac{R_z}{R_y}\frac{R_{yy}}{R_z R_y}
             \Big]
     \bigg) \\
  &\quad + z'x'x'
     \bigg( -3\,\frac{A_3 R_x^2}{R_z^2 R_y^3}
           + 2\frac{A_2 R_x}{R_z^2 R_y^2}
           - \frac{A_1}{R_z^2 R_y} \\
  &\qquad\qquad
           + B_0\,
             \Big[ 2\frac{R_{xz}}{R_z R_y}
                   - 2\frac{R_x}{R_y}\frac{R_{yz}}{R_z R_y}
                   - 2\frac{R_z}{R_y}\frac{R_{xy}}{R_z R_y}
                   + 2\frac{R_x}{R_y}\frac{R_z}{R_y}\frac{R_{yy}}{R_z R_y}
             \Big]
     \bigg) \\
  &\quad + x'x'x'
     \bigg( -\frac{A_3 R_x^3}{R_z^3 R_y^3}
           + \frac{A_2 R_x^2}{R_z^3 R_y^2}
           - \frac{A_1 R_x}{R_z^3 R_y}
           + \frac{A_0}{R_z^3} \\
  &\qquad\qquad
           + B_0\,
             \Big[ \frac{R_{xx}}{R_z R_y}
                   - 2\frac{R_x}{R_y}\frac{R_{xy}}{R_z R_y}
                   + \frac{R_x}{R_y}\frac{R_x}{R_y}\frac{R_{yy}}{R_z R_y}
             \Big]
     \bigg).
\end{aligned}
\]

In the first two lines, only $R_y$ appears in denominators — which is allowed on the open set $\{R_y \neq 0\}$ — so no constraints are imposed on $B_0$ or $A_3$ from these terms.

On the third line, however, $R_z$ appears to the first power in denominators. More precisely, the coefficient of $z'z'x'$ in $J_{yx}$, which we denote by $\big[ z'z'x' \big] \big( J_{yx} \big)$, factors as:
\[
\big[ z'z'x' \big] \big( J_{yx} \big)
= \frac{
   -3A_3 R_x + A_2 R_y + B_0\big[ R_y^2 R_{zz} - 2R_y R_z R_{yz} + R_z^2 R_{yy} \big]
   }{
   R_z R_y^3
   }.
\]

To ensure the holomorphicity after transitioning to the set $\{R_y \neq 0\}$ and after clearing out $R_y$, it is necessary that $R_z$ divides the numerator:
\[
R_z \;\Big|\; -3A_3 R_x + A_2 R_y + B_0\bigl[ R_y^2 R_{zz} - 2R_y R_z R_{yz} + R_z^2 R_{yy} \bigr] \quad \text{in}\quad \mathfrak{R}=\mathbb{C}[x,y,z]/(R).
\]

Since the last two terms inside the brackets are already divisible by \(R_z\), this condition reduces to
\[
R_z \;\Big|\; -3A_3 R_x + A_2 R_y + B_0 R_y^2 R_{zz} \quad \text{in}\quad \mathfrak{R}.
\]

Next, the coefficients of $z'x'x'$ and $x'x'x'$ in $J_{yx}$ are respectively:
\[
\begin{aligned}
\big[ z'x'x' \big] \big( J_{yx} \big)
&= \frac{
   -3A_3 R_x^2 + 2A_2 R_x R_y - A_1 R_y^2
   + 2B_0\big[ R_y^2 R_z R_{xz} - R_x R_y R_z R_{yz} - R_y R_z^2 R_{xy} + R_x R_z^2 R_{yy} \big]
   }{
   R_z^2 R_y^3
   }, \\[4mm]
\big[ x'x'x' \big] \big( J_{yx} \big)
&= \frac{
   -A_3 R_x^3 + A_2 R_x^2 R_y - A_1 R_x R_y^2 + A_0 R_y^3
   + B_0\big[ R_y^2 R_z^2 R_{xx} - 2R_x R_y R_z^2 R_{xy} + R_x^2 R_z^2 R_{yy} \big]
   }{
   R_z^3 R_y^3
   }.
\end{aligned}
\]

After clearing the admissible powers of \(R_y\) from the denominators, the holomorphicity conditions on the chart \(\{R_y \neq 0\}\) translate into the following divisibility requirements in the quotient ring \(\mathfrak{R}\):
\[
\begin{aligned}
R_z^2 &\;\Big\vert\; -3A_3 R_x^2 + 2A_2 R_x R_y - A_1 R_y^2 \\
      &\qquad\; +\; 2B_0\bigl[ R_y^2 R_z R_{xz} - R_x R_y R_z R_{yz} - R_y R_z^2 R_{xy} + R_x R_z^2 R_{yy} \bigr] \quad \text{in}\quad \mathfrak{R}, \\[4mm]
R_z^3 &\;\Big\vert\; -A_3 R_x^3 + A_2 R_x^2 R_y - A_1 R_x R_y^2 + A_0 R_y^3 \\
      &\qquad\; +\; B_0\bigl[ R_y^2 R_z^2 R_{xx} - 2R_x R_y R_z^2 R_{xy} + R_x^2 R_z^2 R_{yy} \bigr] \quad \text{in}\quad \mathfrak{R}.
\end{aligned}
\]

In general, for any $m \geqslant 3$, returning to the expression of $J_{yx}$ given as~\eqref{best J_yx} and applying the jet bundle transition — which replaces $y'x'' - y''x'$ via formula~\eqref{Delta-yx-formula} and $y'$ via formula~\eqref{y-prime-formula} — we obtain an expansion of the form
\[
J_{yx}
= x^t \sum_{i+j+3k=m}
   {x'}^i {z'}^j
   \begin{vmatrix}
   z' & x' \\
   z'' & x''
   \end{vmatrix}^k
   \frac{
      \Eq_{i,j,k}\big( K_{j,k}; j_{x,y,z}^m R \big)
   }{
      (R_z)^i \; R_y^{\,m-2k}
   },
\]
where $\Eq_{i,j,k}$ are certain complicated polynomials in the coefficient functions $K_{j,k}$ appearing in the initial expression of $J_{yx}$ (see~\eqref{best J_yx}) and in the $m$-jet $j_{x,y,z}^m R(x,y,z)$.

Since division by \(R_y^{\,m-2k}\) is allowed on the open set \(\{R_y \neq 0\}\), the only remaining constraints arise from the powers \((R_z)^i\) appearing in the denominator. Holomorphicity on \(\{R_y \neq 0\}\) therefore requires, for every triple \((i,j,k) \in \mathbb{N}^3\) with \(i+j+3k = m\),
\[
(R_z)^i \;\Big\vert\; \Eq_{i,j,k}\bigl( K_{j,k}; j_{x,y,z}^m R \bigr) \quad \text{in}\quad \mathfrak{R}.
\]

Having chosen the defining polynomial to be of the Fermat-type perturbation
\[
R := z^{d} + y^{d} + x^{d} + x^{a} + 1,
\]
for which \(R_z = d z^{d-1}\), these divisibility conditions become
\[
(z^{d-1})^i \;\Big\vert\; \Eq_{i,j,k}\bigl( K_{j,k}; j_{x,y,z}^m R \bigr) \quad \text{in}\quad \mathfrak{R}.
\]

For computational implementation, Lemma~\ref{Lemma-divisibility} guarantees that we need only work with the unique normal form $\operatorname{red}(\Eq_{i,j,k})$ of each $\Eq_{i,j,k}$. This normal form is obtained by repeatedly replacing any factor $y^i$ with $i \geqslant d$ in $\Eq_{i,j,k}$ using the relation
\begin{equation*}
    y^{d} = - (z^{d} + x^{d} + x^{a} + 1),
\end{equation*}
until the resulting polynomial in $\mathbb{C}[x,y,z]$ has degree in $y$ strictly less than $d$.

We then expand each $\operatorname{red}(\Eq_{i,j,k})$ as a polynomial in $z$ and discard all terms of degree $\geqslant (d-1)i$, keeping only the remainder modulo $z^{(d-1)i}$:
\[
\Eq_{i,j,k}^{\text{(trunc)}} := \operatorname{red}(\Eq_{i,j,k}) \pmod{z^{(d-1)i}}.
\]
All remaining terms must vanish identically for the original divisibility condition to hold, thereby converting the geometric constraints into an explicit linear system.

Next, we proceed inductively with respect to \(i = 1, 2, \dots, m\) (starting naturally from \(i = 1\)). For each fixed \(i\), we collect all divisibility equations sharing the same power of \(R_z\), namely:
\begin{equation}\label{divisibility equations in mathbbC[x,y,z]}
    (z^{(d-1)})^i \;\Big\vert\; \Eq_{i,j,k}^{\text{(trunc)}} \quad \text{in } \mathbb{C}[x,y,z], \qquad \text{for all } (j,k) \text{ with } j+3k = m-i.
\end{equation}

The computer then extracts, for each triple \((i,j,k) \in \mathbb{N}^3\) with \(j+3k = m-i\), all coefficients
\[
\bigl[ x^p y^q z^r \bigr] \bigl( \Eq_{i,j,k}^{\text{(trunc)}} \bigr), \qquad r \leqslant (d-1)i-1,
\]
and sets them to zero in order to satisfy the divisibility conditions~\eqref{divisibility equations in mathbbC[x,y,z]}. Here, the notation \([x^p y^q z^r](\cdot)\) denotes the coefficient of the monomial \(x^p y^q z^r\) in the polynomial expansion. These vanishing conditions yield a linear system that must be satisfied by the coefficients of the polynomials \(\{K_{j,k}\}\) appearing in the local expression of the jet differential.

For each fixed $i$, this linear system is solved symbolically. The solutions are stored in memory before proceeding to the next value of $i$, up to $i=m$. At each stage, intermediate feedback — such as the number of nonzero terms remaining in the full expansions of $\{K_{j,k}\}$ — is monitored to track progress.

It remains to check the holomorphicity of the jet differential~\eqref{best J_yx} along the divisor \(\{T = 0\}\) at infinity.

To this end, we rewrite the polynomials \(K_{j,k} \in \mathbb{C}[x,y,z]\) appearing in \eqref{best J_yx} as homogeneous polynomials $\widehat{K}_{j,k} \in \mathbb{C}[T,X,Y,Z]$, where each \(\widehat{K}_{j,k}(T,X,Y,Z)\) satisfies
\[
\deg \widehat{K}_{j,k} = (m-2k)(d-1) - m - t
\quad\text{and}\quad
\deg_Y \widehat{K}_{j,k} < d,
\]
for \(m \leqslant 11\), \(k = 0, \dots, 3\), and \(j = 0, \dots, m-3k\).

Then, as in \eqref{divisibility by the powers of T cpt}, we must verify the following divisibility conditions in the homogeneous quotient ring \(\mathbb{C}[T,X,Y,Z] / (\widehat{R})\):
\begin{equation*}
T^{m-3k-\ell} \;\Big\vert\; \HEq_{\ell,k}(\widehat{K}_{j,k}, X, Y)
\quad \text{in } \mathbb{C}[T,X,Y,Z] / (\widehat{R}), \qquad \text{for every } k = 0, \dots, 3 \text{ and } \ell = 0, \dots, m-3k,
\end{equation*}
where \(\HEq_{\ell,k}(\widehat{K}_{j,k}, X, Y)\) are certain polynomials in the coefficient functions \(\widehat{K}_{j,k}\) (from the initial expression of \(J_{YX}\) in \eqref{best J}) and in \(X, Y\). Explicitly, as in \eqref{divisibility by the powers of T cpt},
\begin{equation*}
\HEq_{\ell,k}(\widehat{K}_{j,k}, X, Y)
=
\sum_{j=\ell}^{m-3k} \widehat{K}_{j,k} \binom{j}{\ell} X^{m-3k-j} Y^{j-\ell}.
\end{equation*}


For computational implementation, Lemma~\ref{Lemma-divisibility} guarantees that we need only work with the unique homogeneous normal form \(\operatorname{red}(\HEq_{\ell,k})\) of each \(\HEq_{\ell,k}\). This homogeneous normal form is obtained by repeatedly replacing any factor \(Y^i\) with \(i \geqslant d\) in \(\HEq_{\ell,k}\) using the relation
\[
Y^{d} = -\bigl(Z^{d} + X^{d} + T^{d-a} X^{a} + T^{d}\bigr),
\]
until the resulting polynomial in \(\mathbb{C}[T,X,Y,Z]\) has degree in \(Y\) strictly less than \(d\).

We then expand each \(\operatorname{red}(\HEq_{\ell,k})\) as a polynomial in \(T\) and discard all terms of degree at least \(m-3k-\ell\), keeping only the remainder modulo \(T^{m-3k-\ell}\). This yields the truncated expression
\[
\HEq_{\ell,k}^{\text{(trunc)}} := \operatorname{red}(\HEq_{\ell,k}) \pmod{T^{m-3k-\ell}},
\]
i.e., the polynomial obtained from \(\operatorname{red}(\HEq_{\ell,k})\) by removing any factor of \(T^{m-3k-\ell}\). For the original divisibility condition to hold, all remaining terms must vanish identically, thereby converting the geometric constraints into another explicit linear system.

Next, we proceed inductively with respect to \(i = 1, 2, \dots, m\) (starting naturally from \(i = 1\)). For each fixed \(i\), we collect all divisibility equations sharing the same power of \(T\), namely:
\begin{equation}\label{divisibility by the powers of T equations cpt}
T^{i} \;\Big\vert\; \HEq_{\ell,k}^{\text{(trunc)}} \quad \text{in } \mathbb{C}[T,X,Y,Z], \qquad \text{for all } (\ell,k) \text{ with } \ell + 3k = m - i.
\end{equation}

The computer then extracts, for each triple \((i,\ell,k) \in \mathbb{N}^3\) with \(\ell+3k = m-i\), all coefficients
\[
\bigl[ T^{s} X^{p} Y^{q} Z^{r} \bigr] \bigl( \HEq_{\ell,k}^{\text{(trunc)}} \bigr), \qquad s \leqslant m-3k-\ell-1 = i-1,
\]
and sets them to zero in order to satisfy the divisibility conditions~\eqref{divisibility by the powers of T equations cpt}. Here, the notation \([T^{s} X^{p} Y^{q} Z^{r}](\cdot)\) denotes the coefficient of the monomial \(T^{s} X^{p} Y^{q} Z^{r}\) in the homogeneous polynomial expansion. These vanishing conditions yield a linear system that must be satisfied by the coefficients of the polynomials \(\{\widehat{K}_{j,k}\}\) in \eqref{best J} (which correspond to the coefficients of \(\{K_{j,k}\}\) in \eqref{best J_yx}) appearing in the local expression of the jet differential.

Finally, combining the two linear systems obtained from the divisibility conditions in \(z\) and in \(T\), a computer algebra verification shows that all coefficients of \(\{K_{j,k}\}\) are forced to vanish. This confirms that for the specified pairs \((m,t)\) with \(d = 17\), only the trivial solution exists. The Key Vanishing Lemma~\ref{lem-1.1} for a generic surface \(X \subset \mathbb{P}^3\) of degree \(d = 17\) then follows by a standard semicontinuity argument.

For execution, a single Maple file handles all cases in the compact setting.\footnote{The Maple code for the compact case is available at \url{https://xiesongyan.github.io/}. The parameters \((m, t)\) and \(d\) at the beginning of the code can be adjusted as needed to verify the required cases.} \qed

\begin{rem}
\label{surprise-2}
\rm
  Our computer-assisted computations for surfaces in \(\mathbb{P}^{3}\) produce several nonvanishing examples that further confirm the reliability of the code.
  For the parameter \((m,t)=(3,1)\), the dimension of
  \(H^{0}\bigl(X,\;E_{2,3}T^{*}_{X}\otimes\mathcal{O}_{X}(-1)\bigr)\) is:
  \begin{itemize}
    \item \(1\) for the degree \(9\) Fermat surface
      \(X^{9}+Y^{9}+Z^{9}+T^{9}=0\);
    \item \(1\) for the degree \(15\) surface
      \(X^{15}+Y^{15}+Z^{15}+T^{15}+X^{7}T^{8}=0\);
    \item \(5\) for the degree \(17\) surface
      \(X^{17}+Y^{17}+Z^{17}+T^{17}+X^{8}T^{9}=0\).
  \end{itemize}
  The presence of these nontrivial sections supports the correctness of our algorithm: if every test had returned only zero, one might have suspected a systematic mistake.
  Instead, the code accurately detects negatively twisted invariant jet differentials whenever they are present.
\end{rem}

\medskip

\section{\bf Applications of Siu's Slanted Vector Fields}
\label{sect:4}

To obtain an additional family of independent negatively twisted jet differentials, we employ Siu's method of slanted vector fields~\cite{Siu2004}. This technique generalizes to $k$-jet spaces the original construction for $1$-jet spaces introduced by Clemens~\cite{Clemens1986}, Ein~\cite{Ein1988}, and Voisin~\cite{Voisin1996} in their study of rational curves.

The method was first implemented in the compact surface case by P\u{a}un~\cite{mihaipaun2008} for surfaces in $\mathbb{P}^3$. Later, Rousseau~\cite{Rousseau2009} established the logarithmic counterpart for pairs $(\mathbb{P}^2, D)$ with curve configurations $D$, and further extended both P\u{a}un's setting and the logarithmic framework to threefolds in~\cite{Rousseau2007cpt,Rousseau2007log}. For hypersurfaces $H \subset \mathbb{P}^n$ of arbitrary dimension, Merker~\cite{Merker2009} achieved the full generalization, thereby realizing Siu's original expectation on generic global generation. The logarithmic counterpart for pairs $(\mathbb{P}^n, H)$ was subsequently established by Darondeau~\cite{Darondeau_vectorfield}.

This section reviews the two-dimensional case of Siu's slanted vector field construction. The presentation follows the classical treatment~\cite{mihaipaun2008, Rousseau2009} and claims no originality.

\subsection{Compact Case}\label{subsec:pre to SVF in cpt case}

Let \(\mathcal{X} \subset \mathbb{P}^3 \times \mathbb{P}^{N_d^3}\) denote the universal family of algebraic surfaces of degree \(d\) in \(\mathbb{P}^3\), where \(N_d^3 \coloneqq \binom{d+3}{3}-1\) is the dimension of the parameter space. It is defined by the homogeneous equation
\[
\sum_{|\alpha| = d} A_{\alpha} Z^{\alpha}=0.
\]
Here \([A] = [A_{\alpha}]_{|\alpha|=d} \in \mathbb{P}^{N_d^3}\) are the coefficients and \([Z] = [Z_0:Z_1:Z_2:Z_3] \in \mathbb{P}^3\) are the homogeneous coordinates. For a multi‑index \(\alpha = (\alpha_0,\alpha_1,\alpha_2,\alpha_3) \in \mathbb{N}^4\) with length \(|\alpha| \coloneqq \alpha_0+\alpha_1+\alpha_2+\alpha_3\), we write \(Z^{\alpha} \coloneqq Z_0^{\alpha_0}Z_1^{\alpha_1}Z_2^{\alpha_2}Z_3^{\alpha_3}\).

The \emph{vertical \(2\)-jet bundle} of the universal family is defined as
\[
J_2^v(\mathcal{X}) \coloneqq \ker\bigl((\mathrm{pr}_2)_*\bigr) \;\subset\; J_2(\mathcal{X}),
\]
where \(\mathrm{pr}_2 \colon \mathcal{X} \to \mathbb{P}^{N_d^3}\) is the projection onto the parameter space and \(J_2(\mathcal{X})\) is the full \(2\)-jet bundle of \(\mathcal{X}\).

On the affine chart \(\{Z_0 \neq 0\} \times \{A_{0d00} \neq 0\}\) of \(\mathbb{P}^3 \times \mathbb{P}^{N_d^3}\), we use the affine coordinates
\[
z_i \coloneqq \frac{Z_i}{Z_0}\;(i=1,2,3),\qquad
a_{\alpha} \coloneqq \frac{A_{\alpha_0\alpha_1\alpha_2\alpha_3}}{A_{0d00}},
\]
where \(\alpha_0 = d - \alpha_1-\alpha_2-\alpha_3\). Set
\[
\mathcal{X}_0 \coloneqq \mathcal{X} \cap \bigl(\{Z_0 \neq 0\} \times \{A_{0d00} \neq 0\}\bigr).
\]
In these coordinates \(\mathcal{X}_0\) is defined by the equation
\[
z_1^{d} + \sum_{\substack{\alpha_1+\alpha_2+\alpha_3 \leqslant d \\ \alpha_1 < d}} a_{\alpha} z^{\alpha}=0.
\]

For jets of order \(j = 1,2\) we write \(\xi_i^{(j)} \coloneqq z_i^{(j)}\) for \(i=1,2,3\). 
A local coordinate system on
\[
J_2^{v}\bigl(\{Z_0 \neq 0\} \times \{A_{0d00} \neq 0\}\bigr) \;\subset\; J_2^{v}\bigl(\mathbb{P}^3 \times \mathbb{P}^{N_d^3}\bigr)
\]
is then given by the tuple
\[
\bigl(z,\; a,\; \xi',\; \xi''\bigr),
\]
where
\begin{itemize}
    \item \(z = (z_1,z_2,z_3)\) are the affine coordinates on \(\{Z_0 \neq 0\}\),
    \item \(a = (a_{\alpha})\) with \(\alpha_1+\alpha_2+\alpha_3 \leqslant d,\ \alpha_1 < d\) are the normalized coefficients,
    \item \(\xi' = (\xi_1',\xi_2',\xi_3')\) and \(\xi'' = (\xi_1'',\xi_2'',\xi_3'')\) are the first and second jet coordinates.
\end{itemize}

 Following Siu’s strategy~\cite{Siu2004} as implemented by Păun~\cite{mihaipaun2008}, we work with the family of slanted vector fields of pole order \(7\).

\begin{thm}[{\cite[Proposition 1.3]{mihaipaun2008}}]\label{thm:globally generated vector field of pole order 7 cpt case}
Let \(\Sigma_0\) be the common zero locus of the three \(2\times 2\) Wronskians
\[
\det\begin{pmatrix}\xi_i' & \xi_j' \\ \xi_i'' & \xi_j''\end{pmatrix}
\qquad (1 \leqslant i < j \leqslant 3),
\]
that is,
\[
\Sigma_0 \coloneqq 
\Bigl\{\bigl(z,\;a,\;\xi',\;\xi''\bigr) \in J_2^v(\mathcal{X}_0)
\;\Bigm|\;
\xi'\wedge \xi'' = 0 \Bigr\}.
\]
Let \(\Sigma\) be the closure of \(\Sigma_0\) in \(J_2^v(\mathcal{X})\). Denote by
\[
p \colon J_2^v(\mathcal{X}) \longrightarrow \mathcal{X}
\]
the natural projection, and by \(\mathrm{pr}_1\colon\mathcal{X}\to\mathbb{P}^3\) and \(\mathrm{pr}_2\colon\mathcal{X}\to\mathbb{P}^{N_d^3}\) the two projections of the universal family. 
Then for some $q \in \mathbb{Z}_{>0}$, the twisted tangent bundle
\[
T_{J_2^v(\mathcal{X})}
\otimes p^*\mathrm{pr}_1^*\mathcal{O}_{\mathbb{P}^3}(7)
\otimes p^*\mathrm{pr}_2^*\mathcal{O}_{\mathbb{P}^{N_d^3}}(q)
\]
is generated by its global sections at every point of \(J_2^v(\mathcal{X})\setminus\Sigma\).
\end{thm}

\begin{rem}\label{rem:compact-exceptional-jets}\rm
The construction of the vector fields in
Theorem~\ref{thm:globally generated vector field of pole order 7 cpt case}
can be carried out equivariantly with respect to the reparameterization group
$\mathbb{G}_2$ (see~\eqref{reparameterization group G_k}). More precisely, by
following the construction and proof of
\cite[Proposition~1.3]{mihaipaun2008}, together with the argument
in \cite[Section~4.3]{Merker2009}, the vector fields may be chosen so as to be
compatible with the $\mathbb{G}_2$-action and to preserve the weighted degree
of weighted-homogeneous polynomials in the jet variables. Thus, both
$\mathbb{G}_2$-invariance and preservation of the weighted degree are intrinsic
to the construction and require no further modification of the vector fields.

Furthermore, as observed in \cite[Remark~1.2]{mihaipaun2008}, if the $2$-jet lifting of a entire curve
\(
f\colon\mathbb{C}\rightarrow \mathcal{X}
\)
is contained in $\Sigma$, then the image of $f$ is contained in a
line.
This exceptional case can be excluded uniformly on a nonempty
Zariski-open subset of the parameter space.
Indeed, consider the closed incidence variety
\[
\mathcal I_d
:=
\bigl\{
(L,a)\in \operatorname{Gr}(2,4)\times\mathbb P^{N_d^3}
\;\big|\;
L\subset X_a
\bigr\},
\]
where \(\operatorname{Gr}(2,4)\) parametrizes projective lines in
\(\mathbb P^3\).
For every line \(L\), the restriction map
\[
H^0\bigl(\mathbb P^3,\mathcal O_{\mathbb P^3}(d)\bigr)
\longrightarrow
H^0\bigl(L,\mathcal O_L(d)\bigr)
\]
is surjective. Thus the condition \(L\subset X_a\) imposes
\(d+1\) independent linear conditions on the defining equation of
\(X_a\). Consequently, \(\mathcal I_d\) is a projective bundle over
\(\operatorname{Gr}(2,4)\), with fibre
\(\mathbb P^{N_d^3-d-1}\), and
\[
\dim\mathcal I_d=4+(N_d^3-d-1)=N_d^3-d+3.
\]
The projection
\(
\mathcal I_d\longrightarrow\mathbb P^{N_d^3}
\)
is projective, so its image
\[
\mathcal B_d
:=
\bigl\{
a\in\mathbb P^{N_d^3}
\;\big|\;
X_a\text{ contains a projective line}
\bigr\}
\]
is Zariski closed. For \(d\geqslant4\), we have
\[
\dim\mathcal B_d\leqslant N_d^3-d+3<N_d^3.
\]
It follows that
\[
U_d^{\mathrm{nl}}
:=
\bigl\{
a\in\mathbb P^{N_d^3}\setminus\mathcal B_d
\;\big|\;
X_a\text{ is smooth}
\bigr\}
\]
is a nonempty Zariski-open subset.

Consequently, for every \(a\in U_d^{\mathrm{nl}}\) and every
nonconstant entire curve \(f:\mathbb C\to X_a\),
\[
j_2(f)(\mathbb C)\not\subset\Sigma,
\]
where \(j_2(f)\) is viewed as a curve in \(J_2^v(\mathcal X)\).
Hence, in the remainder of the argument, it suffices to
consider the open locus where the \(2\)-jets lie outside
\(\Sigma\).
\end{rem}

\medskip
\subsection{Logarithmic Case}\label{subsec:pre to SVF in log case}

Let \(\mathcal{C} \subset \mathbb{P}^2 \times \mathbb{P}^{N_d^2}\) denote the universal family of algebraic curves of degree \(d\) in \(\mathbb{P}^2\), where \(N_d^2 \coloneqq \binom{d+2}{2}-1\) is the dimension of the parameter space. It is defined by the homogeneous equation
\[
\sum_{|\alpha| = d} A_{\alpha} Z^{\alpha}=0.
\]
Here \([A] = [A_{\alpha}]_{|\alpha|=d} \in \mathbb{P}^{N_d^2}\) are the coefficients and \([Z] = [Z_0:Z_1:Z_2] \in \mathbb{P}^2\) the homogeneous coordinates. For a multi‑index \(\alpha = (\alpha_0,\alpha_1,\alpha_2) \in \mathbb{N}^3\) with length \(|\alpha| \coloneqq \alpha_0+\alpha_1+\alpha_2\), we write \(Z^{\alpha} \coloneqq Z_0^{\alpha_0}Z_1^{\alpha_1}Z_2^{\alpha_2}\).

The \emph{vertical logarithmic \(2\)-jet space} associated with this family is defined as
\[
\overline{J_2^v}\bigl(\mathbb{P}^2 \times \mathbb{P}^{N_d^2}\bigr)
\coloneqq \ker\bigl((\mathrm{pr}_2)_*\bigr)
\;\subset\; J_2\bigl(\mathbb{P}^2 \times \mathbb{P}^{N_d^2},\, -\log\mathcal{C}\bigr),
\]
where \(\mathrm{pr}_2 \colon \mathbb{P}^2 \times \mathbb{P}^{N_d^2} \to \mathbb{P}^{N_d^2}\) is the projection onto the parameter space, and \(J_2\bigl(\mathbb{P}^2 \times \mathbb{P}^{N_d^2},\, -\log\mathcal{C}\bigr)\) is the logarithmic \(2\)-jet bundle along the universal curve \(\mathcal{C}\).

We define the universal family \(\mathcal{Y} \subset \mathbb{P}^3 \times \mathbb{P}^{N_d^2+1}\) of degree-\(d\) hypersurfaces in \(\mathbb{P}^3\) by
\[
\mathcal{Y} \coloneqq 
\Bigl\{ A_{\mathbf{0}} Z_3^{d} + \sum_{|\alpha| = d} A_{\alpha} Z^{\alpha} = 0 \Bigr\}
\subset \mathbb{P}^3 \times \mathbb{P}^{N_d^2+1}.
\]
Here \(\mathbb{P}^3\) has homogeneous coordinates \([Z:Z_3]=[Z_0:Z_1:Z_2:Z_3]\), and the coefficient space \(\mathbb{P}^{N_d^2+1}\) is equipped with homogeneous coordinates \([A:A_{\mathbf{0}}]=[\dots:A_{\alpha}:\dots:A_{\mathbf{0}}]_{|\alpha|=d}\).

For notational convenience, let
\[
\{Z=\mathbf{0}\} \coloneqq \{Z_0=Z_1=Z_2=0\}
=
[0: 0: 0: 1]\times \mathbb{P}^{N_d^2+1}
\subset \mathbb{P}^3 \times \mathbb{P}^{N_d^2+1},
\]
\[
\{A=\mathbf{0}\} \coloneqq \{ A_{\alpha}=0 \text{ for all } |\alpha|=d\}
=
\mathbb{P}^3 \times
[0: \cdots : 0: 1]
\subset \mathbb{P}^3 \times \mathbb{P}^{N_d^2+1}.
\]
There is a natural forgetful map
\[
\pi \colon \mathbb{P}^3 \times \mathbb{P}^{N_d^2+1} \setminus \bigl(\{Z=\mathbf{0}\} \cup \{A=\mathbf{0}\}\bigr)
\longrightarrow \mathbb{P}^2 \times \mathbb{P}^{N_d^2}
\]
defined by discarding the auxiliary coordinates:
\[
\pi\bigl([Z:Z_3],\,[A:A_{\mathbf{0}}]\bigr) \mapsto \bigl([Z_0:Z_1:Z_2],\,[A]\bigr).
\]

Observe first the containment
\begin{equation}\label{containment relation}
\mathcal{Y} \cap \bigl(\{Z = \mathbf{0}\} \cup \{A = \mathbf{0}\}\bigr) 
\subset \mathcal{Y} \cap \bigl(\{A_{\mathbf{0}}=0\} \cup \{A=\mathbf{0}\}\bigr),
\end{equation}
which follows from the inclusion $\mathcal{Y} \cap \{Z=\mathbf{0}\} \subset \mathcal{Y} \cap \{A_{\mathbf{0}}=0\}$, a direct consequence of the defining equation of $\mathcal{Y}$.

Now set
\[
\mathcal{Y}^* \coloneqq \mathcal{Y} \setminus \bigl(\{A_{\mathbf{0}}=0\} \cup \{A=\mathbf{0}\}\bigr).
\]
By~\eqref{containment relation}, the restricted forgetful map
\[
\pi|_{\mathcal{Y}^*} \colon \mathcal{Y}^* \longrightarrow \mathbb{P}^2 \times \mathbb{P}^{N_d^2}
\]
is everywhere defined. Moreover, one checks that
\[
\bigl(\pi|_{\mathcal{Y}^*}\bigr)^{-1}(\mathcal{C}) = \{Z_3=0\} \eqqcolon \mathcal{D}.
\]
Thus $\pi$ induces a logarithmic morphism from the pair $(\mathcal{Y}^*,\mathcal{D})$ to $(\mathbb{P}^2\times\mathbb{P}^{N_d^2},\mathcal{C})$, and consequently a dominant map between the corresponding vertical logarithmic jet spaces:
\begin{equation}\label{pi_[2]}
    \pi_{[2]} \colon \overline{J_2^v}(\mathcal{Y}^*) \longrightarrow \overline{J_2^v}\!\bigl(\mathbb{P}^2\times\mathbb{P}^{N_d^2}\bigr).
\end{equation}
Here $\overline{J_2^v}(\mathcal{Y}^*)$ denotes the space of vertical logarithmic $2$-jets with respect to the hyperplane divisor $\mathcal{D}=\{Z_3=0\}$. This construction allows us to carry out all local computations on the more explicit space $\mathcal{Y}^*$ by pulling back via $\pi$, where coordinate expressions are readily available.

On the affine chart \(\{Z_0 \neq 0\} \times \{A_{\mathbf{0}} \neq 0\}\) we introduce the affine coordinates
\[
z_i \coloneqq \frac{Z_i}{Z_0}\;(i=1,2,3),\qquad
a_{\alpha} \coloneqq \frac{A_{\alpha_0\alpha_1\alpha_2}}{A_{\mathbf{0}}},
\]
where \(\alpha_0 = d - \alpha_1 - \alpha_2\). Set
\[
\mathcal{Y}_0 \coloneqq \mathcal{Y} \cap \bigl(\{Z_0 \neq 0\} \times \{A_{\mathbf{0}} \neq 0\}\bigr).
\]
In these coordinates \(\mathcal{Y}_0\) is defined by
\[
z_3^{d} + \sum_{|\alpha|=\alpha_1+\alpha_2 \leqslant d} a_{\alpha} z^{\alpha}=0.
\]

To streamline notation, for jets of order \(j = 1,2\) we define
\[
\xi_i^{(j)} \coloneqq 
\begin{cases}
z_i^{(j)} & (i=1,2),\\[2mm]
(\log z_3)^{(j)} & (i=3).
\end{cases}
\]
A local coordinate system for
\[
\overline{J_2^{v}}\bigl(\{Z_0 \neq 0\} \times \{A_{\mathbf{0}} \neq 0\}\bigr)
\subset J_2^{v}\bigl(\mathbb{P}^3 \times \mathbb{P}^{N_d^2+1},\, -\log\mathcal{D}\bigr)
\]
is then given by the abbreviated tuple
\[
\bigl(z,\; a,\; \xi',\; \xi''\bigr),
\]
where
\begin{itemize}
    \item \(z = (z_1,z_2,z_3)\) are the affine coordinates on \(\{Z_0 \neq 0\}\subset\mathbb{P}^3\);
    \item \(a = (a_{\alpha})\) with \(\alpha_1+\alpha_2 \leqslant d\) are the normalized coefficients in the parameter direction;
    \item \(\xi' = (\xi_1',\xi_2',\xi_3')\) and \(\xi'' = (\xi_1'',\xi_2'',\xi_3'')\) are the logarithmic jet coordinates.
\end{itemize}

Following Siu’s strategy~\cite{Siu2004} and Rousseau’s implementation~\cite{Rousseau2009}, we work with the family of slanted vector fields of pole order \(7\).

\begin{thm}[{\cite[Corollary 8]{Rousseau2009}}]\label{thm:globally generated vector field of pole order 7 log case}
Let \(\Sigma_0\) be the zero locus of the \(2\times 2\) Wronskian \(\det 
\bigl(\xi_i^{(j)}\bigr)_{1\leqslant i,j\leqslant 2}\) inside the vertical logarithmic jet space:
\[
\Sigma_0 \coloneqq 
\Bigl\{\bigl(z_i,\;a_{\alpha},\;\xi_i',\;\xi_i''\bigr)\in \overline{J_2^v}(\mathcal{Y}_0)
\;\Bigm|\;
\det \bigl(\xi_i^{(j)}\bigr)_{1\leqslant i,j\leqslant 2}=0\Bigr\},
\]
and let \(\Sigma\) be the closure of \(\Sigma_0\) in \(\overline{J_2^v}(\mathcal{Y}^*)\). Denote by
\[
p \colon \overline{J_2^v}\!\bigl(\mathbb{P}^2\times\mathbb{P}^{N_d^2}\bigr) \longrightarrow \mathbb{P}^2\times\mathbb{P}^{N_d^2}
\]
the natural projection, and by \(\mathrm{pr}_1\colon\mathbb{P}^2\times\mathbb{P}^{N_d^2}\to\mathbb{P}^2\), \(\mathrm{pr}_2\colon\mathbb{P}^2\times\mathbb{P}^{N_d^2}\to\mathbb{P}^{N_d^2}\) the two projections.
Then for some $q \in \mathbb{Z}_{>0}$, the twisted tangent bundle
\[
T_{\overline{J_2^v}(\mathbb{P}^2\times\mathbb{P}^{N_d^2})}
\otimes p^*\mathrm{pr}_1^*\mathcal{O}_{\mathbb{P}^2}(7)
\otimes p^*\mathrm{pr}_2^*\mathcal{O}_{\mathbb{P}^{N_d^2}}(q)
\]
is generated by its global sections at every point of 
\[
\overline{J_2^v}\bigl(\mathbb{P}^2\times\mathbb{P}^{N_d^2}\bigr) 
\setminus \Big( \overline{\pi_{[2]}(\Sigma)} \cup p^{-1}(\mathcal{C}) \Big)\qquad
\text{[see~\eqref{pi_[2]}]}. 
\]
\qed
\end{thm}

\begin{rem}\rm
The construction of the vector fields in
Theorem~\ref{thm:globally generated vector field of pole order 7 log case}
can be carried out equivariantly with respect to the reparameterization group
$\mathbb{G}_2$ (see~\eqref{reparameterization group G_k}). More precisely, by
following the construction and proof of
\cite[Corollary~8]{Rousseau2009}, together with the argument
in \cite[Section~4.3]{Merker2009}, the vector fields may be chosen so as to be
compatible with the $\mathbb{G}_2$-action and to preserve the weighted degree
of weighted-homogeneous polynomials in the jet variables. Thus, both
$\mathbb{G}_2$-invariance and preservation of the weighted degree are intrinsic
to the construction and require no further modification of the vector fields.

Fix a parameter $a$ corresponding to a smooth curve $\mathcal{C}_a$,
and let $f\colon\mathbb{C}\to\mathbb{P}^2$ be a nonconstant entire curve
with $f(\mathbb{C})\not\subset\mathcal{C}_a$. Write
\[
F_a(z):=(f(z),a),
\qquad
\Omega:=\mathbb{C}\setminus f^{-1}(\mathcal{C}_a).
\]
If the entire lifted jet curve satisfies
\[
j_2(F_a)(\Omega)\subset\pi_{[2]}(\Sigma),
\]
then the planar Wronskian vanishes identically wherever affine coordinates
are defined. Consequently, $f(\mathbb{C})$ is contained in a projective
line; see \cite[Remark~10]{Rousseau2009}.
Away from the boundary, the vanishing of this Wronskian is a closed
condition and also holds on $\overline{\pi_{[2]}(\Sigma)}$.
In particular, an algebraically nondegenerate entire curve has a nonempty
open subset of $\Omega$ on which its lifted jets belong to the
global-generation locus of
Theorem~\ref{thm:globally generated vector field of pole order 7 log case}.

The use of this observation differs in the qualitative and quantitative
arguments. In Subsection~\ref{subsection: 5.4}, the boundary
$\mathcal{C}_a$ is very generic and
$f(\mathbb{C})\subset\mathbb{P}^2\setminus\mathcal{C}_a$.
The algebraically degenerate case is then excluded by
Lemma~\ref{lem:SMT-algebraically-degenerate}, since
$N_f^{[1]}(r,\mathcal{C}_a)=0$.
In Subsection~\ref{subsect:SMT}, entire curves are allowed to meet the
boundary. Algebraically degenerate curves are treated separately, under
the very generic assumption, by the same lemma.
The subsequent slanted-vector-field argument is applied only to
algebraically nondegenerate entire curves.
\end{rem}

\subsection{Extension of $2$-Jet Differentials}

To implement Siu's strategy of slanted vector fields~\cite[pp.\,558--559]{Siu2004}, we need families of negatively twisted invariant $2$-jet differentials whose zero divisors remain irreducible and reduced modulo the vertical divisor $\Gamma_2$. We construct such a family by extending the section defining the residual divisor on the relative Demailly--Semple tower, and then multiplying by the appropriate power of the canonical section of $\Gamma_2$. The extension uses the following consequence of algebraic cohomology and base change.

\begin{thm}[{\cite[Chapter~III, Section~12]{HartshorneGTM52}}]\label{thm:extension of sections}
Let $\tau\colon\mathcal{Y}\to S$ be a projective flat morphism of complex algebraic varieties, with $S$ irreducible, and let $\mathcal{L}$ be an algebraic vector bundle on $\mathcal{Y}$. There exists a dense Zariski-open subset $S^\circ\subset S$ on which $\tau_*\mathcal{L}$ is locally free and the natural base-change maps
\[
(\tau_*\mathcal{L})\otimes\kappa(s)
\longrightarrow
H^0\bigl(\mathcal{Y}_s,\mathcal{L}|_{\mathcal{Y}_s}\bigr)
\]
are isomorphisms. In particular, for every complex point $s_0\in S^\circ$, there is a Zariski-open neighbourhood $U_{s_0}\subset S^\circ$ of $s_0$ such that the restriction map
\[
H^0\bigl(\tau^{-1}(U_{s_0}),\mathcal{L}\bigr)
\longrightarrow
H^0\bigl(\mathcal{Y}_{s_0},\mathcal{L}|_{\mathcal{Y}_{s_0}}\bigr)
\]
is surjective.
\end{thm}

We combine this theorem with the irreducible-component argument of Demailly--El~Goul~\cite[p.\,532]{Demailly-Elgoul2000} (see also~\cite[Lemma~1.3.2]{ElGoul2003} in the logarithmic setting).

\begin{pro}\label{prop:extension and irred. component cpt case}
Let $d\geqslant15$, and let $S^{\mathrm{sm}}\subset\mathbb{P}^{N_d^3}$ be the Zariski-open subset parametrizing smooth surfaces of degree $d$ in $\mathbb{P}^3$. There exist positive integers $m,t$, integers $b_1,b_2$ satisfying
\[
b_1\geqslant2b_2>0,\qquad b_1+b_2=m,
\]
and a nonempty Zariski-open subset $U^{m,t}\subset S^{\mathrm{sm}}$ on which there is an algebraic, hence holomorphic, family of nonzero invariant $2$-jet differentials
\[
\omega_{1,a}\in H^0\bigl(X_a,E_{2,m}T_{X_a}^*\otimes\mathcal{O}_{X_a}(-t)\bigr),
\qquad a\in U^{m,t}.
\]
Viewed as a section on $X_{a,2}$ via~\eqref{direct image formula, cpt}, each $\omega_{1,a}$ has zero divisor
\[
\operatorname{div}(\omega_{1,a})=Z_a+b_1\Gamma_{a,2},
\]
where $Z_a$ is irreducible and reduced and $Z_a\neq\Gamma_{a,2}$. Moreover,
\[
Z_a\sim b_1u_{1,a}+b_2u_{2,a}-t\pi_{2,0}^*h_a,
\]
where $u_{1,a}$ and $u_{2,a}$ denote the classes of $\pi_{2,1}^*\mathcal{O}_{X_{a,1}}(1)$ and $\mathcal{O}_{X_{a,2}}(1)$, respectively, and $h_a$ is the class of $\mathcal{O}_{X_a}(1)$.
\end{pro}

The proof follows the same line as that of \cite[Proposition 4.3]{Hou-Huynh-Merker-Xie2025} and is included here for the reader’s convenience.

\begin{proof}
Restrict the universal family of degree-$d$ surfaces to $S^{\mathrm{sm}}$, and denote it by $\mathcal{X}\to S^{\mathrm{sm}}$. Starting with the relative tangent bundle $T_{\mathcal{X}/S^{\mathrm{sm}}}$, the relative Demailly--Semple construction gives two projective bundles
\[
\mathcal{X}_2\xrightarrow{\Pi_{2,1}}\mathcal{X}_1
\xrightarrow{\Pi_{1,0}}\mathcal{X}.
\]
Their composition with $\mathcal{X}\to S^{\mathrm{sm}}$ defines a smooth projective morphism $q\colon\mathcal{X}_2\to S^{\mathrm{sm}}$. Write $\Pi_{2,0}=\Pi_{1,0}\circ\Pi_{2,1}$, and let $\varGamma_2$ be the relative vertical divisor, whose fibre at $a$ is $\Gamma_{a,2}$. The line bundle $\mathcal{O}_{\mathcal{X}}(1)$ is pulled back from $\mathbb{P}^3$. For every $(c_1,c_2,s)\in\mathbb{Z}^3$, set
\[
\mathcal{L}_{c_1,c_2,s}
:=\Pi_{2,1}^*\mathcal{O}_{\mathcal{X}_1}(c_1)
\otimes\mathcal{O}_{\mathcal{X}_2}(c_2)
\otimes\Pi_{2,0}^*\mathcal{O}_{\mathcal{X}}(-s).
\]
By Theorem~\ref{thm:extension of sections}, there is a dense Zariski-open subset $S_{c_1,c_2,s}\subset S^{\mathrm{sm}}$ on which $q_*\mathcal{L}_{c_1,c_2,s}$ is locally free and commutes with base change. There are only countably many integer triples. Since $S^{\mathrm{sm}}$ is irreducible over the uncountable field $\mathbb{C}$, we may choose a complex point
\[
a_0\in\bigcap_{(c_1,c_2,s)\in\mathbb{Z}^3}S_{c_1,c_2,s}
\]
which also lies outside the Noether--Lefschetz locus. Thus
$\operatorname{Pic}(X_{a_0})=\mathbb{Z}[\mathcal{O}_{X_{a_0}}(1)]$.

By Corollary~\ref{cor:existence of 2 jets, compact case} and the direct-image formula~\eqref{direct image formula, cpt}, there exist positive integers $M,T$ and a nonzero section
\[
\omega_0\in H^0\bigl(X_{a_0,2},
\mathcal{O}_{X_{a_0,2}}(M)\otimes\pi_{2,0}^*\mathcal{O}_{X_{a_0}}(-T)\bigr).
\]
Write $\operatorname{div}(\omega_0)=\sum_jp_jZ_j$, with $p_j\geqslant1$ and the $Z_j$ distinct irreducible reduced divisors. Put $u_i=u_{i,a_0}$ and $h=h_{a_0}$. The projective-bundle formula gives
\[
\operatorname{Pic}(X_{a_0,2})
=\mathbb{Z}\pi_{2,0}^*h\oplus\mathbb{Z}u_1\oplus\mathbb{Z}u_2.
\]
Consequently, for each $j$ we can write
\[
Z_j\sim b_{1,j}u_1+b_{2,j}u_2-t_j\pi_{2,0}^*h,
\qquad b_{1,j},b_{2,j},t_j\in\mathbb{Z}.
\]
The divisor-class argument of~\cite[Lemma~3.3(d)]{Demailly-Elgoul2000} gives the following three possibilities:
\begin{itemize}
    \item $(b_{1,j},b_{2,j})=(0,0)$ and $t_j<0$; in this case $Z_j$ is pulled back from an effective divisor on $X_{a_0}$;
    \item $Z_j=\Gamma_{a_0,2}$, with $(b_{1,j},b_{2,j},t_j)=(-1,1,0)$;
    \item $b_{1,j}\geqslant2b_{2,j}\geqslant0$ and $m_j:=b_{1,j}+b_{2,j}>0$.
\end{itemize}
In the third case, the relation
\[
\pi_{2,1}^*\mathcal O_{X_{1,a_0}}(b_1)
\otimes
\mathcal O_{X_{2,a_0}}(b_2)
\cong
\mathcal O_{X_{2,a_0}}(b_1+b_2)
\otimes
\mathcal O_{X_{2,a_0}}(-b_1\Gamma_{2,a_0}), \quad \text{see \eqref{gamma_2=O(-1, 1)}}
\]
gives
\[
Z_j
\sim
b_{1,j}u_1+b_{2,j}u_2-t_j\pi_{2,0}^*h
\sim
m_j u_2-b_{1,j}\Gamma_{2,a_0}
-t_j\pi_{2,0}^*h.
\]

Let $I$ be the set of indices in the third case. Comparing coefficients in
\[
\sum_jp_j[Z_j]=Mu_2-T\pi_{2,0}^*h
\]
yields
\[
M=\sum_{j\in I}p_jm_j,\qquad
T=\sum_jp_jt_j\leqslant\sum_{j\in I}p_jt_j.
\]
Indeed, the first two cases contribute zero to the sum of the $u_1$- and $u_2$-coefficients and contribute only nonpositive $t_j$. Since $M>0$, the set $I$ is nonempty, and hence
\[
\max_{j\in I}\frac{t_j}{m_j}
\geqslant\frac{\sum_{j\in I}p_jt_j}{\sum_{j\in I}p_jm_j}
\geqslant\frac{T}{M}>0.
\]
Choose a component $Z=Z_j$ attaining this maximum, and put
$b_1=b_{1,j}$, $b_2=b_{2,j}$, $m=b_1+b_2$ and $t=t_j$.
Then $m,t>0$ and $b_1\geqslant2b_2\geqslant0$.

Set $\mathcal{L}=\mathcal{L}_{b_1,b_2,t}$. After identifying
$\mathcal{O}_{X_{a_0,2}}(Z)$ with $\mathcal{L}|_{X_{a_0,2}}$, the canonical section defining $Z$ gives
\[
\widetilde\omega_{a_0}\in H^0\bigl(X_{a_0,2},\mathcal{L}|_{X_{a_0,2}}\bigr),
\qquad\operatorname{div}(\widetilde\omega_{a_0})=Z.
\]
If $b_2=0$, then $b_1=m$, and pushing this nonzero section down the two projective bundles would give
\[
H^0\bigl(X_{a_0},S^mT_{X_{a_0}}^*\otimes\mathcal{O}_{X_{a_0}}(-t)\bigr)\neq0.
\]
This contradicts the vanishing of negatively twisted symmetric differentials recalled in Subsection~\ref{subsec:existence of 2-jet differ.}. Therefore $b_1\geqslant2b_2>0$.

We now extend the section defining $Z$. By the choice of $a_0$, the sheaf $q_*\mathcal{L}$ is locally free near $a_0$ and commutes with base change there. Choosing a local algebraic frame extends $\widetilde\omega_{a_0}$ to a section of $q_*\mathcal{L}$ on a Zariski-open neighbourhood $U$ of $a_0$, hence to
\[
\widetilde\omega\in H^0\bigl(q^{-1}(U),\mathcal{L}\bigr).
\]
The parameters where the fibre restriction of $\widetilde\omega$ vanishes identically form the zero locus of this section of the locally free sheaf $q_*\mathcal{L}$. Shrinking $U$ around $a_0$, we may therefore assume that
$\widetilde\omega|_{X_{a,2}}\not\equiv0$ for every $a\in U$.

Every geometric fibre of $q$ is smooth and integral, so each of these nonzero fibre sections is a nonzerodivisor. The fibrewise criterion for relative effective Cartier divisors shows that
\[
\mathcal{Z}:=V(\widetilde\omega)\subset q^{-1}(U)
\]
is an effective Cartier divisor flat over $U$. As a closed subscheme of the projective $U$-scheme $q^{-1}(U)$, it is proper and of finite presentation over $U$. Its fibre at $a_0$ is the irreducible reduced divisor $Z$, which is geometrically integral over $\mathbb{C}$. The openness of the geometrically integral fibre locus for proper flat morphisms of finite presentation~\cite[Th\'eor\`eme~12.2.4(viii)]{EGA4.2_1966} now allows us to shrink $U$ around $a_0$ so that every fibre $Z_a:=\mathcal{Z}|_a$ is geometrically integral. In particular, for every complex parameter $a\in U$, the divisor $Z_a$ is irreducible and reduced and satisfies
\[
Z_a\sim b_1u_{1,a}+b_2u_{2,a}-t\pi_{2,0}^*h_a.
\]
Since $b_2>0$, its intersection with a general fibre of $\pi_{2,1}$ has positive degree $b_2$; thus $Z_a$ is nonempty and projects onto $X_{a,1}$. Moreover, $Z_a\neq\Gamma_{a,2}$: their coefficients in the two tautological factors of $\operatorname{Pic}(X_{a,2})$ are $(b_1,b_2)$ and $(-1,1)$, respectively. This comparison does not require $\operatorname{Pic}(X_a)$ to have rank one for every $a\in U$.

Finally, the relative version of~\eqref{gamma_2=O(-1, 1)} is
\[
\mathcal{O}_{\mathcal{X}_2}(\varGamma_2)
\cong\Pi_{2,1}^*\mathcal{O}_{\mathcal{X}_1}(-1)
\otimes\mathcal{O}_{\mathcal{X}_2}(1).
\]
Let $s_{\varGamma_2}$ be its canonical section. Since $m=b_1+b_2$, the product
$\boldsymbol{\omega}_1:=s_{\varGamma_2}^{b_1}\widetilde\omega$ is an algebraic section of
\[
\mathcal{O}_{\mathcal{X}_2}(m)\otimes\Pi_{2,0}^*\mathcal{O}_{\mathcal{X}}(-t)
\]
over $q^{-1}(U)$, with
\[
\operatorname{div}(\omega_{1,a})=Z_a+b_1\Gamma_{a,2}
\qquad(a\in U),
\]
where $\omega_{1,a}=\boldsymbol{\omega}_1|_{X_{a,2}}$. The relative direct-image correspondence underlying~\eqref{direct image formula, cpt} identifies these restrictions with an algebraic, hence holomorphic, family of nonzero sections of $E_{2,m}T_{X_a}^*\otimes\mathcal{O}_{X_a}(-t)$. 
Taking $U^{m,t}=U$ completes the proof.
\end{proof}

By the same reasoning, using the vanishing of negatively twisted logarithmic symmetric differentials in~\cite[Lemma~1.4.1]{ElGoul2003}, the analogous statement holds in the logarithmic setting:

\begin{pro}\label{prop:extension and irred. component log case}
Let $d\geqslant11$, and let $S_{\log}^{\mathrm{sm}}\subset\mathbb{P}^{N_d^2}$ be the Zariski-open subset parametrizing smooth plane curves of degree $d$. There exist positive integers $m,t$, integers $b_1,b_2$ satisfying
\[
b_1\geqslant2b_2>0,\qquad b_1+b_2=m,
\]
and a nonempty Zariski-open subset $U^{m,t}\subset S_{\log}^{\mathrm{sm}}$ on which there is an algebraic, hence holomorphic, family of nonzero invariant logarithmic $2$-jet differentials
\[
\omega_{1,a}\in H^0\bigl(\mathbb{P}^2,E_{2,m}T_{\mathbb{P}^2}^*(\log\mathcal{C}_a)
\otimes\mathcal{O}_{\mathbb{P}^2}(-t)\bigr),
\qquad a\in U^{m,t}.
\]
Viewed as a section on the second level $X_{a,2}$ of the logarithmic Demailly--Semple tower of $(\mathbb{P}^2,\mathcal{C}_a)$, each $\omega_{1,a}$ has zero divisor
\[
\operatorname{div}(\omega_{1,a})=Z_a+b_1\Gamma_{a,2},
\]
where $Z_a$ is irreducible and reduced and $Z_a\neq\Gamma_{a,2}$. Moreover,
\[
Z_a\sim b_1u_{1,a}+b_2u_{2,a}-t\pi_{2,0}^*h,
\]
where $u_{1,a}$ and $u_{2,a}$ are the classes of $\pi_{2,1}^*\mathcal{O}_{X_{a,1}}(1)$ and $\mathcal{O}_{X_{a,2}}(1)$, respectively, and $h$ is the class of $\mathcal{O}_{\mathbb{P}^2}(1)$. \qed
\end{pro}

\subsection{Producing an Independent $2$-Jet Differential via Siu's Slanted Vector Fields}

Following P\u{a}un's argument \cite[p.~889]{mihaipaun2008}, we apply the slanted vector fields of pole order $7$ in the compact case.

In both the compact and logarithmic universal families below, the initial
jet differential is an equivariant polynomial on the vertical jet space.
A slanted vector field therefore acts on it naturally by Lie
derivation, producing a genuine global invariant jet differential on each
fibre; the pole order of the vector field is added to the base twist.

\begin{pro}\label{vector field, compact case}
	Given a holomorphic family $\{ \omega_{1, a} \}_{a\in U}$ of negatively twisted invariant $2$-jet differentials as in Proposition \ref{prop:extension and irred. component cpt case} with vanishing order $t \geqslant 8$,
    one can find  some slanted vector field $V$ of pole order $7$ in Theorem \ref{thm:globally generated vector field of pole order 7 cpt case}, such that
    the derivative of $\{ \omega_{1, a} \}$ along $V$ gives a new family of holomorphic invariant $2$-jet differentials $\omega_{2, a}$ with vanishing order $t-7$. Furthermore, the common zero set $\{\omega_{1,a} = 0\} \cap \{\omega_{2,a} = 0\}$ in the second level $\left( X_{a} \right)_2$ of the Demailly–Semple tower of $X_a$ has codimension at least $2$ modulo $\Gamma_{a,2}$ for a generic parameter $a \in U$. \qed
\end{pro}


Similarly, following Rousseau's argument \cite[p.~38]{Rousseau2009}, we apply the slanted vector fields of pole order $7$ in the logarithmic case.

\begin{pro}\label{vector field, logarithmic case}
    Given a holomorphic family $\{ \omega_{1, a} \}_{a\in U}$ of negatively twisted invariant logarithmic $2$-jet differentials as in Proposition \ref{prop:extension and irred. component log case} with vanishing order $t \geqslant 8$,
    one can find  some slanted vector field $V$ of pole order $7$ in Theorem \ref{thm:globally generated vector field of pole order 7 log case}, such that
    the derivative of $\{ \omega_{1, a} \}$ along $V$ gives a new family of holomorphic invariant logarithmic $2$-jet differentials $\omega_{2, a}$ with vanishing order $t-7$. Furthermore, the common zero set $\{\omega_{1,a} = 0\} \cap \{\omega_{2,a} = 0\}$ in the second level $\left( \mathbb{P}^2_{a} \right)_2$ of the logarithmic Demailly–Semple tower of  $\left( \mathbb{P}^2 , \mathcal{C}_a \right)$ has codimension at least $2$ modulo $\Gamma_{a,2}$ for a generic parameter $a \in U$. \qed
\end{pro}

\section{\bf Proof of the Main Results}\label{sect:5}

\subsection{Constructing a Second Independent Invariant $2$-Jet Differential}
Throughout the proof in the compact case, we restrict the parameter
space to the nonempty Zariski-open subset
\(U_d^{\mathrm{nl}}\) defined in
Remark~\ref{rem:compact-exceptional-jets}.
The initial parameter \(a_0\) and all subsequent parameter open subsets
are chosen inside this locus.
This restriction is compatible with the generic cohomology and
base-change conditions used below.
In particular, every surface in the resulting families contains
no projective line.

We now give the detailed proof of Theorem~\ref{thm:GGL-bound}. Let \(X\) be a generic surface of degree \(d \geqslant 17\) in \(\mathbb{P}^3\). By Proposition~\ref{existence of 2 jets, compact case}, for suitable positive integers \(M\) and \(T\) there exists a nonzero section
\[
\omega_0 \in H^0\!\bigl(X,\;E_{2,M}T_X^*\otimes\mathcal{O}_{X}(-T)\bigr).
\]

Assume that there exists an entire curve \(f \colon \mathbb{C} \to X\). Consider the Demailly--Semple tower
\[
X_2 \longrightarrow X_1 \longrightarrow X_0 \coloneqq X,
\]
and let \(f_{[2]}\) and \(f_{[1]}\) denote the lifts of \(f\) to \(X_2\) and \(X_1\), respectively.


\medskip

Following the argument in the proof of Proposition~\ref{prop:extension and irred. component cpt case}, we select a very generic parameter \(a_0\) in the moduli space \(\mathbb{P}^{N_d^3}\) of degree‑\(d\) surfaces in \(\mathbb{P}^3\) such that:

\begin{itemize}
    \item \(X_{a_0}\) is a smooth algebraic surface of general type;
    
    \item  the Picard group satisfies \(\operatorname{Pic}(X_{a_0}) \cong \mathbb{Z}\) (by the Noether--Lefschetz theorem);
    
    
    \item the generic cohomology and base-change conditions for all line bundles
          $\mathcal{L}_{c_1,c_2,s}$ on the relative second Demailly--Semple tower,
          as used in the proof of Proposition~\ref{prop:extension and irred. component cpt case},
          hold at $a_0$.
\end{itemize}

Using Demailly--El~Goul’s construction~\cite[p.\,532]{Demailly-Elgoul2000} (see also the proof of Proposition~\ref{prop:extension and irred. component cpt case}), we obtain an irreducible invariant \(2\)-jet differential
\begin{equation}
    \label{omega-1}
\omega_{1,a_0} \in  H^0\!\bigl(X_{a_0},\;E_{2,m}T_{X_{a_0}}^*\otimes\mathcal{O}_{X_{a_0}}(-t)\bigr),
\end{equation}
whose zero divisor in \(X_{2,a_0}\) consists of an irreducible divisor \(Z_{a_0}\) modulo \(\Gamma_{a_0 ,2}\).

Applying the construction in Proposition~\ref{prop:extension and irred. component cpt case} to the section defining $Z_{a_0}$, and then multiplying its extension by the fixed power $s_{\Gamma_2}^{b_1}$, we obtain an algebraic, hence holomorphic, family
\begin{equation}\label{family of omega_1}
\bigl\{\omega_{1,a} \mid 0 \not\equiv \omega_{1,a} \in H^0\!\bigl(X_a,\;E_{2,m}T_{X_a}^*\otimes\mathcal{O}_{X_a}(-t)\bigr),\; a \in U_{a_0}^{m,t}\bigr\},
\end{equation}
over a Zariski-open neighbourhood $U_{a_0}^{m,t}$ of $a_0$. The residual zero divisors form the proper flat Cartier family constructed in that proposition. Shrinking to a Zariski-open neighbourhood $U^{m,t}\subset U_{a_0}^{m,t}$ of $a_0$, we may therefore assume that every residual fibre is geometrically integral. All $X_a$ remain smooth surfaces of general type, and
\begin{equation}\label{eq:irred-divisor-mod-Gamma}
\operatorname{div}(\omega_{1,a})=Z_a+b_1\Gamma_{a,2},
\qquad Z_a\text{ irreducible and reduced},\quad Z_a\neq\Gamma_{a,2}.
\end{equation}

We now recall the following numerical criterion of Demailly--El~Goul.

\begin{pro}[{\cite[Proposition 3.4]{Demailly-Elgoul2000}}]\label{prop:DEG-numerical-criterion}
Let \(X\) be a smooth surface of degree‑\(d\) of general type in \(\mathbb{P}^3\) with \(\operatorname{Pic}(X)\cong\mathbb{Z}\). Suppose that there exists a nonzero irreducible negatively twisted \(2\)-jet differential
\[
\omega_1 \in 
H^0\!\bigl(X,\;E_{2,m}T_X^*\otimes\mathcal{O}_{X}(-t)\bigr)
\]
satisfying
\begin{equation}\label{ineq:Demailly-El Goul condition}
t < (d-4)\frac{13c_1^{2}-9c_2}{12c_1^{2}}\,m .
\end{equation}
Then the line bundle \(\mathcal{O}_{X_2}(1)\) restricted to the irreducible \(3\)-dimensional component \(Z\)  of the zero locus \(\{\omega_1=0\}\) in \(X_2\) is big. Here $Z$ is the component that projects onto \(X_1\) and is distinct from \(\Gamma_2\).
\end{pro}

For the reader’s convenience, we provide a self-contained proof. Our argument follows a slightly different route from that of~\cite[Proposition 3.4]{Demailly-Elgoul2000}, as we were unable to fully verify one of the technical steps in the original presentation.

The key ingredient is the following statement. A complete proof is given in our earlier work~\cite[Proposition 6.1]{Hou-Huynh-Merker-Xie2025}.

\begin{pro}\label{pro:the higher direct images vanishing}
Let $a_1,a_2\in\mathbb Z$. Denote
\begin{equation}\label{eq:O(a_1,a_2)}
    \mathcal{O}_{X_{2}}(a_1,a_2)\coloneqq\pi_{2,1}^{*}\mathcal{O}_{X_1}(a_1)\otimes\mathcal{O}_{X_2}(a_2), \qquad a_1 , a_2 \in \mathbb{Z}.
\end{equation}
\begin{enumerate}
    \item If $a_2=-1$, then
    \[
    R^q(\pi_{2,0})_*\mathcal O_{X_2}(a_1,-1)=0
    \qquad\text{for every }q\geqslant0.
    \]

    \item If $a_2\geqslant0$, then
    \[
    R^q(\pi_{2,0})_*\mathcal O_{X_2}(a_1,a_2)=0
    \qquad\text{for every }q\geqslant2,
    \]
    and
    \[
    R^1(\pi_{2,0})_*\mathcal O_{X_2}(a_1,a_2)=0
    \quad\Longleftrightarrow\quad
    a_1-2a_2\geqslant-1.
    \]
\end{enumerate}
In particular, for every $p\geqslant0$,
\[
R^q(\pi_{2,0})_*\mathcal O_{X_2}(2p,p)=0
\qquad\text{for every }q\geqslant1. 
\]
\qed
\end{pro}

\begin{proof}[Proof of Proposition~\ref{prop:DEG-numerical-criterion}]
As in the proof of Proposition~\ref{prop:extension and irred. component cpt case}, one can show that the zero divisor of \(\omega_1\) on \(X_2\) decomposes as
\[
\{\omega_1=0\} = Z + b_1 \Gamma_2,
\]
where \(Z\subset X_2\) is irreducible and reduced.  
Write \(u_1 = \pi_{2,1}^*\mathcal{O}_{X_1}(1)\) and \(u_2 = \mathcal{O}_{X_2}(1)\). Since \(\operatorname{Pic}(X_2)\) decomposes as \(\operatorname{Pic}(X)\oplus\mathbb{Z}u_1\oplus\mathbb{Z}u_2\), the divisor \(Z\) is linearly equivalent to
\begin{equation}\label{equivalent class of Z}
Z \sim b_1 u_1 + b_2 u_2 - t\pi_{2,0}^*h ,\qquad 
b_1,b_2\in\mathbb{Z},\; b_1\geqslant 2b_2>0,\; b_1+b_2=m,\; t\in\mathbb{Z}_{+},
\end{equation}
where \(h=\mathcal{O}_X(1)\). 

We shall now prove that \(\mathcal{O}_{X_2}(2,1)|_{Z}\) is big whenever~\eqref{ineq:Demailly-El Goul condition} holds, using a Riemann--Roch computation.  
With the intersection formulas on the Demailly--Semple tower given in~\cite[Lemma 3.3 (e)]{Demailly-Elgoul2000}, a direct calculation yields
\[
(2u_1+u_2)^{3}\cdot Z
= m\bigl(13c_1^{2}-9c_2\bigr) + 12t(c_1\cdot h)
\qquad\text{[cf.~\cite[p.\,530 (\dag\dag\dag)]{Demailly-Elgoul2000}]},
\]
where \(c_1=(4-d)h,\; c_2=(d^{2}-4d+6)h^{2}\) and \(h^{2}=d\). Under condition~\eqref{ineq:Demailly-El Goul condition} we therefore obtain
\[
(2u_1+u_2)^{3}\cdot Z
= m\bigl(13c_1^{2}-9c_2\bigr) - \frac{12}{d-4}\,t\,c_1^{2} \; > \; 0 .
\]

Applying the Riemann--Roch theorem to the divisor \(Z\), we have for sufficiently divisible \(p\gg 1\)
\begin{align*}
h^{0}\!\bigl(Z,\mathcal{O}_{X_2}(2p,p)|_{Z}\bigr)
&+h^{2}\!\bigl(Z,\mathcal{O}_{X_2}(2p,p)|_{Z}\bigr) \\
&\geqslant \chi\!\bigl(Z,\mathcal{O}_{X_2}(2p,p)|_{Z}\bigr) \\
&=\frac{(2u_1+u_2)^{3}\cdot Z}{3!}\,p^{3}+O(p^{2})\gg 1 .
\end{align*}

We next show that
\[
h^{2}\!\bigl(Z,\mathcal{O}_{X_2}(2p,p)|_{Z}\bigr)=0
\qquad\text{for all sufficiently divisible }p\gg 1 .
\]

From the short exact sequence
\[
0\longrightarrow \mathcal{O}_{X_2}(-Z)\otimes\mathcal{O}_{X_2}(2p,p)
\longrightarrow \mathcal{O}_{X_2}(2p,p)
\longrightarrow \mathcal{O}_{Z}\otimes\mathcal{O}_{X_2}(2p,p)
\longrightarrow 0,
\]
the associated long exact cohomology sequence gives the estimate
\[
h^{2}\!\bigl(Z,\mathcal{O}_{X_2}(2p,p)|_{Z}\bigr)
\leqslant h^{2}\!\bigl(X_2,\mathcal{O}_{X_2}(2p,p)\bigr)
      +h^{3}\!\bigl(X_2,\mathcal{O}_{X_2}(-Z)\otimes\mathcal{O}_{X_2}(2p,p)\bigr).
\]

\noindent\textbf{Step 1. Vanishing of \(h^{2}(X_2,\mathcal{O}_{X_2}(2p,p))\).}
Consider the direct image sheaf \((\pi_{2,0})_*\mathcal{O}_{X_2}(2p,p)\) on \(X\). By Proposition~\ref{pro:the higher direct images vanishing} we have
\(R^{q}(\pi_{2,0})_*\mathcal{O}_{X_2}(2p,p)=0\) for all \(q\geqslant 1\). Hence Leray's spectral sequence yields
\[
H^{2}\!\bigl(X_2,\mathcal{O}_{X_2}(2p,p)\bigr)
\cong H^{2}\!\bigl(X,(\pi_{2,0})_*\mathcal{O}_{X_2}(2p,p)\bigr).
\]

Using the isomorphism (cf.~\cite[Lemma 3.3 (c)]{Demailly-Elgoul2000})
\[
(\pi_{2,0})_*\mathcal{O}_{X_2}(2p,p)\cong E_{2,3p}T_{X}^{*},
\]
and Serre duality, we obtain
\begin{equation}\label{use Serre duality}
H^{2}\!\bigl(X,(\pi_{2,0})_*\mathcal{O}_{X_2}(2p,p)\bigr)
\cong H^{2}\!\bigl(X,E_{2,3p}T_{X}^{*}\bigr)
\cong H^{0}\!\bigl(X,E_{2,3p}T_{X}\otimes K_{X}\bigr)^*.
\end{equation}
Here we use Serre duality $H^2(X,\mathcal{F}) \cong H^0(X,\mathcal{F}^* \otimes K_X)^*$, together with the fact that $(E_{2,3p}T_X^*)^* \cong E_{2,3p}T_X$.

The bundle \(E_{2,3p}T_{X}\otimes K_{X}\) inherits a filtration from the dual of~\eqref{filtration-of-E2m, cpt} twisted by \(K_{X}\); its graded pieces are
\[
S^{3p-3j}T_{X}\otimes K_{X}^{1-j}\qquad(j=0,1,\dots ,p).
\]

For all sufficiently divisible \(p\gg 1\), one has \(3p-3j>2(1-j)\) for every \(j=0,1,\dots ,p\). By the Bogomolov vanishing theorem (Theorem~\ref{thm:Bogomolov vanishing theorem}) we therefore get
\[
H^{0}\!\bigl(X,S^{3p-3j}T_{X}\otimes K_{X}^{1-j}\bigr)=0\qquad(j=0,1,\dots ,p),
\]
which forces \(H^{0}(X,E_{2,3p}T_{X}\otimes K_{X})=0\). Combining this with~\eqref{use Serre duality}, we obtain
\begin{equation}\label{H^2 of the direct image of O(2p,p) tensoring O(-p tau)}
h^{2}\!\bigl(X_2,\mathcal{O}_{X_2}(2p,p)\bigr)=0 .
\end{equation}

\noindent\textbf{Step 2. Vanishing of \(h^{3}(X_2,\mathcal{O}_{X_2}(-Z)\otimes\mathcal{O}_{X_2}(2p,p))\).}
By~\eqref{equivalent class of Z}, set
\[
\mathcal M_p:= \mathcal{O}_{X_2}(2p,p)\otimes\mathcal O_{X_2}(-Z)
=\mathcal O_{X_2}(2p-b_1,p-b_2)
\otimes\pi_{2,0}^*\mathcal O_{X}(t).
\]
Write
\[
r:=b_1-2b_2\geqslant0.
\]
Choose $p$ so large that $p-b_2\geqslant0$ and, when $r\geqslant2$,
\[
p-b_2\geqslant\left\lfloor\frac{r-2}{3}\right\rfloor.
\]
Since $p-b_2\geqslant0$,
\[
R^q(\pi_{2,1})_*\mathcal M_p=0\quad(q\geqslant1),
\]
whereas $(\pi_{2,1})_*\mathcal M_p$ has a filtration induced by
\[
0\longrightarrow\mathcal O_{X_1}(1)
\longrightarrow V_1^*
\longrightarrow\pi_{1,0}^*K_{X}
       \otimes\mathcal O_{X_1}(-2)
\longrightarrow0.
\]
After reindexing by $j=p-b_2-k$, the corresponding graded terms are
\begin{equation}\label{eq:bad-piece-on-x1}
\mathcal Q_{p,k}=
\pi_{1,0}^*\!\left(
K_{X}^{\,p-b_2-k}
\otimes\mathcal O_{X}(t)
\right)
\otimes\mathcal O_{X_1}(-r+3k),
\qquad 0\leqslant k\leqslant p-b_2.
\end{equation}
Set
\[
A_{p,k}:= 
K_{X}^{\,p-b_2-k}
\otimes\mathcal O_{X}(t),
\qquad
d_k:=-r+3k,
\]
so that
\[
\mathcal Q_{p,k}
=
\pi_{1,0}^*A_{p,k}\otimes\mathcal O_{X_1}(d_k).
\]
If $d_k\geqslant-1$, then the projective bundle formula and the
projection formula give
\[
R^1(\pi_{1,0})_*\mathcal Q_{p,k}
\simeq
A_{p,k}\otimes
R^1(\pi_{1,0})_*\mathcal O_{X_1}(d_k)
=0.
\]
Since $\pi_{1,0}$ has one-dimensional fibers and
$\dim X=2$, the Leray spectral sequence yields
\[
H^3(X_1,\mathcal Q_{p,k})
\simeq
H^2\bigl(X,R^1(\pi_{1,0})_*\mathcal Q_{p,k}\bigr)
=0.
\]
Consequently, only the indices satisfying $d_k\leqslant-2$ may
contribute; equivalently,
\[
0\leqslant k\leqslant
\left\lfloor\frac{r-2}{3}\right\rfloor.
\]

The only possible nonvanishing terms are therefore indexed by the fixed
finite set
\begin{equation}\label{eq:finite-bad-index-set}
0\leqslant k\leqslant
\left\lfloor\frac{r-2}{3}\right\rfloor.
\end{equation}
For such $k$, relative Serre duality for
$\pi_{1,0}\colon X_1=\mathbb P(T_{X})\to\mathbb X$
gives
\begin{align}\label{eq:r1-bad-piece-corrected}
R^1(\pi_{1,0})_*\mathcal Q_{p,k}
\simeq 
S^{r-3k-2}T_{X}
\otimes K_X^{\,p-b_2-k-1}
\otimes\mathcal{O}_X(t).
\end{align}
For each $k$ in~\eqref{eq:finite-bad-index-set}, the symmetric power in
\eqref{eq:r1-bad-piece-corrected} is independent of $p$, whereas the power of the
ample line bundle $K_X\simeq\mathcal{O}_X(d-4)$ tends to infinity with $p$.  Serre vanishing therefore gives
\[
H^2\bigl(X,R^1(\pi_{1,0})_*\mathcal Q_{p,k}\bigr)=0
\qquad(p\gg1).
\]
The Leray spectral sequence for $\pi_{1,0}$ now yields
$H^3(X_1,\mathcal Q_{p,k})=0$ for every graded term of the filtration.
Induction along the filtration gives
\[
H^3(X_1,(\pi_{2,1})_*\mathcal M_p)=0.
\]
Using $R^q(\pi_{2,1})_*\mathcal M_p=0$ once more, we obtain
\begin{equation}\label{H^3 of the left term}
H^3(X_2,\mathcal M_p)=0.
\end{equation}

Combining~\eqref{H^2 of the direct image of O(2p,p) tensoring O(-p tau)} and~\eqref{H^3 of the left term} we obtain
\(h^{2}(Z,\mathcal{O}_{X_2}(2p,p)|_{Z})=0\) for all sufficiently divisible \(p\gg 1\). Consequently,
\[
h^{0}\!\bigl(Z,\mathcal{O}_{X_2}(2p,p)|_{Z}\bigr)\geqslant\frac{(2u_1+u_2)^{3}\cdot Z}{3!}\,p^{3}+O(p^{2})\longrightarrow\infty
\quad\text{as }p\to\infty,
\]
which proves that \(\mathcal{O}_{X_2}(2,1)|_{Z}\) is big. 

Finally, by
$\mathcal{O}_{X_2}(\Gamma_2)\simeq\mathcal{O}_{X_2}(-1,1)$,
\[
\mathcal{O}_{X_2}(3)|_Z
\simeq
\mathcal{O}_{X_2}(2,1)|_Z\otimes\mathcal{O}_Z(2\Gamma_2).
\]
Since $\Gamma_2$ is effective, $\mathcal{O}_{X_2}(3)|_Z$ is big, and therefore so
is $\mathcal{O}_{X_2}(1)|_Z$. This completes the proof of the proposition.
\end{proof}

\subsection{Proof of Theorem~\ref{thm:GGL-bound} under Assumption~\eqref{ineq:Demailly-El Goul condition}}
\label{subsection: 5.2}
Assume that the pair $(m,t)$ in~\eqref{omega-1} satisfies inequality~\eqref{ineq:Demailly-El Goul condition}. Put $U=U^{m,t}$. The construction in Proposition~\ref{prop:extension and irred. component cpt case} gives a relative effective Cartier divisor
\[
\mathcal Z\subset\mathcal X_{2,U},\qquad
\rho\colon\mathcal Z\longrightarrow U,
\]
where $\rho$ is projective and flat, every fibre $Z_a$ is geometrically integral of dimension $3$, and
\[
\operatorname{div}(\omega_{1,a})=Z_a+b_1\Gamma_{a,2}
\qquad(a\in U).
\]
Here $m,t,b_1,b_2$ are fixed throughout the family. We use $\Pi_{i,j}$ for the relative tower projections over $U$ and $\pi_{i,j}$ for their restrictions to individual fibres.

For each positive integer $\ell$, define the line bundle
\[
\mathcal L_\ell
:=\left(
\mathcal O_{\mathcal X_{2,U}}(\ell)
\otimes\Pi_{2,0}^*\mathcal O_{\mathcal X_U}(-1)
\right)|_{\mathcal Z}.
\]
By Theorem~\ref{thm:extension of sections}, there is a dense Zariski-open subset $U_\ell\subset U$ on which $\rho_*\mathcal L_\ell$ is locally free and the natural maps
\[
\rho_*\mathcal L_\ell\otimes\kappa(a)
\longrightarrow H^0\bigl(Z_a,\mathcal L_\ell|_{Z_a}\bigr)
\]
are isomorphisms. Since there are only countably many positive integers $\ell$, we may choose a complex point
\[
a_*\in\bigcap_{\ell\geqslant1}U_\ell
\]
outside the countable union of proper closed subsets forming the Noether--Lefschetz locus. Thus
$\operatorname{Pic}(X_{a_*})=\mathbb Z[\mathcal O_{X_{a_*}}(1)]$.

The numerical assumption is unchanged at $a_*$, because $m,t$ and the degree $d$ are fixed. Proposition~\ref{prop:DEG-numerical-criterion} therefore implies that
$\mathcal O_{X_{a_*,2}}(1)|_{Z_{a_*}}$ is big. A sufficiently large power of a big line bundle has a nonzero section after tensoring with any fixed line bundle. We can consequently fix a positive integer $\ell$ and a section
\begin{equation}\label{omega_2,a_0}
0\neq\omega_{2,a_*}\in
H^0\!\left(Z_{a_*},
\left(\mathcal O_{X_{a_*,2}}(\ell)
\otimes\pi_{2,0}^*\mathcal O_{X_{a_*}}(-1)\right)|_{Z_{a_*}}\right).
\end{equation}
Because $a_*\in U_\ell$, this section is a vector in the fibre of the locally free sheaf $\rho_*\mathcal L_\ell$ at $a_*$. A local algebraic frame extends it to a section of $\rho_*\mathcal L_\ell$ over a Zariski-open neighbourhood $V\subset U_\ell$ of $a_*$, hence to
\[
\boldsymbol\omega_2\in H^0\bigl(\rho^{-1}(V),\mathcal L_\ell\bigr).
\]
By base change, the parameters where the fibre restriction vanishes identically form the zero locus of this section of $\rho_*\mathcal L_\ell$. Shrinking $V$ around $a_*$, we obtain an algebraic, hence holomorphic, family
\begin{equation}\label{family of omega_2}
\left\{\omega_{2,a}\;\middle|\;
0\neq\omega_{2,a}\in H^0\bigl(Z_a,\mathcal L_\ell|_{Z_a}\bigr),
\quad a\in V\right\}.
\end{equation}

Let
\[
\mathcal Z':=V(\boldsymbol\omega_2)\subset\mathcal Z|_V,
\qquad Z'_a:=V(\omega_{2,a})\subset Z_a.
\]
Since $Z_a$ is integral and $\omega_{2,a}$ is nonzero, its zero scheme is an effective Cartier divisor, possibly empty. In particular, every nonempty $Z'_a$ has dimension $2$, and in all cases $\dim Z'_a\leqslant2$.

Fix $a\in V$ and a nonconstant entire curve $f\colon\mathbb C\to X_a$. The fundamental vanishing theorem applied to $\omega_{1,a}$ gives
\(
f_{[2]}(\mathbb C)\subset Z_a\cup\Gamma_{a,2}.
\)
On the dense open subset where $f'(z)\neq0$, we have
\(
d\pi_{1,0}\bigl(f_{[1]}'(z)\bigr)=f'(z)\neq0.
\)
Thus the line represented by $f_{[2]}(z)$ is not vertical, so $f_{[2]}(z)\notin\Gamma_{a,2}$. The divisor decomposition then gives $f_{[2]}(z)\in Z_a$ on this dense open subset, and the closedness of $Z_a$ yields
\(
f_{[2]}(\mathbb C)\subset Z_a.
\)

The proof of the fundamental vanishing theorem also applies to sections of
$\bigl(\mathcal O_{X_{a,2}}(\ell)\otimes\pi_{2,0}^*\mathcal O_{X_a}(-1)\bigr)|_{Z_a}$
when the lifted entire curve is contained in $Z_a$; see~\cite[Section~4, proof of the Main Theorem]{Demailly-Elgoul2000}. This restricted-section form does not require $Z_a$ to be smooth. Applying it to $\omega_{2,a}$ gives
\(
f_{[2]}(\mathbb C)\subset Z'_a.
\)

Finally, let $\mathcal S\subset\mathcal X_{1,V}$ be the image of $\mathcal Z'$ under $\Pi_{2,1}$, endowed with its reduced structure. The projection is proper, so $\mathcal S$ is a closed algebraic subset. For each $a\in V$, put $S_a=(\mathcal S_a)_{\mathrm{red}}$. Since the projection is a morphism over $V$, we have, as closed subsets,
\[
S_a=\pi_{2,1}(Z'_a),\qquad
\dim S_a\leqslant2<3=\dim X_{a,1}.
\]
Thus the $S_a$ are proper subvarieties of $X_{a,1}$, given by the fibres of a closed algebraic family, and every nonconstant entire curve satisfies
\[
f_{[1]}(\mathbb C)
=\pi_{2,1}\bigl(f_{[2]}(\mathbb C)\bigr)\subset S_a.
\]
This proves both assertions of Theorem~\ref{thm:GGL-bound} under assumption~\eqref{ineq:Demailly-El Goul condition}. \qed

\subsection{Proof of  Theorem~\ref{thm:GGL-bound}  when~\eqref{ineq:Demailly-El Goul condition} Fails}\label{subsection: 5.3}
We now work on the complementary side of inequality~\eqref{ineq:Demailly-El Goul condition}, i.e.\ we assume
\begin{equation}\label{numerical condition, compact case}
t \geqslant (d-4)\frac{13c_1^{2}-9c_2}{12c_1^{2}}\,m
      = \frac{4d^{2}-68d+154}{12(d-4)}\,m \; \eqqcolon \; C(d)\,m  \qquad \text{[see \eqref{Chern numbers, cpt case}]}.
\end{equation}

By the fundamental vanishing theorem for entire curves, it suffices to produce two algebraically independent negatively twisted $2$-jet differentials $\omega_1$ and $\omega_2$. Given an initial family $\omega_1$, if its vanishing order $t$ satisfies $t \geqslant 8$, then differentiating $\omega_1$ along a suitably chosen slanted vector field of pole order $7$ (Proposition~\ref{vector field, compact case}) yields a second independent family $\omega_2$.

For \(d=17\) we have \(C(d)=77/78\). By Key Vanishing Lemma~\ref{lem-1.1}, there is no non-zero negatively twisted holomorphic \(2\)-jet differential satisfying condition~\eqref{numerical condition, compact case} for \(3 \leqslant m \leqslant 7\). Consequently, we must have \(m \geqslant 8\). In this range, condition~\eqref{numerical condition, compact case} implies \(t \geqslant 8\), so Proposition~\ref{vector field, compact case} supplies the required \(\omega_2\).
This completes the proof.
\qed

\begin{rem}\label{rem:vanishing results for cpt case d >= 18}\rm
As an alternative to P\u{a}un's approach using slanted vector fields of pole order $3$ in~\cite[pp. 886-889]{mihaipaun2008} for the proof of Theorem~\ref{thm:GGL-bound} when $d \geqslant 18$, one may instead employ directly the slanted vector fields of pole order $7$.

For $d \geqslant 21$, this alternative approach follows directly from a combination of the slanted vector fields technique of Siu and the zero-locus technique of Demailly--El Goul, as employed in this paper. 
For degrees $18 \leqslant d \leqslant 20$, the proof can be completed by establishing the following vanishing results, which can be verified using the supplementary Maple code mentioned in Section~\ref{sect:3}:

\smallskip \noindent 
{\bf Fact.}\footnote{The associated Maple exports are available at \url{https://github.com/LeiHou-code/Hyperbolicity-Dimension-2-Conventions-Facts-2026}.}\ \ 
{\it Let $X \subset \mathbb{P}^3$ be a generic smooth surface of degree $d$. Then:
\begin{itemize}
    \item For $d = 20$, $H^{0}\bigl(X,\;E_{2,3}T_{X}^{*} \otimes \mathcal{O}(-7)\bigr) = 0$.
    \item For $d = 19$, $H^{0}\bigl(X,\;E_{2,m}T_{X}^{*} \otimes \mathcal{O}(-t)\bigr) = 0$ for all $(m,t)$ in
    \(
    \{(3,6),\, (4,7)\}.
    \)
    \item For $d = 18$, $H^{0}\bigl(X,\;E_{2,m}T_{X}^{*} \otimes \mathcal{O}(-t)\bigr) = 0$ for all $(m,t)$ in
    \(
    \{(3,5),\,(4,6),\,(5,7)\}.
    \)
\end{itemize}}
\end{rem}

\subsection{Proof of Theorem~\ref{thm:log-hyperbolic}}\label{subsection: 5.4}
Let $\mathcal{C}\subset\mathbb{P}^2$ be a very generic smooth plane curve
of degree $d\geqslant12$, chosen in the Zariski-open loci required by the
constructions below. Suppose that there exists a nonconstant entire curve
\[
f\colon\mathbb{C}\longrightarrow\mathbb{P}^2\setminus\mathcal{C}.
\]
Since $N_f^{[1]}(r,\mathcal{C})=0$, the algebraically degenerate case is
excluded by Lemma~\ref{lem:SMT-algebraically-degenerate}.
We may therefore assume that $f$ is algebraically nondegenerate.

McQuillan's theorem (Theorem~\ref{thm:McQuillan}) and the construction of slanted vector fields extend directly to the logarithmic setting; see \cite[Section 2]{ElGoul2003}  and \cite[Section 3]{Rousseau2009}.

First, by Proposition~\ref{existence of 2 jets, logarithmic case}, there exist positive integers \(m,t\) and a nonzero section
\[
\omega_1 \in H^0\!\bigl(\mathbb{P}^2,\;E_{2,m}T_{\mathbb{P}^2}^*(\log\mathcal{C})\otimes\mathcal{O}_{\mathbb{P}^2}(-t)\bigr).
\]

McQuillan's result~\cite{Mcquillan1998} shows that if the lift
$f_{[1]}$ to the first level $X_1$ of the Demailly--Semple tower of
$(\mathbb{P}^2,\mathcal{C})$ is algebraically degenerate, then
$f(\mathbb{C})$ is algebraically degenerate in $\mathbb{P}^2$.
Since this case has already been excluded, it suffices to prove that
$f_{[1]}$ is algebraically degenerate.

A numerical criterion analogous to Proposition~\ref{prop:DEG-numerical-criterion} is provided by El~Goul~\cite[Theorem 1.3.3]{ElGoul2003}. It allows us to assume
\begin{equation}\label{numerical condition, logarithmic case}
t \geqslant (d-3)\frac{13\bar{c}_1^{2}-9\bar{c}_2}{12\bar{c}_1^{2}}\,m 
      = \frac{4d^{2}-51d+90}{12(d-3)}\,m \; \eqqcolon \; \overline{C}(d)\,m  \qquad \text{[by \eqref{Chern numbers, log case}]}.
\end{equation}

Since \(\operatorname{Pic}(\mathbb{P}^2)\cong\mathbb{Z}\), we can mimic and simplify the arguments of the compact case. For brevity, we work directly with a given \(\omega_1\) on the chosen curve \(\mathcal{C}\) and omit the step of extending it to a holomorphic family.

Under the condition~\eqref{numerical condition, logarithmic case} we now examine the remaining possibilities for \((m,t)\) when \(d=13\) and \(d=12\).

For \(d=13\) we have \(\overline{C}(d)=\frac{103}{120}\). By Key Vanishing Lemma~\ref{lem-1.2}, there is no non-zero negatively twisted logarithmic holomorphic \(2\)-jet differential satisfying condition~\eqref{numerical condition, logarithmic case} for \(3 \leqslant m \leqslant 8\). Consequently, we may assume \(m \geqslant 9\). In this range, condition~\eqref{numerical condition, logarithmic case} implies \(t \geqslant 8\), so Proposition~\ref{vector field, logarithmic case} supplies the required second differential \(\omega_2\).

For \(d=12\) we have \(\overline{C}(d)=\frac{1}{2}\). Key Vanishing Lemma~\ref{lem-1.2} ensures that no non-zero negatively twisted logarithmic \(2\)-jet differential satisfying condition~\eqref{numerical condition, logarithmic case} exists for \(3 \leqslant m \leqslant 14\). Hence, we may restrict to \(m \geqslant 15\). For such \(m\), condition~\eqref{numerical condition, logarithmic case} yields \(t \geqslant 8\); therefore Proposition~\ref{vector field, logarithmic case} again provides the desired \(\omega_2\).

The preceding argument excludes every nonconstant entire curve in
$\mathbb P^2\setminus\mathcal C$, so the complement is Brody hyperbolic.
Moreover, the smooth plane curve $\mathcal C$ has genus
\[
g(\mathcal C)=\frac{(d-1)(d-2)}{2}\geqslant 2,
\]
and is therefore Brody hyperbolic.  We may now apply Green's criterion:
if $X$ is a smooth projective variety, $D=\bigcup_{i=1}^m D_i$ is a
union of Cartier divisors, the complement $X\setminus D$ is Brody
hyperbolic, and every stratum
\[
\bigcap_{i\in I}D_i\setminus\bigcup_{j\in J}D_j,
\qquad
I\sqcup J=\{1,\ldots,m\},
\]
is Brody hyperbolic, then $X\setminus D$ is hyperbolically embedded in
$X$; see~\cite{Green1977} and the precise restatement
\cite[Theorem~4.1]{JavanpeykarLevin2022}.  In the present case,
$D=\mathcal C$ has only one irreducible component, so the only
nontrivial boundary stratum is $\mathcal C$ itself, whose Brody
hyperbolicity was just noted.  Hence
$\mathbb P^2\setminus\mathcal C$ is hyperbolically embedded in
$\mathbb P^2$, and in particular is Kobayashi hyperbolic.  This
completes the proof.
\qed

\begin{rem}\label{rem:vanishing results for log case d >= 14}\rm
As an alternative to Rousseau's approach using slanted vector fields of pole order $3$ in~\cite[Proposition 17]{Rousseau2009} for the proof of the Kobayashi hyperbolicity of the complement $\mathbb{P}^2 \setminus \mathcal{C}$ of a very generic algebraic curve $\mathcal{C} \subset \mathbb{P}^2$ of degree $d \geqslant 14$, one may instead employ directly the slanted vector fields of pole order $7$.

For $d \geqslant 18$, this alternative approach follows directly from a combination of the slanted vector fields technique of Siu and the zero-locus technique of Demailly--El Goul, as employed in this paper. For degrees $14 \leqslant d \leqslant 17$, the proof can be completed by establishing the following vanishing results, which can be verified using the supplementary Maple code mentioned in Section~\ref{sect:proof of lemma 2}:

\smallskip \noindent 
{\bf Fact.}\footnote{The associated Maple exports are available at \url{https://github.com/LeiHou-code/Hyperbolicity-Dimension-2-Conventions-Facts-2026}.}\ \ 
{\it Let $\mathcal{C} \subset \mathbb{P}^2$ be a generic smooth curve of degree $d$. Then:
\begin{itemize}
    \item For $d = 17$, $H^{0}\bigl(\mathbb{P}^{2},\;E_{2,3}T_{\mathbb{P}^{2}}^{*}(\log \mathcal{C}) \otimes \mathcal{O}(-7)\bigr) = 0$.
    \item For $d = 16$, $H^{0}\bigl(\mathbb{P}^{2},\;E_{2,3}T_{\mathbb{P}^{2}}^{*}(\log \mathcal{C}) \otimes \mathcal{O}(-6)\bigr) = 0$.
    \item For $d = 15$, $H^{0}\bigl(\mathbb{P}^{2},\;E_{2,m}T_{\mathbb{P}^{2}}^{*}(\log \mathcal{C}) \otimes \mathcal{O}(-t)\bigr) = 0$ for all $(m,t)$ in
    \(
    \{(3,5),\,(4,7)\}.
    \)
    \item For $d = 14$, $H^{0}\bigl(\mathbb{P}^{2},\;E_{2,m}T_{\mathbb{P}^{2}}^{*}(\log \mathcal{C}) \otimes \mathcal{O}(-t)\bigr) = 0$ for all $(m,t)$ in
    \(
    \{(3,4),\,(4,5),\,(5,7)\}.
    \)
\end{itemize}}
\end{rem}

\medskip

\subsection{Proof of Theorem~\ref{SMT for curves of d geqslant 12}}\label{subsect:SMT}

We first prove the estimate for algebraically nondegenerate entire curves,
following the strategy developed in our earlier
work~\cite{Hou-Huynh-Merker-Xie2025}.
Let $U_d$ denote the nonempty Zariski-open parameter locus on which the
constructions below apply. Fix $a\in U_d$ and write
$\mathcal{C}=\mathcal{C}_a$.
Throughout the jet-differential argument below, we assume that
\(
f\colon\mathbb{C}\to\mathbb{P}^2
\)
is algebraically nondegenerate.

The additional assertion for arbitrary nonconstant entire curves over
very generic parameters will follow at the end from
Lemma~\ref{lem:SMT-algebraically-degenerate}.

Propositions~\ref{existence of 2 jets, logarithmic case} and~\ref{prop:extension and irred. component log case} guarantee the existence of an irreducible, negatively twisted invariant logarithmic \(2\)-jet differential
\begin{equation}\label{omega1, m, t}
\omega_1 \in H^{0}\!\bigl(\mathbb{P}^{2},\; E_{2,m}T^{*}_{\mathbb{P}^{2}}(\log\mathcal{C}) \otimes \mathcal{O}_{\mathbb{P}^{2}}(-t)\bigr),
\end{equation}
with weighted degree \(m \geqslant 1\) and vanishing order \(t \geqslant 1\).

If \(f^{*}\omega_1 \not\equiv 0\),  Theorem~\ref{smt-form-logarithmic-diff-jet} immediately gives the desired Second Main Theorem with a finite constant \(C_d\).

In particular, a Riemann–Roch computation based on the filtration~\eqref{filtration-of-E2m, log} and a Bogomolov-type vanishing theorem show that for any rational number \(\delta \in (0,\frac{1}{3})\) and sufficiently divisible \(m \gg 1\),
\begin{multline*}
h^{0}\!\bigl(E_{2,m}T_{\mathbb{P}^{2}}^{*}(\log\mathcal{C}) \otimes \overline{K}_{\mathbb{P}^{2}}^{\,-\delta m}\bigr) \geqslant \frac{m^{4}}{648}\bigl((54\delta^{2}-48\delta+13)\bar{c}_{1}^{2} - 9\bar{c}_{2}\bigr) + O(m^{3})
\quad \text{[cf.\ \cite[Proposition 4.2]{Hou-Huynh-Merker-Xie2025}]}.
\end{multline*}
The leading coefficient (after removing the common factor \(\frac{1}{648}\)) is
\[
\varphi(\delta) \coloneqq (54\delta^{2} - 48\delta + 13)\bar{c}_{1}^{2} - 9\bar{c}_{2}
        = 54\bar{c}_{1}^{2}\,\delta^{2} - 48\bar{c}_{1}^{2}\,\delta + 13\bar{c}_{1}^{2} - 9\bar{c}_{2}.
\]
Note that \(13\bar{c}_{1}^{2} - 9\bar{c}_{2} > 0\) for \(d \geqslant 11\). The two roots of the equation $\varphi(\delta)=0$ are
\[
\delta_{1}= \frac{8 - \sqrt{54\bar{c}_{2}/\bar{c}_{1}^{2} - 14}}{18}>0,
\qquad 
\delta_{2}= \frac{8 + \sqrt{54\bar{c}_{2}/\bar{c}_{1}^{2} - 14}}{18}>0.
\]

Hence \(\varphi(\delta) > 0\) for all \(\delta \in (0,\delta_{1})\). Consequently, in the existence result for \(\omega_{1}\) stated in~\eqref{omega1, m, t}, the ratio \(t/m\) can be taken arbitrarily close to
\begin{equation}\label{ratio of omega_1}
(d-3)\delta_{1} = \frac{8(d-3) - \sqrt{2(4d-3)(5d-6)}}{18}
\qquad\text{[use \eqref{Chern numbers, log case}]}.
\end{equation}
We may therefore assume that the ratio \(t/m\) of \(\omega_{1}\) in~\eqref{omega1, m, t} satisfies
\[
(d-3)\delta_{1} - \varepsilon' \;<\; \frac{t}{m},
\]
where \(\varepsilon' > 0\) is a fixed arbitrarily small constant. If \(f^{*}\omega_{1} \not\equiv 0\), Theorem~\ref{smt-form-logarithmic-diff-jet} yields the estimate~\eqref{eq:smt for curves of d geqslant 12} with constant \(C_{d}=m/t\) satisfying
\begin{equation}\label{C_d from omega_1}
      0  
      \;<\; C_{d} = \frac{m}{t} 
      \;<\; \frac{1}{(d-3)\delta_{1} - \varepsilon'},
\end{equation}
where ${1}/{((d-3)\delta_{1} - \varepsilon')}$ can be taken very close to 
\begin{equation*}
    \frac{1}{(d-3)\delta_{1}}
    =
    \frac{18}{8(d-3) - \sqrt{2(4d-3)(5d-6)}}. 
\end{equation*}

Thus we may assume \(f^{*}\omega_{1} \equiv 0\) in the remainder of the proof. The subsequent argument splits according to the values of \((m,t)\).

\medskip
\noindent\textbf{Case 1: \(\omega_{1}\) has high vanishing order.} When \(m = 3,4,5\) and \(t \geqslant 4\), we apply Siu’s strategy of slanted vector fields \cite{Siu2004}, which in the two‑dimensional logarithmic setting was carried out by Rousseau in the proof of \cite[Proposition 17]{Rousseau2009}. Differentiating \(\omega_{1}\) by a slanted vector field twisted by \(\mathcal{O}_{\mathbb{P}^{2}}(3)\) produces an ``independent'' nonzero invariant logarithmic \(2\)-jet differential
\[
\omega_{2} \in H^{0}\!\bigl(\mathbb{P}^{2},\; E_{2,m}T^{*}_{\mathbb{P}^{2}}(\log\mathcal{C}) \otimes \mathcal{O}_{\mathbb{P}^{2}}(-t+3)\bigr).
\]
When $m\geqslant 6$ and \(t \geqslant 8\), we may instead differentiate \(\omega_{1}\) by a slanted vector field twisted by \(\mathcal{O}_{\mathbb{P}^{2}}(7)\) (see Proposition~\ref{vector field, logarithmic case}), obtaining an ``independent'' nonzero invariant logarithmic \(2\)-jet differential
\[
\omega_{2} \in H^{0}\!\bigl(\mathbb{P}^{2},\; E_{2,m}T^{*}_{\mathbb{P}^{2}}(\log\mathcal{C}) \otimes \mathcal{O}_{\mathbb{P}^{2}}(-t+7)\bigr).
\]

We then distinguish two subcases:
\begin{itemize}
    \item If \(f^{*}\omega_{2} \not\equiv 0\), Theorem~\ref{smt-form-logarithmic-diff-jet} directly gives a Second Main Theorem type estimate.
    \item If \(f^{*}\omega_{1} \equiv 0\) and \(f^{*}\omega_{2} \equiv 0\), the algebraic independence of \(\omega_{1}\) and \(\omega_{2}\) implies that the lift \(f_{[1]}(\mathbb{C}) \subset X_{1}\) is algebraically degenerate. Hence McQuillan's Theorem~\ref{McQuillan's result} applies.
\end{itemize}

\medskip
\noindent\textbf{Case 2: \(\omega_{1}\) has low vanishing order.}  
Assume
\begin{equation}\label{low vanishing cond}
0 < t < (d-3)\,\frac{13\bar{c}_{1}^{2}-9\bar{c}_{2}}{12\bar{c}_{1}^{2}}\,m
      = \frac{4d^{2}-51d+90}{12(d-3)}\,m \;\eqqcolon\; \overline{C}(d)\,m. \qquad \text{[see \eqref{numerical condition, logarithmic case}]}
\end{equation}
Following a strategy of Demailly–El~Goul \cite[Proposition 3.4]{Demailly-Elgoul2000} (see also \cite[Theorem 1.3.3]{ElGoul2003}), a Riemann–Roch computation yields the existence of a nonzero invariant logarithmic \(2\)-jet differential \(\omega_{2}\) with negative twist. The definition of \(\omega_{2}\) is subtle: it is not defined on the whole of \(\mathbb{P}^{2}\), but only on the zero locus \(\{\omega_{1}=0\}\) inside the second level \(X_{2}\) of the Demailly–Semple tower for \((\mathbb{P}^{2},\mathcal{C})\). We again distinguish two subcases:

\begin{itemize}
    \item If \(f^{*}\omega_{2} \equiv 0\), the simultaneous vanishing of \(f^{*}\omega_{1}\) and \(f^{*}\omega_{2}\) forces the lift \(f_{[1]}(\mathbb{C})\) to be algebraically degenerate in the first level \(X_{1}\) of the logarithmic Demailly–Semple tower. McQuillan's Theorem~\ref{McQuillan's result} then applies.
    
    \item If \(f^{*}\omega_{2} \not\equiv 0\), a Second Main Theorem type estimate can be obtained by an argument analogous to that in our earlier work \cite[Section 6]{Hou-Huynh-Merker-Xie2025}. The reasoning proceeds as follows.
\end{itemize}

As in the proof of Proposition~\ref{prop:DEG-numerical-criterion}, one can show that the zero divisor of \(\omega_{1}\) on \(X_{2}\) decomposes as
\[
\{\omega_{1}=0\} = Z + b_{1}\Gamma_{2},
\]
where \(Z\subset X_{2}\) is irreducible and reduced.  
We now  determine an effective small rational number \(\tau > 0\) such that the \(\mathbb{Q}\)-line bundle
\[
\bigl( \mathcal{O}_{X_{2}}(2,1) \otimes \pi_{2,0}^{*}\mathcal{O}_{\mathbb{P}^{2}}(-\tau) \bigr)\big|_{Z}
\]
is big. (Here we use the same notation as in \eqref{eq:O(a_1,a_2)} for the logarithmic case). This is achieved by a similar Riemann–Roch computation and the use of Proposition~\ref{pro:the higher direct images vanishing} as in \cite[Proposition 6.2]{Hou-Huynh-Merker-Xie2025}.

\begin{pro}\label{prop: bigness on Z}
If \(0 < \frac{t}{m} < \overline{C}(d)\) and \(0 \leqslant \tau < \tau_{1}\!\bigl(\frac{t}{m}\bigr)\) with \(\tau\in\mathbb{Q}\), where
{\small
\begin{align}\label{what is tau}
\tau_{1}\!\Bigl(\frac{t}{m}\Bigr)
&= -2\bar{c}_{1} - \frac{3}{2}\,\frac{t}{m}
   - \frac{\sqrt{3}}{6}
     \sqrt{\,27\Bigl(\frac{t}{m}\Bigr)^{2} + 24\bar{c}_{1}\frac{t}{m} + 36\bar{c}_{2} - 4\bar{c}_{1}^{2}} \\
&= 2(d-3) - \frac{3}{2}\,\frac{t}{m}
   - \frac{\sqrt{3}}{6}
     \sqrt{\,27\Bigl(\frac{t}{m}\Bigr)^{2} - 24(d-3)\frac{t}{m} + 4\bigl(8d^{2}-21d+18\bigr)}, \notag
\end{align}}
then the restricted \(\mathbb{Q}\)-line bundle
\[
\bigl( \mathcal{O}_{X_{2}}(2,1) \otimes \pi_{2,0}^{*}\mathcal{O}_{\mathbb{P}^{2}}(-\tau) \bigr)\big|_{Z}
\]
is big on \(Z\). \qed
\end{pro}

We first make precise the evaluation of a section that is defined only on
$Z$.

\begin{lem}[Tautological evaluation of a restricted section]
\label{lem:one-component-tautological-evaluation}
Let $Z\subset X_2$ be an effective Cartier divisor and
\[
 \sigma\in H^0\!\left(
 Z,\mathcal O_{X_2}(2p,p)\otimes
 \pi_{2,0}^*\mathcal O_{\mathbb P^2}(-\widetilde t)\big|_Z
 \right),
 \qquad p,\widetilde t\in\mathbb Z_{>0}.
\]
Suppose that $g\colon\mathbb C\to\mathbb P^2$ is nonconstant,
$g(\mathbb C)\not\subset\mathcal C$, and its logarithmic lift extends and
satisfies $g_{[2]}(\mathbb C)\subset Z$.  On
$\mathbb C\setminus g^{-1}(\mathcal C)$, regard
$dg_{[1]}\colon T_{\mathbb C}\to g_{[1]}^*V_1$ as a differential in the
directed structure.  Wherever it is nonzero, its image is the tautological
line $g_{[2]}^*\mathcal O_{X_2}(-1)$.  It therefore defines, by meromorphic
continuation across the remaining discrete set, the intrinsic tautological
derivative
\[
 \mathcal Dg_{[1]}\in H^0_{\mathrm{mer}}\!\left(
 \mathbb C,K_{\mathbb C}\otimes
 g_{[2]}^*\mathcal O_{X_2}(-1)\right).
\]
The inclusion
\[
 \iota_p\colon
 \mathcal O_{X_2}(2p,p)
 =\mathcal O_{X_2}(3p)(-2p\Gamma_2)
 \hookrightarrow\mathcal O_{X_2}(3p)
\]
together with this derivative defines canonically
\begin{equation}\label{eq:one-component-canonical-evaluation}
 s_{\sigma,g}:=
 \left\langle g_{[2]}^*(\iota_p\sigma),
 (\mathcal Dg_{[1]})^{\otimes3p}\right\rangle
 \in H^0_{\mathrm{mer}}\!\left(
 \mathbb C,K_{\mathbb C}^{\otimes3p}\otimes
 g^*\mathcal O_{\mathbb P^2}(-\widetilde t)\right).
\end{equation}
It is independent of every ambient local lifting of $\sigma$.  If it is not
identically zero, then
\begin{equation}\label{eq:one-component-canonical-pole-bound}
 \operatorname{div}_{\infty}(s_{\sigma,g})
 \leqslant3p(g^*\mathcal C)_{\mathrm{red}}
\end{equation}
and
\begin{equation}\label{eq:one-component-canonical-proximity}
 \int_0^{2\pi}\log^+
 \|s_{\sigma,g}(re^{\sqrt{-1}\theta})\|\frac{d\theta}{2\pi}
 =O\!\left(\log T_g(r)+\log r\right)\ \|.
\end{equation}
\end{lem}

\begin{proof}
The factorization of $g_{[2]}$ through $Z$ makes
$g_{[2]}^*(\iota_p\sigma)$ intrinsic.  Pairing it with the indicated power
of the tautological derivative gives
\eqref{eq:one-component-canonical-evaluation}.  If ambient local liftings are
used to represent $\sigma$, their differences lie in the ideal of $Z$ and
therefore pull back to zero.

Equip the line bundles with smooth Hermitian metrics.  Since $Z$ is
projective, the norm of the holomorphic section $\iota_p\sigma$ is bounded
on $Z$.  Hence
\[
 \|s_{\sigma,g}\|
 \leqslant C\,\|\mathcal Dg_{[1]}\|^{3p}
\]
for a uniform constant $C$.

Near $g^{-1}(\mathcal C)$, take a logarithmic coordinate $x$ with
$\mathcal C=\{x=0\}$.  The boundary component of the derivative of the
regular first lift is $(x\circ g)'/(x\circ g)$, while its vertical
projective-coordinate components are holomorphic. At an isolated zero of $g'$ outside
$g^{-1}(\mathcal C)$, the regularized Semple lift gives only a removable
singularity or a zero.  If the critical point lies over $\mathcal C$, the
logarithmic component $(x\circ g)'/(x\circ g)$ may still have a simple pole;
that pole is already counted by
\eqref{eq:one-component-canonical-pole-bound}, and no additional Semple-type
pole occurs.
Thus
$\mathcal Dg_{[1]}$ has at most a simple pole along
$(g^*\mathcal C)_{\mathrm{red}}$, which proves
\eqref{eq:one-component-canonical-pole-bound} after taking the $3p$-th
tensor power.

Choose a finite logarithmic base cover and, above it, finitely many
relatively compact Semple charts.  A partition of unity is used only to
estimate the norm of the already-defined global section.  On each chart,
the components of $\mathcal Dg_{[1]}$ are logarithmic derivatives of local
coordinate functions.  A projective Semple coordinate $u$ is a rational
function on $X_1$ and, in a logarithmic frame, a ratio of first logarithmic
derivatives; therefore
\[
 T_{u\circ g_{[1]}}(r)=O\bigl(T_g(r)+\log r\bigr).
\]
For a partition function $\chi$ supported in a relatively compact chart,
choose $a\notin\overline{u(\operatorname{supp}\chi)}$.  On that support,
$u-a$ is bounded above and away from zero, so
\[
 (u\circ g_{[1]})'
 =(u\circ g_{[1]}-a)
 \frac{(u\circ g_{[1]})'}{u\circ g_{[1]}-a}.
\]
The logarithmic-derivative lemma applied to $u\circ g_{[1]}-a$ controls the
vertical component by
\[
 O\bigl(\log^+T_g(r)+\log r\bigr).
\]
Summing over the finite
cover, together with the analogous horizontal estimates, gives
\eqref{eq:one-component-canonical-proximity}.  This is precisely the
partition-of-unity argument of
\cite[Thm.~3.1, (3.3)--(3.6)]{Huynh-Vu-Xie-2017}, with its
Semple-coordinate adaptation made explicit.
\end{proof}

The preceding lemma gives the following restricted-section form of the
Second Main Theorem.

\begin{thm}\label{pro:SMT on Z}
Let \(g \colon \mathbb{C} \to \mathbb{P}^{2}\) be an entire curve whose lift \(g_{[2]} \colon \mathbb{C} \to X_{2}\) satisfies \(g_{[2]}(\mathbb{C}) \subset Z\). Suppose there exists a section
\begin{equation*}\label{sigma of thm 5.3}
    \sigma \in H^0\Big(Z, \big( \mathcal{O}_{X_2}(2,1)^{\tilde{m}} \otimes \pi_{2,0}^*\mathcal{O}_{\mathbb{P}^2}(-\tilde{t}) \big)|_Z \Big) \cong H^0\Big(Z, \big( \mathcal{O}_{X_2}(3) \otimes \mathcal{O}(- 2 \Gamma_2) \big)^{\tilde{m}} \otimes \pi_{2,0}^*\mathcal{O}_{\mathbb{P}^2}(-\tilde{t}) \big|_Z \Big),
\end{equation*}
such that the canonical evaluation $s_{\sigma,g}$ of
Lemma~\ref{lem:one-component-tautological-evaluation}, with
$p=\tilde m$, is not identically zero.
Then the following Second Main Theorem type estimate holds:
\[
T_{g}(r) \leqslant \frac{3\tilde{m}}{\tilde{t}}\, N_{g}^{[1]}(r,\mathcal{C}) + o\bigl(T_{g}(r)\bigr) \qquad\|. 
\]
\end{thm}

\begin{proof}
Give $K_{\mathbb C}$ the flat metric defined by $dz$ and
$\mathcal O_{\mathbb P^2}(-1)$ the metric dual to the Fubini--Study metric.
Poincar\'e--Lelong and Jensen's formula applied to the intrinsic section
$s_{\sigma,g}$ give
\[
 \widetilde t\,T_g(r)
 \leqslant N\!\left(r,\operatorname{div}_{\infty}(s_{\sigma,g})\right)
 +\int_0^{2\pi}\log^+
 \|s_{\sigma,g}(re^{\sqrt{-1}\theta})\|\frac{d\theta}{2\pi}+O(1).
\]
Equations~\eqref{eq:one-component-canonical-pole-bound} and
\eqref{eq:one-component-canonical-proximity}, with $p=\tilde m$, yield the
stated coefficient $3\tilde m/\tilde t$.
\end{proof}

Summarizing, Proposition~\ref{prop: bigness on Z} yields, for any rational \(\tau \in (0,\tau_{1}(t/m))\), a nontrivial global section  
\[
\sigma \in H^{0}\!\Bigl(Z,\; 
        \bigl( \mathcal{O}_{X_{2}}(2,1) \otimes \pi_{2,0}^{*}\mathcal{O}_{\mathbb{P}^{2}}(-\tau) \bigr)^{\otimes\ell} \Bigr)
\]
for a sufficiently divisible integer \(\ell \gg 1\) (depending on \(\tau\)).

Recall that the assumption of the present subcase is
\(f^{*}\omega_{1}\equiv0\), while the canonical evaluation of \(\sigma\)
along \(f_{[2]}\) is nonzero.  Thus Theorem~\ref{pro:SMT on Z} gives the
estimate~\eqref{eq:smt for curves of d geqslant 12} with constant
(see~\eqref{what is tau})
\[
C_{d} = \frac{3}{\tau} \;>\; \frac{3}{\tau_{1}(t/m)}.
\]
Since \(\tau\) can be chosen arbitrarily close to \(\tau_{1}(t/m)\), the estimate~\eqref{eq:smt for curves of d geqslant 12} holds for any constant \(C_{d} > \frac{3}{\tau_{1}(t/m)}\). 

\medskip

Now we compute the constant \(C_d\) appearing in the estimate~\eqref{eq:smt for curves of d geqslant 12} for \(d \geqslant 14\).

A direct computation gives
\[
\overline{C}(d)=\frac{4d^{2}-51d+90}{12(d-3)}
               =\frac{d-3}{3}-\frac{9}{4(d-3)}-\frac{9}{4}
         > \frac{7}{6},
\]
where the last inequality follows from the assumption \(d \geqslant 14\).

Let \(\frac{7}{6} < \gamma < \overline{C}(d)\) be a parameter that will be chosen later.

If \(t \leqslant \gamma m\) (which falls into \textbf{Case 2}), we obtain the estimate~\eqref{eq:smt for curves of d geqslant 12} with constant
\[
C_d = \frac{3}{\tau_1(\gamma)}
    = \frac{3}{2\left(d-3-\frac{3}{4}\gamma
                 -\frac{\sqrt{3}}{12}
                   \sqrt{27\gamma^{2}-24(d-3)\gamma+4\bigl(8d^{2}-21d+18\bigr)}\right)}
    > \frac{1}{d-3},
\]
where the lower bound \(\frac{1}{d-3}\) is precisely the constant that would follow from Theorem~\ref{McQuillan's result}.

If \(t \geqslant \frac{6\gamma-7}{6}\,m+7\) and \(m \geqslant 6\) (or if \(m = 3, 4, 5\) and \(t \geqslant \frac{3\gamma-3}{3}\,m+3\)), then we are in \textbf{Case 1} and the estimate~\eqref{eq:smt for curves of d geqslant 12} holds with constant 
\(C_d = \frac{6}{6\gamma-7}\) (respectively \(C_d = \frac{3}{3\gamma-3}\)). A straightforward computation shows that for \(\frac{7}{6} < \gamma < \overline{C}(d)\) and \(d \geqslant 14\),
\begin{equation}\label{constants of Case I, d geqslant 14}
    \frac{6}{6\gamma-7} \;>\; \frac{3}{3\gamma-3} \;>\; \frac{1}{d-3}.
\end{equation}

Observe the following:
\begin{itemize}
    \item The pairs \((m,t)=(6,\lfloor 6\gamma\rfloor)\) and \((m,t)=(3,\lfloor 3\gamma\rfloor)\) satisfy \(t \leqslant \gamma m\).
    \item The pairs \((m,t)=(6,\lceil 6\gamma\rceil)\) and \((m,t)=(3,\lceil 3\gamma\rceil)\) satisfy respectively  
          \(t \geqslant \frac{6\gamma-7}{6}\,m+7\) and \(t \geqslant \frac{3\gamma-3}{3}\,m+3\).
    \item The lines \(t = \gamma m\) and \(t = \frac{6\gamma-7}{6}m+7\) intersect at \((6,6\gamma)\);  
          the lines \(t = \gamma m\) and \(t = \frac{3\gamma-3}{3}m+3\) intersect at \((3,3\gamma)\).
\end{itemize}
Consequently, every possible value of \((m,t)\) is covered by the above discussion.  
For a convenient uniform bound, let \(\gamma_0\) be the unique solution of
\[
\frac{6}{6\gamma-7}= \frac{3}{\tau_1(\gamma)} .
\]
A straightforward calculation yields
\begin{equation}\label{gamma0}
\gamma_0(d)=\frac{32d-33-\sqrt{640d^{2}-1248d+1017}}{72},
\end{equation}
which lies in the interval $ (\frac{7}{6}, \overline{C}(d))$. 
At this admissible split,
\[
\frac{6}{6\gamma_0(d)-7}= \frac{3}{\tau_1(\gamma_0(d))}
    =\frac{72}{32d-117-\sqrt{640d^{2}-1248d+1017}} .
\]
A direct comparison shows that this quantity is strictly larger, for every
$d\geqslant14$, than the other high-slope constant in
\eqref{constants of Case I, d geqslant 14} and the limiting initial-section
constant in \eqref{C_d from omega_1}.  Choosing the Riemann--Roch parameter
sufficiently close to its endpoint therefore gives the uniform value
\[
 C_d=\frac{72}{32d-117-\sqrt{640d^{2}-1248d+1017}}
 \qquad(d\geqslant14).
\]

\medskip
\noindent
\textbf{Case 3: Remaining possibilities.}
For \(d = 12,13\), there remain a few pairs \((m,t)\) not covered by the previous two cases.  
 Key Vanishing Lemma~\ref{lem-1.2} excludes the existence of any nonzero \(\omega_{1}\) for those pairs.  
Note that vanishing for \((m,t)\) implies vanishing for all \((m,t')\) with
\(t'>t\).  In enumerating the residual pairs we apply this monotonicity when \(d=12,13\).

We now compute the constant \(C_d\) in the estimate~\eqref{eq:smt for curves of d geqslant 12} for \(d=12\) and \(13\).

\medskip\noindent
\textbf{Case \(d=13\).}
A direct computation gives
\[
\frac{1}{(d-3)\delta_1}
=
\frac{18}{80-7\sqrt{118}}
\approx 4.545
\qquad(d=13).
\]
Since
\[
\frac67<\overline{C}(13)=\frac{103}{120},
\qquad
\frac{103}{120}-\frac67=\frac{1}{840},
\]
and, by~\eqref{what is tau}, the function \(\tau_1\) is positive,
strictly decreasing, and continuous on
\((0,\overline{C}(13))\), for every \(\eta>0\) we may choose
\[
\gamma=\frac67+\varepsilon,
\qquad 0<\varepsilon<\frac{1}{840},
\]
with \(\varepsilon\) sufficiently small that
\[
10
<
\frac{3}{\tau_1(\tfrac67)}
<
\frac{3}{\tau_1(\gamma)}
<
\frac{3}{\tau_1(\tfrac67)}+\eta.
\]

If \(t\leqslant\gamma m\), then
\[
0<\frac{t}{m}\leqslant\gamma<\overline{C}(13),
\qquad
\frac{3}{\tau_1(t/m)}
\leqslant
\frac{3}{\tau_1(\gamma)}
<
\frac{3}{\tau_1(\tfrac67)}+\eta.
\]
Thus \textbf{Case 2} yields
the estimate~\eqref{eq:smt for curves of d geqslant 12}
with constant \({3}/{\tau_1(\tfrac67)}+\eta\).

If \(m\geqslant6\) and
\(
t\geqslant\frac{1}{10}m+7,
\)
then \textbf{Case 1} gives the same estimate with constant \(10\), since
\[
\frac{m}{t-7}\leqslant10,
\qquad
\frac{1}{d-3}=\frac{1}{10}<10.
\]

It remains to consider \(t>\gamma m\).
We distinguish the following possibilities:
\begin{itemize}
    \item If \(3\leqslant m\leqslant8\), the integer \(t\) is at least
    \[
    3,4,5,6,7,7,
    \]
    respectively. All these possibilities are excluded by
    Key Vanishing Lemma~\ref{lem-1.2} and twist monotonicity.

    \item If \(m=9\), then
    \(
    t>\gamma m>\frac67\cdot9=\frac{54}{7}.
    \)
    Since \(t\) is an integer, we have
    \[
    t\geqslant8>\frac{9}{10}+7,
    \]
    so \textbf{Case 1} applies.

    \item If \(m\geqslant10\), then
    \[
    t>\gamma m>\frac67m>\frac{1}{10}m+7.
    \]
    Hence \textbf{Case 1} again applies.
\end{itemize}

Finally, choose \(\varepsilon'>0\) in~\eqref{C_d from omega_1}
sufficiently small that the constant supplied there is less than \(10\).
This choice can be fixed independently of \(\eta\).
Consequently, taking the limit as $\eta$ tends to $0$, we conclude that all cases are covered with
\[
C_{13}=\frac{63}{393-2\sqrt{38607}}\approx2357.920.
\]

\medskip\noindent
\textbf{Case \(d=12\).}
A direct computation yields
\[
\frac{1}{(d-3)\delta_1}
=
\frac{1}{4-\sqrt{15}}
\approx 7.873
\qquad(d=12).
\]
Since
\(
{7}/{15}<\overline{C}(12)=1/2,
\)
and, by~\eqref{what is tau}, the function \(\tau_1\) is positive,
strictly decreasing, and continuous on
\((0,\overline{C}(12))\), for every \(\eta>0\) we may choose
\[
\gamma=\frac{7}{15}+\varepsilon,
\qquad 0<\varepsilon<\frac{1}{30},
\]
with \(\varepsilon\) sufficiently small that
\[
17
<\frac{3}{\tau_1(\tfrac{7}{15})}
<
\frac{3}{\tau_1(\gamma)}
<
\frac{3}{\tau_1(\tfrac{7}{15})}+\eta.
\]

If \(t\leqslant\gamma m\), then
\[
0<\frac{t}{m}\leqslant\gamma<\overline{C}(12),
\qquad
\frac{3}{\tau_1(t/m)}
\leqslant
\frac{3}{\tau_1(\gamma)}
<
\frac{3}{\tau_1(\tfrac{7}{15})}+\eta.
\]
Thus \textbf{Case 2} yields
the estimate~\eqref{eq:smt for curves of d geqslant 12}
with constant \({3}/{\tau_1(\tfrac{7}{15})}+\eta\).

If \(m\geqslant6\) and
\(
t\geqslant\frac{1}{17}m+7,
\)
then \textbf{Case 1} gives the same estimate with constant \(17\), since
\[
\frac{m}{t-7}\leqslant17,
\qquad
\frac{1}{d-3}=\frac19<17.
\]

It remains to consider \(t>\gamma m\).
We distinguish the following possibilities:
\begin{itemize}
    \item If \(3\leqslant m\leqslant14\), the integer \(t\) is at least
    \[
    2,2,3,3,4,4,5,5,6,6,7,7,
    \]
    respectively. All these possibilities are excluded by
    Key Vanishing Lemma~\ref{lem-1.2} and twist monotonicity.

    \item If \(15\leqslant m\leqslant17\), then
    \(
    t>\gamma m>\frac{7}{15}m\geqslant7.
    \)
    Since \(t\) is an integer, we have
    \[
    t\geqslant8\geqslant\frac{1}{17}m+7,
    \]
    so \textbf{Case 1} applies.

    \item If \(m\geqslant18\), then
    \[
    t>\gamma m>\frac{7}{15}m
    >
    \frac{1}{17}m+7.
    \]
    Hence \textbf{Case 1} again applies.
\end{itemize}

Finally, choose \(\varepsilon'>0\) in~\eqref{C_d from omega_1}
sufficiently small that the constant supplied there is less than \(17\).
This choice can be fixed independently of \(\eta\).
Consequently, taking the limit as $\eta$ tends to $0$, we obtain
\[
C_{12}
=\frac{3}{\tau_1(\tfrac{7}{15})}
=\frac{30}{173-\sqrt{29809}}
\approx 86.413.
\]
This completes the $d=12$ case. We have thus obtained the stated estimate
for algebraically nondegenerate entire curves over the nonempty
Zariski-open parameter locus $U_d$.

Now choose a very generic parameter $a\in U_d$ for which
Lemma~\ref{lem:SMT-algebraically-degenerate} holds, and let
$f\colon\mathbb{C}\to\mathbb{P}^2$ be any nonconstant entire curve.
If $f$ is algebraically nondegenerate, the estimate has already been
proved. If $f$ is algebraically degenerate, then
$f(\mathbb{C})\not\subset\mathcal{C}_a$, because the smooth curve
$\mathcal{C}_a$ has genus $(d-1)(d-2)/2\geqslant2$ and admits no
nonconstant entire curve. Lemma~\ref{lem:SMT-algebraically-degenerate}
therefore gives
\[
T_f(r)\leqslant
\frac{1}{d-4}N_f^{[1]}(r,\mathcal{C}_a)
+o\bigl(T_f(r)\bigr)\quad\parallel.
\]
We compare this coefficient with $C_d$. For $d=12,13$, the displayed
formulas give $C_d>1>1/(d-4)$. For $d\geqslant14$, one has
\[
32d-117-\sqrt{640d^2-1248d+1017}
<32d-117<72(d-4).
\]
Thus $1/(d-4)<C_d$ for every $d\geqslant12$, so the same estimate with
the stated constant $C_d$ also holds in the algebraically degenerate
case. This proves the additional assertion for very generic parameters
and completes the proof of
Theorem~\ref{SMT for curves of d geqslant 12}.
\qed

\section{\bf Further Perspectives}
\label{section:6}

\subsection{A Computational Challenge}

To reach the degree bounds $15$ in the compact case and $11$ in the logarithmic case — the current thresholds down to which the existence of nonzero negatively twisted invariant $2$-jet differentials is guaranteed by Riemann–Roch computations — it would suffice to establish the following vanishing statements.

\begin{ques}\rm\label{ques-remaining-cases}
Are the following vanishing statements true?

\begin{enumerate}

\smallskip
    \item Let \(\mathcal{C} \subset \mathbb{P}^2\) be a generic curve of degree \(11\). Then
      \[
      H^0\!\bigl(\mathbb{P}^2,\;E_{2,m}T_{\mathbb{P}^2}^{*}(\log \mathcal{C}) \otimes \mathcal{O}(-t)\bigr) = 0
      \]
      for all pairs \((m,t) \in \mathbb{N}^2\) such that
      \[
      t = \Big\lceil \frac{13 m}{96} \Big\rceil , \qquad 3 \leqslant m \leqslant 51.
      \]

\item Let \(X \subset \mathbb{P}^3\) be a generic surface of degree \(d\). Then:
    
    \begin{itemize}
    \smallskip
        \item For \(d = 16\), \(H^{0}\!\bigl(X,\;E_{2,m}T_{X}^{*} \otimes \mathcal{O}_X (-t)\bigr) = 0\) for all \((m,t)\) in
        \[
        \{(3,2),\,(4,3),\,(5,4),\,(6,4),\,(7,5),\,(8,5),\,(9,6),\,(10,7),\,(11,7)\}.
        \]
        
        \smallskip
        \item For \(d = 15\),
        \(H^0\!\bigl(X,\;E_{2,m}T_X^{*} \otimes \mathcal{O}_X (-t)\bigr) = 0\) for all pairs \((m,t) \in \mathbb{N}^2\) such that
        \[
        t = \Big\lceil \frac{17 m}{66} \Big\rceil , \qquad 3 \leqslant m \leqslant 27.
        \]
    \end{itemize}
 \end{enumerate}
\end{ques}

\smallskip

\begin{rem}\rm
\label{smt for degree 11}
If the vanishing statements listed in
Question~\ref{ques-remaining-cases}--(1) for generic curves
$\mathcal{C}\subset\mathbb{P}^2$ of degree $11$ were to hold, then the
argument presented in Subsection~\ref{subsect:SMT} would yield, over a
nonempty Zariski-open parameter locus $U_{11}$, the following Second Main
Theorem type estimate for every algebraically nondegenerate entire curve
$f\colon\mathbb{C}\to\mathbb{P}^2$:
\[
T_f(r) \leqslant C_{11} \, N_f^{[1]}(r, \mathcal{C}) + o\bigl(T_f(r)\bigr) \quad\parallel,
\]
After the proposed finite vanishing range is removed, the largest surviving
ratio below $13/96$ is $5/37$.  Hence the argument gives every constant
\[
C_{11}>\frac{3}{\tau_1(5/37)}
       =\frac{3}{2}\left(3507+\sqrt{12298605}\right)
       \approx 10520.905.
\]
Under the same vanishing assumption, for very generic parameters in
$U_{11}$ the estimate would hold for every nonconstant entire curve.
Indeed, the smooth degree-$11$ boundary admits no nonconstant entire
curve, and Lemma~\ref{lem:SMT-algebraically-degenerate} handles the
algebraically degenerate case with coefficient $1/(11-4)=1/7<C_{11}$.
\end{rem}

\begin{rem}\label{rmk:computational-limits}\rm
  The hyperelliptic‑type ansatz \(y^{d}=P(x)\) is insufficient to prove the conjectured Key Vanishing Lemma for degree~\(11\) (see Question~\ref{ques-remaining-cases}--(1)): in all our Maple experiments, nontrivial sections invariably occur for \((m,t)=(3,1)\).
  To settle degree~\(11\) one must use more involved defining equations, which dramatically raise the computational cost — the resulting linear systems would involve far more than twenty million variables, well beyond Maple’s reach.

  For surfaces in \(\mathbb{P}^{3}\), the situation is mixed.
  With the one‑term Fermat perturbation, degree~\(16\) likely works (no obstruction at \((3,1)\)), but the linear system already exceeds two million variables.
  A C\texttt{++} implementation with parallel computing could plausibly handle it; however, the main purpose of this paper is to introduce a usable algebraic reduction framework and test its viability on concrete non‑trivial examples (degree~\(17\) compact, degrees~\(12,13\) logarithmic), so we do not pursue degree~\(16\) here.

  Degree~\(15\) is the hardest case: the same ansatz \emph{fails}, since our Maple experiments already exhibit nontrivial sections for \((m,t)=(3,1)\) on
  \(X^{15}+Y^{15}+Z^{15}+T^{15}+X^{7}T^{8}=0\).
  Resolving degree~\(15\) would demand more sophisticated equations and lead to linear systems with well over sixty million variables — an immense computational challenge that will require new algorithmic ideas and a substantial engineering effort.
\end{rem}

\begin{ques}\rm
\label{ques:geometric-interpretation}
Is there a natural algebro‑geometric construction that explains some of the nonvanishing sections for \((m,t)=(3,1)\) reported in Remark~\ref{rmk:computational-limits}?  
\end{ques}

\subsection{Relaxing ``Very Generic'' to ``Generic''}\label{Demailly-ElGoul-question}

The problem of replacing ``very generic'' by ``generic'' in the compact
case was raised by Demailly--El Goul~\cite[Remark~7.2]{Demailly-Elgoul2000}.

\begin{rem}\label{rem 7.5}\rm
Xie--Zhao~\cite[Theorems~A and~C]{xie2026kobayashihyperbolicitygeneralsurfaces}
establish this strengthening for surfaces of degree at least $18$.
Their family criterion
\cite[Theorem~B and Corollary~4.2.2]{xie2026kobayashihyperbolicitygeneralsurfaces},
combined with Theorems~\ref{thm:GGL-bound} and~\ref{thm:Clemens-Xu},
implies that a generic smooth surface $X\subset\mathbb{P}^3$
of degree $d\geqslant17$ contains no rational or elliptic curves.
Together with Theorem~\ref{thm:GGL-generic}, this strengthens
Theorem~\ref{thm:compact-hyperbolic} to generic surfaces
of degree $d\geqslant17$.
\end{rem}

\begin{rem}\label{rem 7.4}\rm
The additional input in
\cite{xie2026kobayashihyperbolicitygeneralsurfaces}
is a Poincar\'e-type degree bound for invariant curves of fixed
geometric genus on auxiliary surfaces of general type.
Applied in families, this bound and constructibility turn the
very generic exclusion into a Zariski-open condition.
The geometry of the foliations is essential: universal degree bounds
fail for the unrestricted Poincar\'e problem~\cite{Poincare1891,LinsNeto2002}.
\end{rem}

In the logarithmic case,
Xie--Zhao~\cite[Theorem~F]{xie2026kobayashihyperbolicitygeneralsurfaces}
prove hyperbolicity and hyperbolic embedding for complements of generic
pairs of smooth transverse plane curves in certain degree ranges,
including two cubics.
Their results concern a boundary with two components, whereas
Theorem~\ref{thm:log-hyperbolic} concerns a single smooth plane curve.
The techniques developed in \cite{xie2026kobayashihyperbolicitygeneralsurfaces}
are expected to yield an analogous strengthening of
Theorem~\ref{thm:log-hyperbolic}, with ``very generic'' replaced by
``generic'' for every $d\geqslant12$.

\section*{Declaration of AI Use}

Lemma~\ref{lem:SMT-algebraically-degenerate} was independently
formulated and proved by ChatGPT (GPT-6 Astra). The authors
have carefully examined the statement and proof and verified
their mathematical correctness.

AI tools were also used to assist with writing and to improve
the language and clarity of the paper. The authors have
reviewed and approved all AI-assisted content and take full
responsibility for the mathematical results, arguments, and
final text of this paper.

\bigskip

\noindent{\bf Acknowledgements} \ 
S.-Y. Xie would like to thank Daniel Barlet for inspiring discussions on cycle spaces and Yum-Tong Siu for suggesting the use of the direct image theorem.  
S.-Y. Xie and L. Hou thank the Tianyuan Mathematics Research Center for its support and excellent research environment. 
We thank Junjiro Noguchi for pointing us to the reference~\cite{Masuda_Noguchi_1996}.   
We thank Yi C. Huang and Xuanyu Pan for helpful suggestions concerning the exposition.

We are grateful to Jianjun Liu and Tao Cui for designing a C\texttt{++} implementation with parallel computing, which independently verified the key vanishing lemmas established in this paper. Their work provides a double assurance of the computational results presented here.

The results of this article were first presented by S.-Y. Xie at \textit{The Seminar on Several Complex Variables and Complex Geometry} (July 27--August 2, 2025), organized by Xiangyu Zhou, Yum-Tong Siu, and John Erik Forn\ae ss.

\bigskip

\bigskip

\noindent{\bf Funding} \  
S.-Y. Xie acknowledges partial support from the National Key R\&D Program of China under Grants No. 2023YFA1010500 and No. 2021YFA1003100, and from the National Natural Science Foundation of China under Grants No. 12288201 and No. 12471081, as well as support from the  Xiaomi Young Talents Program. D.T. Huynh is supported by the Vietnam National Foundation for Science and Technology Development
(NAFOSTED) under the grant number 101.04-2025.19.

\medskip

\def\cprime{$'$}

\end{document}